\documentclass[a4paper,twoside,11pt]{amsart}

\numberwithin{equation}{section}

\usepackage{amsmath, amssymb, amscd, amsthm, mathtools, amsfonts, tikz-cd, bbm, enumitem}
\usepackage[mathscr]{euscript}
\usepackage{hyperref}
\usepackage[utf8]{inputenc}
\usepackage[english]{babel}
\usepackage{stmaryrd}
\usepackage{xcolor}
\usepackage{marginnote}

\let\mathscrtmp\mathscr
\usepackage{mathrsfs}
\let\mathrsfs\mathscr
\let\mathscr\mathscrtmp

\newcommand{\cD}{\mathrsfs{D}}
\newcommand{\cN}{\mathrsfs{N}}

\let\itemref\ref

\DeclareMathOperator{\Ind}{Ind}
\newcommand{\ttBan}{\mathsf{Ban}}
\newcommand{\ttBanleq}{\mathsf{Ban}^{\leq 1}}
\newcommand{\ttC}{\mathsf{C}}
\newcommand{\ttP}{\mathsf{P}}

\newcommand{\T}{\mathbb{T}}
\newcommand{\M}{\mathbb{M}}
\newcommand{\N}{\mathbb{N}}
\newcommand{\K}{\mathbb{K}}
\newcommand{\Q}{\mathbb{Q}}
\newcommand{\Z}{\mathbb{Z}}

\newcommand{\C}{\mathbb{C}}

\newcommand{\A}{\mathbb{A}}
\newcommand{\G}{\mathbb{G}}

\newcommand{\V}{\mathbb{V}}

\newcommand{\OO}{\mathcal{O}}

\newcommand{\BB}{\mathcal{B}}

\newcommand{\NN}{\mathcal{N}}

\newcommand{\LL}{\mathcal{L}}
\newcommand{\MM}{\mathcal{M}}

\newcommand{\RR}{\mathcal{R}}

\newcommand{\sN}{\mathrm{N}}
\newcommand{\sD}{\mathrm{D}}
\newcommand{\sL}{\mathrm{L}}
\newcommand{\sLfil}{\sL^{\mathrm{fil}}}

\newcommand{\BBBan}{\mathscr{B}\mathrm{an}}
\newcommand{\BBB}{\mathscr{B}}
\newcommand{\CCC}{\mathscr{C}}
\newcommand{\DDD}{\mathscr{D}}

\newcommand{\EEE}{\mathscr{E}}
\newcommand{\AAA}{\mathscr{A}}

\newcommand{\TTT}{\mathscr{T}}
\newcommand{\PPP}{\mathscr{P}}

\newcommand{\wotimes}{\widehat{\otimes}}

\newcommand{\bDelta}{\mathbf{\Delta}}

\mathchardef\hyph="2D

\tikzset{%
	symbol/.style={%
		draw=none,
		every to/.append style={%
			edge node={node [sloped, allow upside down, auto=false]{$#1$}}}
	}
}

\newcommand{\na}{\mathrm{nA}}

\DeclareMathOperator{\Map}{Map}

\DeclareMathOperator{\id}{id}

\DeclareMathOperator{\Hom}{Hom}

\DeclareMathOperator*{\colim}{colim}

\let\lim\relax 
\DeclareMathOperator*{\lim}{lim}

\DeclareMathOperator{\Mod}{Mod}

\newcommand{\op}{\mathrm{op}}

\DeclareMathOperator{\Spec}{Spec}

\DeclareMathOperator{\Proj}{Proj}
\DeclareMathOperator{\Sym}{Sym}

\newcommand{\cl}{\mathrm{cl}}
\DeclareMathOperator{\LSym}{LSym}
\DeclareMathOperator{\im}{im}
\DeclareMathOperator{\Arr}{Arr}

\newcommand{\ext}{{\operatorname{ext}}}

\DeclareMathOperator{\Res}{Res}

\DeclareMathOperator{\triv}{triv}

\DeclareMathOperator{\Der}{Der}
\DeclareMathOperator{\fgf}{\mathrm{fgf}}

\newcommand{\ev}{\mathrm{ev}}

\newcommand{\Poly}{\mathscr{P}\mathrm{oly}}

\newcommand{\sSet}{\mathrm{s}\mathscr{S}\mathrm{et}}

\newcommand{\Aff}{\mathscr{A}\mathrm{ff}}
\newcommand{\Alg}{\mathscr{A}\mathrm{lg}}

\renewcommand{\Mod}{\mathscr{M}\mathrm{od}}

\newcommand{\St}{\mathscr{S}\mathrm{t}}

\newcommand{\Cat}{\mathscr{C}\mathrm{at}}
\newcommand{\Set}{\mathscr{S}\mathrm{et}}

\newcommand{\Fun}{\mathrm{Fun}}
\newcommand{\Space}{\mathscr{S}}

\newcommand{\QCoh}{\mathrm{Q}\mathscr{C}\mathrm{oh}}

\DeclareMathOperator{\Bl}{Bl}

\newcommand{\Einfty}{\mathbf{E}_\infty}
\newcommand{\Sp}{\mathscr{S}\mathrm{p}}

\newcommand{\PrL}{\mathscr{P}\mathrm{r}^{\mathrm{L}}}

\newcommand{\DAlg}{\mathrm{D}\mathscr{A}\mathrm{lg}}
\newcommand{\CAlg}{\mathrm{C}\mathscr{A}\mathrm{lg}}

\newcommand{\nc}{\mathrm{nc}}
\newcommand{\StMod}{\mathscr{S}\mathrm{t}\mathscr{M}\mathrm{od}}
\newcommand{\QAlg}{\mathrm{Q}\mathscr{A}\mathrm{lg}}

\mathchardef\hyph="2D

\makeatletter
\newtheorem*{rep@theorem}{\rep@title}
\newcommand{\newreptheorem}[2]{%
	\newenvironment{rep#1}[1]{%
		\def\rep@title{#2 \ref{##1}}%
		\begin{rep@theorem}}%
		{\end{rep@theorem}}}
\makeatother

\theoremstyle{definition}
\newtheorem{Def}{Definition}[section]
\newtheorem{Not}[Def]{Notation}
\newtheorem{Rem}[Def]{Remark}
\newtheorem{Rem*}[]{Remark}
\newtheorem{Exm}[Def]{Example}

\newtheorem{Ass}[Def]{Assumption}
\newtheorem{Warn}[Def]{Warning}

\theoremstyle{plain}
\newtheorem{Prop}[Def]{Proposition}
\newtheorem{Thm}[Def]{Theorem}
\newreptheorem{theorem}{Theorem}
\newtheorem{Lem}[Def]{Lemma}
\newtheorem{Cor}[Def]{Corollary}

\usepackage{abstract}
\addto\captionsenglish{}

\title[Blow-ups and normal bundles in  derived geometries]{Blow-ups and normal bundles in connective and nonconnective derived geometries}

\author{Oren Ben-Bassat}
\address{University of Haifa\\Department of 
	Mathematics\\Haifa 3498838 \\Israel}
\email{ben-bassat@math.haifa.ac.il}

\author{Jeroen Hekking} 
\address{University of Regensburg\\Department of Mathematics\\Regensburg 93053\\Germany}
\email{jeroen.hekking@ur.de}

\begin{document}	
		
\maketitle

\begin{abstract}
This work presents a generalization of derived blow-ups and of the derived deformation to the normal bundle from derived algebraic geometry to any geometric context. The latter is our proposed globalization of a derived algebraic context, itself a generalization of the theory of simplicial commutative rings and due to Bhatt--Mathew and Raksit.

One key difference between a geometric context and ordinary derived algebraic geometry is that the coordinate ring of an affine object in the former is not necessarily connective. When constructing  generalized blow-ups, this not only turns out to be remarkably convenient, but also leads to a wider existence result. Indeed, we show that the derived Rees algebra and the derived blow-up exist for any affine morphism of stacks in a given geometric context. However, in general the derived Rees algebra will no longer be connective, hence in general the derived blow-up will not live in the connective part of the theory. Unsurprisingly, this can be solved by restricting the input to closed immersions. The proof of the latter statement uses a derived deformation to the normal bundle in any given geometric context, which is also of independent interest.

Besides the geometric context which extends algebraic geometry, the second main example of a geometric context is an extension of analytic geometry as based on categories of Ind-Banach spaces or modules. The latter is a recent construction, and includes many different flavors of analytic geometry, such as complex analytic geometry, non-Archimedean rigid analytic geometry and analytic geometry over the integers. The present work thus provides derived blow-ups and a derived deformation to the normal bundle in all of these, which is expected to have many applications.
\end{abstract}

\setcounter{tocdepth}{1}
\tableofcontents
\setcounter{tocdepth}{2}

\section*{Introduction} 
The theory of derived blow-ups has recently been developed in \cite{KhanVirtual, HekkingGraded, Weil}. It gives an analogue in derived algebraic geometry to blow-ups in classical algebraic geometry, and has found applications in (virtual) intersection theory, in algebraic $K$-theory, and in stabilizer reduction of derived Artin stacks.  

 The key constructions in \cite{Weil} are to a large extent formal, since they ultimately depend on the theory of derived Weil restrictions (as first studied in \cite{LurieSpectral} in the context of spectral algebraic geometry), which is essentially an $\infty$-categorical idea. Since the applications of derived blow-ups have been quite satisfying so far, this leads to the natural question whether the same techniques can be used to define derived blow-ups in contexts other than algebraic geometry. 

The answer presented in this paper takes as starting point a derived algebraic context, the theory of which is due to Bhatt--Mathew (since their work has not yet appeared, we will use the exposition in \cite{RaksitHKR}). Such a derived algebraic context $\CCC$ gives rise to derived rings in $\CCC$, which are generalizations of  simplicial commutative rings but which need not be connective. We then use spectra of derived rings in a given derived algebraic context as affine building blocks to define geometric contexts. 

We will first present our main results in this introduction, after which we explain the central definitions in more detail. A recent history of derived blow-ups and its applications in derived algebraic geometry is given in \S \ref{Sec:The_development}.

\subsection*{Main results}
Fix a derived algebraic context $\CCC$ and a geometric context relative to $\CCC$ (Definition \ref{Def:GeomCont}). To this, we associate the $\infty$-category ${\St_{\CCC}}$ of so-called \textit{ $\CCC$-stacks}. There is a natural contravariant inclusion from the $\infty$-category of derived rings into ${\St_{\CCC}}$, written $\Spec^\nc(-)$. Let ${\Aff_{\CCC}}$ be the essential image of $\Spec^\nc(-)$, and let ${\Aff_{\CCC_{\geq 0}}}$ be the full subcategory spanned by objects of the form $\Spec^\nc R$ for $R$ connective. The datum of the geometric context includes a topology on ${\Aff_{\CCC_{\geq 0}}}$, and a \emph{$\CCC_{\geq 0}$-stack} is by definition a sheaf on ${\Aff_{\CCC_{\geq 0}}}$ (in the $\infty$-categorical sense) for this topology. We first show:
\begin{reptheorem}{Thm:Stacks}
	The category ${\St_{\CCC_{\geq 0}}}$ of $\CCC_{\geq 0}$-stacks is the smallest full subcategory of ${\St_{\CCC}}$ which contains ${\Aff_{\CCC_{\geq 0}}}$ and is closed under colimits.
\end{reptheorem}
The primary example to keep in mind is $\CCC = \Mod_\Z$, with the topology on $\Aff_{\Mod_\Z}$ induced by the \'{e}tale topology on $\Aff_{(\Mod_\Z)_{\geq 0}}$ via Example \ref{Ex:InducedContext}. Then $\St_{(\Mod_\Z)_{\geq 0}}$ is the $\infty$-category of derived stacks in the sense of ordinary derived algebraic geometry.

Roughly speaking, a morphism $h:U \to V$ of $\CCC$-stacks is \emph{nonconnectively affine} if it is locally of the form $\Spec^\nc B \to \Spec^\nc A$, and \emph{affine} if additionally we can choose $A,B$ to be connective (Definition \ref{Def:Affine}). Following \cite{Weil}, we introduce the \emph{deformation space} $\sD_{X/Y}$ of a morphism $X \to Y$ of $\CCC$-stacks as the $\CCC$-stack over $Y \times \A^1_\CCC$ such that a $T$-point of $\sD_{X/Y}$ is a morphism $T \times_{Y \times \A^1_\CCC} (Y \times \{0\}) \to X$ over $Y$, where $\A^1_\CCC$ is the affine line for the given geometric context (Definition \ref{Def:DefSpace}). The next main result leads to the definition of Rees algebras:
\begin{reptheorem}{Thm:Rees-affine}
	The deformation space $\sD_{X/Y}$ of an affine morphism $f:X \to Y$ of $\CCC_{\geq 0}$-stacks is nonconnectively affine over $Y \times \A^1_\CCC$.
\end{reptheorem}

In the situation of the above theorem, we define the \emph{extended Rees algebra} of $f$ as the quasi-coherent $\OO_Y$-algebra $\RR_{X/Y}^\ext$ such that $\sD_{X/Y}$ is the relative spectrum of $\RR_{X/Y}^\ext$ over $Y$ (Definition \ref{Def:Rees}). The construction will induce a natural $\Z$-grading on $\RR_{X/Y}^\ext$. Using this, we define the \emph{blow-up} of $Y$ in $X$ as $\Bl_XY \coloneqq \Proj \RR_{X/Y}$, where $\RR_{X/Y}$ is the $\N$-graded algebra obtained from $\RR_{X/Y}^\ext$ by discarding everything with negative homogeneous degrees, and $\Proj(-)$ is the projective spectrum construction carried out in the given geometric context (Definition \ref{Def:Blup}).

At this point the Rees algebra can fail to be connective, and thus the blow-up can fail to be a $\CCC_{\geq 0}$-stack. This is solved by restricting the input to closed immersions. Taking our cue from ordinary derived algebraic geometry, we say that a morphism $f :X \to Y$ of $\CCC_{\geq 0}$-stacks is a \emph{closed immersion} if it is affine, and locally of the form $\Spec^\nc B \to \Spec^\nc A$, where $A \to B$ is a morphism of connective derived rings with connective fiber. We then have:
\begin{reptheorem}{Thm:ReesClosed}
	Let $X \to Y$ be a closed immersion of $\CCC_{\geq 0}$-stacks. Then the extended Rees algebra $\RR_{X/Y}^\ext$ is connective, and consequently the blow-up $\Bl_XY$ lives in ${\St_{\CCC_{\geq 0}}}$.
\end{reptheorem}
Interestingly, the proof of this theorem uses the deformation to the normal bundle, which is formulated in the following result:
\begin{reptheorem}{Thm:Deformation}
	In the diagram 
	\begin{center}
		\begin{tikzcd}
			X\ar[r]\ar[d] & X\times \A^1_\CCC\ar[d] & X\times \G_{m,\CCC}\ar[l]\ar[d]\\
			\sN_{X/Y}\ar[r]\ar[d] & \sD_{X/Y}\ar[d] & Y\times \G_{m,\CCC}\ar[l]\ar[d]\\
			Y\ar[r] & Y\times \A^1_\CCC & Y\times \G_{m,\CCC},\ar[l]		
		\end{tikzcd}	
	\end{center}
	all squares are Cartesian, assuming $X \to Y$ admits a cotangent complex.
\end{reptheorem}
Here, the \emph{normal bundle} is $\sN_{X/Y} \coloneqq \Spec^\nc (\LSym_{\OO_X} \LL_{X/Y}[-1])$, where $\LL_{X/Y}$ is the cotangent complex as recalled in \S \ref{Sec:Algebraic_contexts} for the affine case and defined in \S \ref{Sec:Geometric_contexts} for the global case. The diagram is naturally acted upon by  the multiplicative group $\G_{m,\CCC}$ for the given geometric context.

In ordinary derived algebraic geometry, the deformation stack $\cD_{X/Y} \coloneqq [\sD_{X/Y}/\G_{m,\CCC}]$ can also be described as the classifying stack of virtual Cartier divisors over $X \to Y$ \cite{Weil}. This description goes through in the general setting, as shown in Proposition \ref{Prop:VCD}. This leads to our final main result, which generalizes the fact that $\Bl_ZX$ classifies excessive virtual Cartier divisors over $Z \to X$ in ordinary derived algebraic geometry \cite{Weil}.

\begin{reptheorem}{Thm:BlFOPaff}
	For $Z \to X$ a closed immersion of $\CCC_{\geq 0}$-stacks, the blow-up $\Bl_ZX$ classifies strict virtual Cartier divisors over $Z \to X$. Moreover, under appropriate finiteness assumptions, the structure map $\Bl_ZX \to X \times [\A^1_{\CCC}/\G_{m,\CCC}]$ is nonconnectively affine.
\end{reptheorem}	

This result uses Assumption \ref{Ass:Bcover}, which roughly says that covers of the form $\{ \Spec A_f \to \Spec A \mid f \in I \}$ behave as expected.

\subsection*{Synopsis}
\S \ref{Sec:Algebraic_contexts} is a review of the theory of derived algebraic contexts found in \cite{RaksitHKR} that is relevant to the present work. This story is globalized in \S \ref{Sec:Geometric_contexts}, which leads to the geometric framework used for generalized blow-ups and normal bundles in \S \ref{Sec:Deformation_spaces_normal_bundles_and_blowups}. The paper ends in \S \ref{Sec:Applications_and_examples} with an exposition of our constructions in the two main examples of a geometric context: derived algebraic geometry and derived analytic geometry. The remainder of this introduction provides more details on the key definitions used in the paper, a short overview of derived analytic geometry, and an explanation of the benefits of nonconnective geometry in the construction of derived blow-ups.

\subsection*{Derived algebraic contexts and graded modules}
A derived algebraic context, simply called an \emph{algebraic context},  is a symmetric monoidal, stable $\infty$-category $\CCC$ endowed with a t-structure, together with a symmetric monoidal subcategory $\CCC^0 \subset \CCC^\heartsuit$, satisfying certain compatibility conditions. The objects of $\CCC$---called \emph{$\CCC$-modules}---play the role of modules (in the derived sense) in the context under consideration. For example, taking $\CCC = \Mod_\Z$ leads to derived algebraic geometry, where $\Mod_\Z$ is the derived $\infty$-category of $\Z$-modules (unbounded and with homological indexing notation).

To produce the algebras for $\CCC$, one first takes the left derived functor $\LSym_{\CCC_{\geq 0}}$ of the functor $\Sym_{\CCC^\heartsuit} : \CCC^0 \to \CCC_{\geq 0}$, where $\Sym_{\CCC^\heartsuit}$ is the classical symmetric algebra functor with respect to the symmetric monoidal structure on $\CCC^\heartsuit$. A key result in the theory of derived algebraic contexts is that $\LSym_{\CCC_{\geq 0}}$ then extends to a monad $\LSym_\CCC : \CCC \to \CCC$. This produces the $\infty$-category ${\DAlg_{\CCC}}$ of algebra objects for $\LSym_\CCC$, simply called \emph{$\CCC$-algebras}. Observe that the underlying module of a given $\CCC$-algebra is typically not connective. We write ${\DAlg_{\CCC_{\geq 0}}}$ for the subcategory of algebras that do have connective underlying module.

In \S \ref{Sec:Algebraic_contexts}, we also develop the $\M$-graded version $\CCC^\M$ of an algebraic context $\CCC$, for any discrete commutative monoid $\M$. We think of the objects in $\CCC^\M$ as $\M$-graded $\CCC$-modules. Such an object is not hard to define, since it is just a set of objects $N_a \in \CCC$ of \emph{elements of degree $a$}, indexed by $a \in \M$.\footnote{There is a potential clash of notation and terminology here with underlying 1-categorical, simplicial objects, at least when $N$ is connective. Since in the present work we never consider such 1-categorical models this should not lead to any confusion.} However, as in classical algebra, the commutative monoid structure on $\M$ comes into play when defining the appropriate tensor product in $\CCC^\M$. We do this through Day convolution, since an explicit description is in general infeasible in the derived setting. The prime examples that we are interested in are $\M =\N$ and $\M = \Z$, since this leads to $\N$-graded and $\Z$-graded $\CCC$-algebras respectively, which we use to construct the blow-up from the Rees algebra.

\subsection*{Geometric contexts}
In \S \ref{Sec:Geometric_contexts} we propose a perspective on geometry relative to a fixed algebraic context $\CCC$, meaning that $\DAlg_{\CCC}^\op$ will be the category of (nonconnective) affines. The datum of a \textit{geometric context} will thus be an algebraic context $\CCC$, together with a Grothendieck topology $J$ on ${\Aff_{\CCC}} \coloneqq \DAlg_{\CCC}^\op$. Since we are ultimately interested in the connective part of the theory, we impose certain conditions on $J$ such that it restricts to a topology on ${\Aff_{\CCC_{\geq 0}}} \coloneqq \DAlg_{\CCC_{\geq 0}}^\op$, and such that the induced functor 
\[
	G \colon {\St_{\CCC_{\geq 0}}} \to {\St_{\CCC}}
\]
from stacks on ${\Aff_{\CCC_{\geq 0}}}$ to stacks on ${\Aff_{\CCC}}$ is fully faithful and preserves limits, colimits and affine objects.

After this, we introduce some basic notions for a given geometric context, such as quasi-coherent modules. We also globalize graded algebraic contexts, which leads to a comparison between quasi-coherent, $\Z$-graded algebras over a given base $X$, and affine morphisms $Y \to X$ endowed with a (fibre-wise) $\G_{m,\CCC}$-action on $Y$, where $\G_{m,\CCC} \coloneqq \Spec (\mathbbm{1}_{\CCC}[\Z])$. This enables us to define the projective spectrum of a given $\N$-graded, quasi-coherent $\OO_X$-algebra $\BB$ as 
\[
	\Proj \BB \coloneqq [((\Spec^\nc \BB) \setminus V(\BB_+)) / \G_{m,\CCC}],
\]
where $V(\BB_+) \coloneqq \Spec^\nc(\BB_0)$, with $\BB_0$  the part of $\BB$ of degree 0. 

\subsection*{Derived analytic geometry}
The first applications of generalized derived blow-ups outside derived algebraic geometry are expected in derived analytic geometry. The latter is a fairly recent field, which has (at least) two main approaches. The first is work done by Porta in \cite{PortaDerivedI, PortaDerivedII}, \cite{PortaRiemannHilbert} on derived complex geometry, going back to \cite{LurieDAGIX}, and work done by Porta--Yue Yu in \cite{PortaNonArchimedean,PortaRepresentability}, which is mostly in derived non-Archimedean analytic geometry. The second main approach is carried out in
\cite{BB,BassatNonArchimedean,BBK,BBK2,BassatAnalytification,BeKr,FDA}. We follow the latter.

In \cite{BassatNonArchimedean} it is shown how to cast non-Archimedean analytic geometry as a geometry relative to a homotopical algebraic context, as in \cite{ToenHAGII}. This, together with a study of analytification and homotopy epimorphisms in \cite{BassatAnalytification}, is the groundwork for \cite{FDA}, where the authors construct a (huge) derived algebraic context for derived analytic geometry. The construction depends on some parameters, and varying these produce derived $\infty$-categories of Ind-Banach $R$-modules, or of bornological $R$-modules,  over a fixed Banach ring $R$, which can be Archimedean or non-Archimedean. Imposing extra conditions or structure on the derived rings in such a context produces the basic affine building blocks for derived versions of rigid geometry \cite{BassatNonArchimedean}, Berkovich geometry, complex analytic geometry, dagger analytic geometry \cite{BB} and Stein geometry (complex or non-Archimedean) \cite{BBK}, amongst others. The present work thus provides derived blow-ups in all of these settings.

\subsection*{Benefits of the nonconnective setting}
Besides derived analytic geometry, the theory of generalized derived blow-ups is also of value to derived algebraic geometry itself. Indeed, the constructions and arguments are significantly simpler than the ones used in \cite{HekkingGraded}. 
The key difference with the previous approach is that we are now able to pass through nonconnective algebras, even though connective algebras are our main interest. Since simplicial commutative $k$-algebras are generally different from connective $\Einfty$-$k$-algebras, the latter are not suitable for this. In contrast, connective derived $k$-algebras do coincide with  simplicial $k$-algebras.

 In ordinary derived algebraic geometry the derived  Rees algebra only makes sense for closed immersions of derived schemes. For example, it can be shown that the deformation space $\sD_{\A^1/\{0\}}$ is equivalent to $B\G_a$, which is not connectively affine. 
 
 At the time of this writing, it was not realized that the $\infty$-category of $A$-algebras $B$ with $A \to B$ surjective on $\pi_0$ is compactly generated, which essentially follows from the theory of Smith ideals \cite{MaoRevisiting}. Without this, it is generally hard to define a functor out of this $\infty$-category.
 
 On the other hand, the $\infty$-category of $A$-algebras $B$ is clearly generated under sifted colimits by finitely generated, polynomial $A$-algebras, which means that it is much easier to define a Rees algebra functor on this category. In order to do this, we need to relax the notion of affine objects, as the example $\sD_{\A^1/\{0\}} \simeq B\G_a$ shows. This example also suggests how to relax this notion: since $\G_a \simeq * \times_{B\G_a} *$, we expect $\OO_{\G_a} \simeq k \otimes_{\OO_{B\G_a}} k$, and thus $\OO_{B\G_a}$ to be a symmetric algebra on $k[-1]$. This leads to our solution of taking nonconnective derived rings as the target category of the Rees algebra functor.

 A by-product of this approach is a derived  Rees algebra for any $A$-algebra $B$, and thus a derived blow-up of any derived stack $X$ in a relatively affine center $Z \to X$. On first encounter, this seems like a somewhat bizarre idea. Of course, since the derived  Rees algebra of an $A$-algebra $B$ will in general be nonconnective unless $A \to B$ is surjective on $\pi_0$, the derived blow-up of $X$ in $Z$ will in general not be an ordinary derived algebro-geometric object, unless $Z \to X$ is a closed immersion. It is expected that the negative connectivity of the derived Rees algebra is related to the stackiness of the corresponding derived blow-up.   It will be interesting to explore this direction further, and perhaps derived blow-ups in centers other than closed immersions are more natural in a context other than derived algebraic geometry.

\subsection*{Future work}
A natural question to ask is whether Theorem~\ref{Thm:Rees-affine} can be further generalized to allow also nonconnectively affine morphisms $f \colon X \to Y$ as input. In the affine case, the question becomes whether one can construct the extended Rees algebra of any morphism $A \to B$ of (possibly nonconnective) derived rings. This question has been answered positively in upcoming work by Gardner and the second named author in \cite{Ideals}. The latter uses the theory of Smith ideals in derived geometry, thereby completing the picture familiar from ordinary algebraic geometry by rephrasing extended Rees algebras in terms of adic filtrations. 

Theorem~\ref{Thm:Rees-affine} (and its generalization) is used in \cite{HekHKR} to  give a global version of the HKR filtration on Hochschild homology. Geometrically, this entails the construction of a filtered loop stack $\sLfil(X)$ over $[\A^1_\CCC/\G_{m,\CCC}]$ which is generically the loop stack $\sL(X)$ and has special fiber the odd tangent bundle $\T X[-1] = \V(\LL_X[1])$. The geometric statement can already be deduced directly from Theorem~\ref{Thm:Deformation} applied to the second diagonal $X \to \sL(X)$. Theorem~\ref{Thm:Rees-affine} is then used to deduce the necessary base change properties in the deformation to the normal bundle diagram, to conclude the existence of a filtration on the global sections of the loop stack, with associated graded the underlying module of the derived de Rham complex. Further, a reformulation of the geometric picture in terms of a global filtered circle is given, with the aim to give additional algebraic structure on the filtered loop stack that recovers the circle action on the generic fiber and induces a mixed graded structure on the derived de Rham complex---thus globalizing \cite[Thm.~6.2.6]{RaksitHKR}.

Recently a new approach to derived analytic geometry based on condensed mathematics has been proposed by Clausen and Scholze. This does not fit our framework where geometric objects are locally modeled on derived rings in derived algebraic contexts, essentially because the compact projective generators of the category of condensed abelian groups do not form a symmetric monoidal subcategory. In the forthcoming work \cite{LocCont}, joint with Kelly, we propose a generalization of a derived algebraic context to a so-called \emph{localized context}, and develop a theory of derived analytic rings over such a localized context. This allow us to view the liquid, solid, or gaseous types of analytic geometry of Clausen--Scholze, as well as Bornological or Banach types, as different examples of the same thing. We will look at various simple Grothendieck topologies in these contexts and prove their basic properties. This leads to geometric objects which are locally the spectra of derived analytic rings over such a localized context, which allows us to define formal smoothness and \'{e}taleness via the cotangent complex. We expect to give an HKR-type result also in this more general setting, based on the deformation to the normal bundle. Ultimately, there should be a well-behaved notion of algebraicity that leads to  a more refined version of Theorem \ref{Thm:Rees-affine}, similar to \cite{Weil}. We expect that there will be a strong link between nonconnectivity on the algebra side and algebraicity on the geometry side. In the future, we also intend to examine notions of properness as well as \'{e}taleness beyond formal \'{e}taleness.

\subsection*{Acknowledgments}
The authors thank David Rydh for detailed and valuable feedback on previous versions of this paper. It should also be stressed that the work at hand is a generalization of \cite{Weil}, where most of the key ideas in this paper have their origin. The paper is a result of two research visits of Hekking to visit Ben-Bassat in Haifa, partially supported by the Signeul foundation and by the CMSC of the BSF (2021717).

The first named author would like to acknowledge the helpful conversations and support of Tony Pantev and his NSF grant (includes the NSF-BSF program joint with the first named author) corresponding to NSF award DMS-2200914 with project Title: NSF-BSF: Derived and quantum corrected structures on arithmetic and geometric moduli.

The second named author was supported by the G\"oran Gustafsson Foundation, 
 the G\"oran Gustafsson Foundation for Research in Natural Sciences and Medicine, 
by the Knut and Alice Wallenberg Foundation (project number 2021.0287), and by the DFG via the SFB 1085: Higher Invariants (project number 224262486).

 \subsection*{Notation \& conventions} 
 This work is written in the language of $\infty$-categories, with \cite{LurieHTT, LurieHA} as main sources. Since we will rarely consider 1-categories, we agree that all categorical language is implicitly $\infty$-categorical from here on, unless otherwise stated. Likewise, all algebro-geometric language is implicitly derived, meaning that a \emph{scheme} is a derived scheme etc.
 
 Familiarity with derived algebraic geometry is strictly speaking only necessary to appreciate some of the examples. However, all constructions will be generalizations from derived algebraic geometry, which means that some background knowledge could be helpful. We will mostly leave this out, and instead refer the reader to the standard literature on the subject, such as \cite{ToenHAGI, ToenHAGII, GaitsgoryStudy, LurieSpectral}. Summaries that will certainly suffice can be found in \cite{Weil} and in \cite{DRS}. Likewise, we refer the reader to \cite{HekkingGraded, Weil} for background on blow-ups and the deformation to the normal bundle in derived algebraic geometry.

 We will use Grothendieck universes $\mathfrak{U}_0 \in \mathfrak{U}_1 \in \mathfrak{U}_2 \in \cdots$ in our set-theoretic background. This allows us to  ignore set-theoretic issues in the main text, with the understanding that appropriate smallness assumptions are made implicitly. For example, when talking about colimits in a $\mathfrak{U}_{i+1}$-small category, we will mean $\mathfrak{U}_{i}$-small colimits. We provide a few separate remarks throughout the text on how to deal with this, and refer to the appendix for more details.
 
 We will use the following terminology.
 \begin{itemize}
 	\item $\PrL$ is the category of presentable categories, with colimit-preserving functors between them. Recall that all categories are assumed $\infty$-categories unless otherwise stated.
 	\item For a category $\CCC$ and a morphism $f:A \to B$ in $\CCC$, we write $\CCC_{A/B}$ for the \emph{double slice category} $(\CCC_{A/})_{f}$, so that an object in $\CCC_{A/B}$ is a factorization $A \to X \to B$ of $f$.
    \item For objects $x,y$ in a category $\CCC$ we write $\CCC(x,y)$ for the mapping space of morphisms $x \to y$ in $\CCC$.
 	\item The limit of a cosimplicial diagram is called its \textit{totalization}.
 	\item The colimit of a simplicial diagram is called its \emph{geometric realization}.
 	\item $\Space$ is the category of spaces, and $\Sp$ is the category of spectra.
 	\item $\Fun(\CCC,\DDD)$ is the category of functors $\CCC \to \DDD$, and $\PPP(\EEE) \coloneqq \Fun(\EEE^\op,\Space)$.
 \end{itemize}

Formally, a \emph{topology} will be a Grothendieck topology. However, when we need to be explicit, we will do so in terms of a Grothendieck pretopology, and silently pass to the topology it generates.

\section{The development and applications of derived blow-ups}
\label{Sec:The_development}
\subsection*{History of derived blow-ups}
The first appearance of derived blow-ups is in the proof of Weibel's conjecture in \cite{Kerz}, which says that a Noetherian scheme of dimension $d$ has no non-trivial $K$-groups below degree $-d$. The authors define the derived blow-up of a classical scheme $X$ in a center $Z = V(f_1,\dots,f_n)$ as the derived pullback of the classical blow-up of $\A^n$ in $\{0\}$ along the map $X \to \A^n$ induced by $f_1,\dots,f_n$.  With further applications in $K$-theory and in virtual intersection theory in mind (see, e.g., \cite{KhanVirtualofStacks, KhanBlowCDH, AnnalaBivariant}), this construction is generalized to blow-ups of derived schemes in quasi-smooth centers in \cite{KhanVirtual}. This suggests a generalization to allow arbitrary derived centers, which is carried out by the second named author in \cite{HekkingGraded}, and further developed and generalized to derived algebraic stacks by Khan, Rydh and the second named author in \cite{Weil}. 

\subsection*{Derived reduction of stabilizers}
One of the major goals in allowing derived blow-ups in arbitrary centers is a derived version of the reduction of stabilizers of classical Artin stacks. Two independent paths have led to this. First, a reduction of stabilizers algorithm for classical Artin stacks with good moduli spaces is given in \cite{EdidinCanon}. Since (saturated) blow-ups are a key ingredient, the main application mentioned in \cite{Weil}, as formulated by Rydh, is a derived version of this algorithm. 

Second, Kiem--Li--Savvas developed an independent approach to the reduction of stabilizers for classical DM-stacks in \cite{KiemGeneralized}, generalized to classical Artin stack with good moduli spaces in \cite{SavvasGeneralizedDT} by Savvas, with applications in generalized Donaldson--Thomas theory and virtual intersection theory in mind. The authors use a construction they call the intrinsic blow-up, which they expected to be the classical shadow of a derived construction. 
 
The goal of a derived stabilizer reduction is partially met by Rydh, Savvas and the second named author in \cite{DRS}, where a derived reduction of stabilizers algorithm is given under certain finiteness assumptions and over the base scheme $\Spec \C$, using a comparison between intrinsic and derived blow-ups. The work in \cite{DRS} has applications in (derived) shifted symplectic geometry and enumerative invariants, such as generalized Donaldson--Thomas theory, and generalized Vafa--Witten invariants. A crucial assumption made in \cite{DRS} is the existence of a good moduli space on the level of underlying classical objects. A fully derived picture of good moduli spaces is given in \cite{GMS}, where it is shown that a derived 1-Artin stack $X$ admits a derived good moduli space if and only if $X_\cl$ admits a classical good moduli space.

\subsection*{Deformation to the normal bundle}
The deformation to the normal bundle plays a prominent role in classical intersection theory, as in \cite{FultonIntersection}. A derived version was first defined for quasi-smooth maps of derived schemes in \cite{KhanVirtual}, then generalized to allow arbitrary morphisms of derived algebraic stacks in \cite{Weil}. The derived version has found applications in algebraic $K$-theory and virtual intersection theory in \cite{KhanKG}, in derived algebraic cobordism in  \cite{AnnalaPrecobordism}, and in microlocal sheaf theory in \cite{SchefersMicrolocal}.

Independently of the present work there are (derived) deformation to the normal bundle results in various parts of (derived) analytic geometry. For example, in \cite{JorgeSpreading} such a result for derived analytic spaces is obtained and then used in spreading out Hodge filtrations. The construction is similar to the one found \cite{GaitsgoryStudyII} for the formal algebraic setting, which is independent from \cite{Weil} and involves gluing. A non-derived example can be found in \cite{NiDeformation}, where the author defines a deformation to the normal cone in Arakelov geometry, later used in \cite{NiHilbertSamuel} to prove an arithmetic Hilbert--Samuel type theorem. 

The advantages of the approach in the present work are as follows: it proposes a unification of  previous constructions, it is robust in the sense that it does not use any glueing, and it is general in the sense that it provides a deformation to the normal bundle for any morphism which allows a cotangent complex (in a given geometric context). Of course, in each example one still has to check whether it coincides with the known formalism for a deformation to the normal bundle. Although interesting, these questions are not pursued in the work at hand.

\section{Algebraic contexts}
\label{Sec:Algebraic_contexts}
\subsection*{Recollection on algebraic contexts}
We will use the theory of derived algebraic contexts as found in \cite{RaksitHKR}, simply called algebraic contexts in the present work. The associated algebra objects will be the affine building blocks of our geometry later on. Let us first summarize the parts of \cite{RaksitHKR} relevant to us. For a symmetric monoidal category $\CCC$ we sometimes write the tensor product $(-)\otimes_\CCC(-)$ simply as $(-)\otimes(-)$.

\begin{Ass}
	We assume that all presentable, symmetric monoidal categories are commutative algebra objects in $(\PrL,\otimes)$, where $\otimes$ is the Lurie tensor product. Equivalently, for each presentable, symmetric monoidal category $\EEE$, we assume that the tensor product in $\EEE$ commutes with colimits in each variable separately, see \cite[\S 4.8.1]{LurieHA}. 
\end{Ass}

\begin{Def}
	Let $\CCC$ be a stable, presentable, symmetric monoidal category. A t-structure on $\CCC$ is \textit{compatible} if
	\begin{itemize}
		\item $\CCC_{\leq 0}$ is closed under filtered colimits,
		\item the unit object lies in $\CCC_{\geq 0}$, and
		\item $\CCC_{\geq 0}$ is closed under $\otimes$.
	\end{itemize}
\end{Def}

Let $\CCC$ be a stable category with t-structure such that $\CCC_{\leq 0}$ is closed under filtered colimits. Then the truncation functor $\tau_{\geq 0} : \CCC \to \CCC_{\geq 0}$ commutes with filtered colimits. Indeed, let $M_\alpha$ be a filtered system in $\CCC$, with colimit $M$. Then, since $\CCC_{\leq 0}$ is closed under filtered colimits, it holds that $\colim \tau_{\leq -1} M_\alpha \in \CCC_{\leq -1}$. Since $\tau_{\leq -1}$ is left adjoint to the inclusion $\CCC_{\leq -1} \subset \CCC$, it moreover commutes with all colimits. Therefore, the exact sequences
\[
	\tau_{\geq 0} M_\alpha \to M_\alpha \to \tau_{\leq -1} M_\alpha
\]
induces an exact sequence
\[
	\colim \tau_{\geq 0} M_\alpha \to M \to \tau_{\leq -1} M
\]
which shows that $\colim \tau_{\geq 0} M_\alpha \simeq \tau_{\geq 0} M$.

Recall that an object $X$ in a category $\EEE$ is \textit{projective} if $\EEE(X,-)$ commutes with geometric realizations, \textit{compact} if $\EEE(X,-)$ commutes with filtered colimits, and \textit{compact projective} if $\EEE(X,-)$ commutes with sifted colimits. Then $X$ is compact projective if and only if it is compact and projective, see \cite[Cor.\ 5.5.8.17]{LurieHTT}. 

\begin{Rem}
	If $\EEE$ is the derived category $D_{\geq 0}(\AAA)$ of an abelian category $\AAA$, then $P \in \AAA$ is projective in $\EEE$ (via the canonical map $\AAA \to \EEE$) if and only if $P$ is projective in the classical sense, i.e., if $\pi_0\AAA(P,-)$ is exact, see \cite[Exm.\ 5.5.8.21]{LurieHTT}.
\end{Rem}

Let $\EEE$ be a presentable category, let $G$ be a finite group, and let $BG$ be the classifying groupoid of $G$, which can be constructed as the nerve of $G$. Then an action of $G$ on an object $M \in \EEE$ is a functor $BG \to \EEE$ which sends the unique point in $BG$ to $M$. This gives the category $G\EEE \coloneqq \Fun(BG,\EEE)$ of $G$-objects in $\EEE$.\footnote{In \S \ref{Par:Equivariant_C_geometry} we will encounter actions of group objects internal to $\EEE$. These two notions are unrelated in the present work, in the sense that they are used in completely different parts. We therefore do not dwell on a comparison between the two.} Composition with the projection $BG \to *$ gives the functor $\EEE \to G\EEE$ which endows an object $M \in \EEE$ with trivial $G$-action. The left adjoint is written 
\[
(-)_G \colon G\EEE \to \EEE,
\]
and it sends $M \in G\EEE$ to the \emph{orbits} $M_G$. For example, when $\EEE$ is the category of spaces, then for a $G$-space $X$ it holds that $X_G$ is the space of homotopy orbits.

Now suppose that $\EEE$ is symmetric monoidal. Let $\Sigma_n$ be the symmetric group on $n$ elements. Then the functor
\begin{align*}
	\Sym_\EEE \colon \EEE &\to \EEE \\ M &\mapsto \bigoplus_{n \in \N} (M^{\otimes n})_{\Sigma_n}
\end{align*}
carries the structure of a monad (\cite[Constr.\ 4.1.1]{RaksitHKR}), and thus induces an adjunction
\begin{center}
    \begin{tikzcd}
        \EEE \arrow[r, shift left, "\Sym_\EEE"] & \CAlg_{\EEE}, \arrow[l, shift left, "V_\EEE"]
    \end{tikzcd}
\end{center}
where $\CAlg_{\EEE}$ are the algebra objects for the monad $\Sym_\EEE$, and $V_\EEE$ is the forgetful functor. Note that $\CAlg_{\EEE}$ is also the category of $\Einfty$-algebra objects in $\EEE$ with respect to the symmetric monoidal structure.

\begin{Def}
	An \textit{algebraic context} is a stable, presentable, symmetric monoidal category $\CCC$ endowed with a t-structure, together with a full subcategory $\CCC^0\subset \CCC^\heartsuit$, such that:
	\begin{itemize}
		\item the t-structure is compatible and right-complete,
		\item $\CCC^0$ is a symmetric monoidal subcategory of $\CCC$, and is closed under $\CCC^\heartsuit$-symmetric powers, meaning that for $X \in \CCC^0$ and $n\geq 0$ it holds $\pi_0(X^{\otimes n}_{\Sigma_n}) \in \CCC^0$, and
		\item $\CCC^0$ is closed under finite coproducts in $\CCC$, and the objects form a set of compact projective generators of $\CCC_{\geq 0}$.
	\end{itemize}
\end{Def}

\begin{Rem}
	To deal with set-theoretic issues, we assume that $\CCC$ is $1$-small. 
\end{Rem}

For $T$ a category with finite coproducts, we write $\PPP_\Sigma(T)$ for the full subcategory of $\Fun(T^\op, \Space)$ consisting of presheaves that send finite coproducts in $T$ to products in $\Space$.\footnote{This construction is called the category of \textit{homotopy varieties}, or the \textit{nonabelian derived category}, or the \textit{animation}.}  Yoneda induces a fully faithful functor $T \to \PPP_\Sigma(T)$. If $T$ is symmetric monoidal, then $\PPP_\Sigma(T)$ inherits a symmetric monoidal structure through Day-convolution for which the tensor product commutes with colimits in each variable separately and the Yoneda functor is symmetric monoidal, see \cite[Prop.\ 4.8.1.10, Rem.\ 4.8.1.13]{LurieHA}.

For $\EEE$ a presentable, symmetric monoidal category, the stabilization $\Sp(\EEE)$ can be computed via the tensor  product in $\PrL$ as $\Sp(\EEE) \simeq \Sp \otimes \EEE$, and is thus a commutative algebra object in the category of stable, presentable symmetric monoidal categories, see \cite[\S 4.8.2]{LurieHA}

\begin{Exm}
	For $T$ the category of finitely generated, free $\Z$-modules, it holds $\PPP_\Sigma(T) \simeq (\Mod_\Z)_{\geq 0}$. By \cite[Cor.\ 4.8.1.12]{LurieHA}, the symmetric monoidal structure on $\PPP_\Sigma(T)$ induced by the tensor product on $T$ coincides with the existing symmetric monoidal structure on $(\Mod_\Z)_{\geq 0}$. Likewise, we have $\Sp(\PPP_\Sigma(T)) \simeq \Mod_\Z$, and the symmetric monoidal structure on $\Sp(\PPP_\Sigma(T))$ induced by the one on $\PPP_\Sigma(T)$ coincides with the symmetric monoidal structure on $\Mod_\Z$ induced by the standard one on $\Sp$.
\end{Exm}

From here on, let $\CCC$ be an algebraic context. Then $\CCC^\heartsuit$ is presentable, and has a symmetric monoidal structure given by $X \otimes_{\CCC^\heartsuit} Y \coloneqq \pi_0(X \otimes_\CCC Y)$. In particular, we have a monad $\Sym_{\CCC^\heartsuit}$ on $\CCC^\heartsuit$. 
\begin{Def}
	Let $\LSym_{\CCC_{\geq 0}}$ be the left derived functor of the functor $\CCC^0 \to \CCC_{\geq 0}$ given by $X \mapsto \Sym_{\CCC^\heartsuit} X$.
\end{Def}
In \cite[Constr.\ 4.2.20]{RaksitHKR}, it is shown that $\LSym_{\CCC_{\geq 0}}$ extends to a sifted colimit-preserving functor with monad structure
\[
	\LSym_\CCC \colon \CCC \to \CCC.
\]

\begin{Rem}
	Observe that $\Sym_{\CCC_{\geq 0}}$ and $\LSym_{\CCC}$ both restrict to a sifted colimit-preserving functor $\CCC_{\geq 0} \to \CCC_{\geq 0}$. The difference between these two functors is therefore completely determined by their difference on $\CCC^0$, which in turn comes about from the fact that, in general, $X^{\otimes n}_{\Sigma n}$ is not discrete for $X \in \CCC^0$. For example, in the context $\CCC = \Mod_\Z$, it holds that $\pi_k(\Z^{\otimes n}_{\Sigma_n})$ is  group homology $H_k(\Sigma_n;\Z)$ (with trivial $\Sigma_n$-action on $\Z$). 
\end{Rem}

\begin{Def}
	We write ${\DAlg_{\CCC}}$ for the algebra objects over the monad $\LSym_\CCC$, simply called \textit{$\CCC$-algebras}. The full subcategory of ${\DAlg_{\CCC}}$ spanned by those algebras for which the underlying object in $\CCC$ is connective (resp.\
	 discrete), is written ${\DAlg_{\CCC_{\geq 0}}}$ (resp.\
	  $\DAlg_{\CCC^\heartsuit}$). We also write $\DAlg_{\CCC^0}$ for the full subcategory $\{\LSym_\CCC(M) \mid M \in \CCC^0\} \subset {\DAlg_{\CCC}}$.
\end{Def}

\begin{Prop}
	\label{Prop:DACsum}
	Let $\CCC$ be an algebraic context.
	\begin{enumerate}
		\item The functor $\PPP_\Sigma(\CCC^0) \to \CCC_{\geq 0}$ induced by the inclusion $\CCC^0 \to \CCC_{\geq 0}$ is a symmetric monoidal equivalence. 
		\item The category $\CCC$ is the stabilization $\Sp(\CCC_{\geq 0})$ of $\CCC_{\geq 0}$, and the symmetric monoidal structure on $\CCC$ coincides with the one inherited from $\CCC_{\geq 0}$.
		\item The category $\CCC^\heartsuit$ is the category of finite product-preserving functors $(\CCC^0)^\op \to \Set$.
		\item It holds ${\DAlg_{\CCC_{\geq 0}}} \simeq \PPP_\Sigma(\DAlg_{\CCC^0})$.
		\item Likewise, it holds $\DAlg_{\CCC^\heartsuit} \simeq \CAlg_{\CCC^\heartsuit}$.
	\end{enumerate}
\end{Prop}

\begin{proof}
	See \cite[Rem.\ 4.2.2]{RaksitHKR} for (1), (2), (3), and \cite[Rem.\ 4.2.24]{RaksitHKR} for (4), (5).
\end{proof}

We write the forgetful functor as $U_\CCC \colon {\DAlg_{\CCC}} \to \CCC$. By definition, we have a monadic adjunction
\begin{center}
    \begin{tikzcd}
        \CCC \arrow[r, shift left, "\LSym_\CCC"] & \DAlg_{\CCC}. \arrow[l, shift left, "U_\CCC"]
    \end{tikzcd}
\end{center}

For $A \in {\DAlg_{\CCC}}$ (resp.\ $A' \in \CAlg_{\CCC}$), we write $\DAlg_A$ (resp.\  $\CAlg_{A'}$) for the slice category $(\DAlg_{\CCC})_{A/}$  (resp.\ $(\CAlg_{\CCC})_{A'/}$). The slice category is symmetric monoidal, with tensor product $R \otimes_A R'$ computed as a pushout in ${\DAlg_{\CCC}}$ (resp.\ in $\CAlg_{\CCC}$).

There is a limit and colimit preserving functor $\Theta\colon{\DAlg_{\CCC}} \to \CAlg_{\CCC}$, which induces a symmetric monoidal functor
\[
	\DAlg_A \to \CAlg_{\Theta A},
\]
for each $A\in {\DAlg_{\CCC}}$. We then write $\Mod_A$ for the category of $\Theta A$-modules in $\CCC$. This gives a monadic adjunction 
\begin{center}
    \begin{tikzcd}
        \Mod_A \arrow[r, shift left, "\LSym_A"] & \DAlg_A, \arrow[l, shift left, "U_A"]
    \end{tikzcd}
\end{center}
where $U_A$ is the forgetful functor (which preserves sifted colimits). If $A$ is connective, then the category $\Mod_A$ inherits a t-structure from $\CCC$, with subcategory of connective objects $\Mod_{A_{\geq 0}}$ those $A$-modules which are connective as $\CCC$-modules.

\begin{Lem}\label{Lem:ZeroDef}
	Let $A \in {\DAlg_{\CCC_{\geq 0}}}$, and write  $\DAlg_{A^0}$ for the full subcategory of $\DAlg_A$ of objects of the form $\LSym_A(A \otimes M)$, where $M \in \CCC^0$. Then the canonical functor $F\colon\PPP_\Sigma(\DAlg_{A^0}) \to \DAlg_{A_{\geq 0}} \coloneqq (\DAlg_{\CCC_{\geq 0}})_{A/}$ is an equivalence.
\end{Lem}

\begin{proof}
	By \cite[Prop.\ 5.5.8.22]{LurieHTT}, to show that \[F\colon\PPP_\Sigma(\DAlg_{A^0}) \to \DAlg_{A_{\geq 0}}\] is fully faithful, it suffices to show that the inclusion \[f\colon\DAlg_{A^0} \to \DAlg_{A_{\geq 0}}\] is fully faithful, with essential image consisting of compact projective objects of $\DAlg_{A_{\geq 0}}$. The first point is true by construction. Using that the forgetful functor $\DAlg_{A_{\geq 0}} \to \Mod_{A_{\geq 0}}$ preserves sifted colimits, one shows that objects in $\DAlg_{A^0}$ are compact projective in $\DAlg_{A_{\geq 0}}$. 
	
	Consider the functor
	\[
	H\colon\DAlg_{A_{\geq 0}} \to \PPP_\Sigma(\DAlg_{A^0})
	\]
	induced by restricted Yoneda. Recall that $F$ is the left Kan extension of $\DAlg_{A^0} \to \DAlg_{A_{\geq 0}}$ along the inclusion $\DAlg_{A^0} \to \PPP_\Sigma(\DAlg_{A^0})$ induced by Yoneda. It follows that $G \coloneqq FH$ is the functor 
\begin{align*}
			G\colon\DAlg_{A_{\geq 0}} &\to \DAlg_{A_{\geq 0}} \\ B &\mapsto \colim_{(R \to B) \in \DAlg_{A^0/B}} R
\end{align*}
	 where $\DAlg_{A^0/B}$ is the full subcategory of $\DAlg_{A/B} \coloneqq (\DAlg_A)_{/B}$ spanned by objects of the form $R \to B$ in the arrow-category $\Arr(\DAlg_A)$ with $R \in \DAlg_{A^0}$.  To show that $F$ is essentially surjective, it thus suffices to show that $G$ is essentially surjective. 
	
	Let $B \in \DAlg_{A_{\geq 0}}$ be given. We will show that $G(B) \simeq B$. To this end, observe that for any $M \in \CCC^0$, it holds
	 \[
	 \DAlg_A(\LSym_A(A\otimes M), B) \simeq {\DAlg_{\CCC}}(\LSym(M),B)
	 \] 
	 by \cite[Rem.\ 4.2.29]{RaksitHKR}, which gives us an equivalence $\DAlg_{A^0/B} \simeq \DAlg_{\CCC^0/B}$ of categories, where $\DAlg_{\CCC^0/B} \subset (\DAlg_{\CCC})_{/B}$ is similarly defined as $\DAlg_{A^0/B}$.  We thus have
	 \[
	 G(B) \simeq  \colim_{(R \to B) \in \DAlg_{\CCC^0/B}} R \simeq B
	 \]
	 because the canonical functor $\PPP_\Sigma(\DAlg_{\CCC^0}) \to {\DAlg_{\CCC_{\geq 0}}}$ is an equivalence by Proposition \ref{Prop:DACsum}. 
\end{proof}

\begin{Exm}
	Consider $\CCC = \Mod_\Z$, and let $\CCC^0$ be the full subcategory spanned by finite free $\Z$-modules, with standard t-structure and symmetric monoidal structure. This is an algebraic context. For this context, we drop the $\CCC$ from the notation. Then $\DAlg_{\geq 0}$ is equivalent to the category $\PPP_{\Sigma}(\Poly)$ of simplicial commutative rings, where $\Poly$ is the category of finitely generated polynomial rings over $\Z$.  
\end{Exm}

\begin{Exm}
	The symmetric monoidal category $\Sp$ of spectra, with the standard t-structure, is not an algebraic context. Indeed, the unit in $\Sp$ is the sphere spectrum, which is not discrete, contradicting the requirement that $\CCC^0 \subset \CCC$ is a symmetric monoidal subcategory of a given algebraic context $\CCC$.
\end{Exm}

\begin{Def}
	Let $\CCC,\DDD$ be algebraic contexts. Then a \textit{morphism of algebraic contexts} $\CCC \to \DDD$ is a colimit preserving, symmetric monoidal functor $\CCC \to \DDD$ which is right t-exact and which carries $\CCC^0$ into $\DDD^0$.
\end{Def}
\begin{Rem}
	\label{Rem:InducedAdjoint}
	A morphism $F \colon \CCC \to \DDD$ of algebraic contexts has a right adjoint $G$. By \cite[Rem.\ 4.2.25]{RaksitHKR}, this adjunction ascends to an adjunction between ${\DAlg_{\CCC}}$ and $\DAlg_{\DDD}$, also written $F \dashv G$. Here, $F$ commutes with the forgetful functors and the symmetric algebra functors, and $G$ commutes with the forgetful functors.
\end{Rem}

For any algebraic context $\CCC$, there is a unique morphism of algebraic contexts $\Mod_\Z \to \CCC$ by \cite[Rem.\ 4.3.2]{RaksitHKR}. This morphism induces a functor $\DAlg \to {\DAlg_{\CCC}}$ which restrict to connective objects. For $R \in \DAlg$, we write the image of $R$ under $\DAlg \to {\DAlg_{\CCC}}$ simply as $R$ when convenient.

\begin{Prop}
	\label{Prop:iotatau}
	Let $\iota\colon {\DAlg_{\CCC_{\geq 0}}} \to {\DAlg_{\CCC}}$ be the inclusion. Then the right adjoint $\tau_{\geq 0}$ to the inclusion $\CCC_{\geq 0} \to \CCC$ induces an adjunction
    \begin{center}
    \begin{tikzcd}
        \DAlg_{\CCC_{\geq 0}} \arrow[r, shift left, "\iota"] & \DAlg_{\CCC} \arrow[l, shift left, "\tau_{\geq 0}"]
    \end{tikzcd}
    \end{center}
    in the sense that $\tau_{\geq 0} \circ U_\CCC \simeq U_\CCC \circ \tau_{\geq 0}$, where $U_\CCC$ is the forgetful functor.
\end{Prop}

\begin{proof}
	Since $\DAlg_{\CCC^0} \to {\DAlg_{\CCC}}$ preserves finite coproducts, the functor $\iota$ preserves all colimits by \cite[Prop.\ 5.5.8.15]{LurieHTT}, hence admits a right adjoint, written $\tau_{\geq 0}$. Now $\tau_{\geq 0} \circ U_\CCC \simeq U_\CCC \circ \tau_{\geq 0}$ follows from the fact that $\LSym_\CCC$ commutes with the inclusions $\CCC_{\geq 0} \to \CCC$ and $\iota$.
\end{proof}

\subsection*{$\M$-graded algebraic contexts}
Let $\CCC$ be an algebraic context. We consider commutative monoids (always discrete) as discrete, symmetric monoidal categories with no morphisms other than identities.
\begin{Def}
	For a commutative monoid $\M$, let $\CCC^\M$ be the category $\Fun(\M,\CCC)$ endowed with symmetric monoidal structure through Day convolution.
\end{Def}

For $N \in \CCC^\M$ and $a\in \M$, we write $N_a$ for the image of $a$ under $N$ in $\CCC$. For $N' \in \CCC$, we write $N'(a) \in \CCC^\M$ for the object which is $N'$ concentrated in degree $-a$. We then define $(\CCC^\M)^0$ as the full subcategory of $\CCC^\M$ spanned by finite coproducts of objects of the form $N(a)$, with  $a\in \M$ and $N \in \CCC^0$. Furthermore, we consider $\CCC_{\geq 0}^\M = \Fun(\M,\CCC_{\geq 0})$ as a full subcategory of $\CCC^\M$, and similarly for $\CCC_{\leq 0}^\M$. This gives a t-structure on $\CCC^\M$.  
\begin{Lem}
	The category $\CCC^\M$, with  t-structure $(\CCC^\M_{\geq 0},\CCC^\M_{\leq 0})$ and subcategory $(\CCC^\M)^0$, defines an algebraic context. 
\end{Lem}
\begin{proof}
	This is straightforward, using that limits and colimits in $\CCC^\M$ are computed homogeneous-degree wise, together with the formula for the Day convolution product, which  reads that
	\[
		(N \otimes N')_a = \bigoplus_{a = b +c} N_b \otimes N'_c
	\]
	for $N,N' \in \CCC^\M$ and $a \in \M$.
\end{proof}
Objects in $\CCC^\M$ are called \textit{$\M$-graded objects in $\CCC$}. Likewise for $\DAlg_{\CCC^\M}$, etc.

\begin{Not}
	Let $l\colon \M \to \K$ be a homomorphism of commutative monoids. We write $l^!\colon\CCC^\K \to \CCC^\M$ for the functor which is given by precomposition with $l$, and $l_!$ for the left adjoint of $l^!$. 
\end{Not}
Observe that $l_!(N)$ is the left Kan extension of $N$ along $l$, for $N \in \CCC^\M$. Hence, for $b \in \K$ it holds that
\[
	l_!(N)_b = \bigoplus_{la = b} N_a.
\]
\begin{Prop}
	\label{Prop:EquivAdj}
	For $l\colon \M \to \K$ a map of commutative monoids, the morphism
	\[
		l_! \colon \CCC^\M \to \CCC^\K
	\]
	is a morphism of algebraic contexts. 
\end{Prop}
\begin{proof}
	Clearly $l_!$ preserves colimits. It is right t-exact since $\CCC_{\geq 0}$ is closed under colimits in $\CCC$. Using the formula for the Day convolution, together with the fact that the tensor product commutes with colimits in each variable separately, one shows that $l_!$ is also symmetric monoidal. 
\end{proof}
We thus get an adjunction
\begin{center}
    \begin{tikzcd}
        \DAlg_{\CCC^\M} \arrow[r, shift left, "l_!"] & \DAlg_{\CCC^\K}. \arrow[l, shift left, "l^!"]
    \end{tikzcd}
\end{center}
The functors $l_!,l^!$ both commute with morphisms of algebraic contexts. Consequently, we get a functor
\[
	(\CCC,\M) \mapsto  \DAlg_{\CCC^\M}
\]
from the category of pairs $(\CCC,\M)$, where $\CCC$ is an algebraic context and $\M$ is a commutative monoid, to the category $\PrL$. 

\begin{Exm}
	Consider the unique map $p\colon \M \to 0$. Then $p_!$ is the functor which forgets the grading. That is, it sends $B \in \DAlg_{\CCC^\M}$ to the algebra with underlying module $\bigoplus_{a \in \M} B_a \in \CCC$. We can thus think of $p^!\colon {\DAlg_{\CCC}} \to \DAlg_{\CCC^\M}$ as the functor that sends $R \in {\DAlg_{\CCC}}$ to the monoid ring of $\M$ over $R$, and we put
\[
		R[\M] \coloneqq p_{!}p^!(R).
\]
\end{Exm}

\begin{Exm}
	Consider the unique map $j\colon 0 \to \K$. Then $j_!$ is the functor that endows an object with trivial grading, and $j^!$ is the functor that sends a graded object $X$ to the graded piece $X_0$.
\end{Exm}

 Let $\CCC,\DDD$ be algebraic contexts, write $\bDelta_{-\infty}$ for the category of objects of the form $\{-\infty\} \cup [0,m] \subset \{-\infty\} \cup \N$ for $m \geq -1$, with order-preserving maps which preserves $-\infty$ as morphisms. Recall that the Barr--Beck--Lurie Theorem states that an adjunction $F \dashv G \colon \CCC\rightleftarrows \DDD$ is comonadic (i.e., that $G^\op \dashv F^\op$ is monadic) if and only if $F$ is conservative and preserves $F$-cosplit totalizations \cite[Thm.\ 4.7.3.5]{LurieHA}. Here, a totalization 
 \[  X^{-1} \to X^0 \rightrightarrows X^1 \mathrel{\substack{\textstyle\rightarrow\\[-0.6ex]
 		\textstyle\rightarrow \\[-0.6ex]
 		\textstyle\rightarrow}} \cdots  \]
 of a cosimplicial object $X \colon [n] \mapsto X^n$ in $\CCC$ is \emph{$F$-cosplit} if the augmented cosimplicial object $FX$ extends to $\bDelta_{-\infty}$, in which case necesarily $FX_{-1} \simeq \lim FX$ \cite[Lem.\ 6.1.3.16]{LurieHTT}.

\begin{Lem}
	\label{Lem:AscMonadic}
	Let $F \colon \CCC \to \DDD$ be a left adjoint with right adjoint $G$, with $F$ conservative, and suppose we have a commutative diagram
	\begin{center}
		\begin{tikzcd}
			\CCC \arrow[d, "U"] \arrow[r, "F"] & \DDD \arrow[d, "V"] \\
			\CCC' \arrow[r, "F'"] & \DDD'
		\end{tikzcd}
	\end{center}
	such that $F'$ preserves $F'$-cosplit totalizations, $U,V$ preserve totalizations, and $V$ is conservative. Then $F \dashv G$ is comonadic.
\end{Lem}

\begin{proof}
	Let $A$ be an $F$-cosplit cosimplicial diagram in $\CCC$. Then since $VF \simeq F'U$, the composition $UA$ is $F'$-cosplit, and thus 
	\begin{align*}
			VF(\lim A) &\simeq F' U(\lim A)\\
			 &\simeq F'(\lim UA) \simeq \lim (F'UA) \simeq \lim (VF A) \simeq V \lim (FA)
	\end{align*}
	since $U,V$ preserve totalizations. Since $V$ is conservative, this shows that $F(\lim A) \simeq \lim (FA)$.
\end{proof}

\begin{Lem}
	\label{Lem:Monadicity}
	Let $F\colon \CCC \to \DDD$ be a morphism of algebraic contexts, with right adjoint $G$, such that $F \dashv G$ is comonadic. Then the induced adjunction
    \begin{center}
    \begin{tikzcd}
        \DAlg_{\CCC} \arrow[r, shift left, "F'"] & \DAlg_{\DDD} \arrow[l, shift left, "G'"]
    \end{tikzcd}
    \end{center}
	is comonadic.
\end{Lem}

\begin{proof}
	Since the forgetful functors $U_\CCC\colon {\DAlg_{\CCC}} \to \CCC$ and $U_\DDD \colon \DAlg_{\DDD} \to \DDD$ preserve limits and are conservative, this follows from Lemma \ref{Lem:AscMonadic}.
\end{proof}

\begin{Prop}
	\label{Prop:GradvsAct}
	The adjunction
    \begin{center}
    \begin{tikzcd}
        \DAlg_{\CCC^\M}  \arrow[r, shift left, "\mathrm{forget}"] & \DAlg_{\CCC} \arrow[l, shift left, "{(-)[\M]}"]
    \end{tikzcd}
    \end{center}
    induced by $p \colon \M \to 0$ is comonadic. Consequently, the forgetful functor induces a symmetric monoidal equivalence betweem $\DAlg_{\CCC^\M}$ and the category of coalgebra objects over $\mathbbm{1}_\CCC[\M]$ in $\DAlg_{\CCC}$.
\end{Prop}

\begin{proof}
	We first show that the stated adjunction is comonadic. By Lemma \ref{Lem:Monadicity}, it suffices to show that the adjunction
    \begin{center}
    \begin{tikzcd}
        \CCC^\M  \arrow[r, shift left, "p_!"] & \CCC \arrow[l, shift left, "p^!"]
    \end{tikzcd}
\end{center}
	is comonadic. Clearly, $p_!$ is conservative. By the Barr--Beck--Lurie Theorem, it thus suffices to show that $p_!$ preserves $p_!$-cosplit totalizations.
	
	Let $X$ be a $p_!$-cosplit cosimplicial diagram $[n] \mapsto X^n$ in $\CCC^\M$, write $Z$ for the limit of $p_!X$. For $a \in \M$, write $X_a$ for the cosimplicial diagram $[n] \mapsto X^n_a$ in $\CCC$. Taking the limit gives us an augmentation $X^{-1}_a \to X_a$. Write $X_a'$ for the resulting augmented cosimplicial object in $\CCC$, and $p_!X'$ for the augmented cosimplicial object in $\CCC$ resulting from restricting the cosplitting of $p_!X$. Observe that the latter augmentation is $Z \to p_!X$.
	
	We have a retract diagram
	\[
		X_a' \xrightarrow{\iota_a} p_!X' \xrightarrow{\rho_a} X_a'.
	\]
	Indeed, the inclusions $\iota_a^n \colon X_a^n \to (p_!X)^n$ and projections $\rho_a^n \colon (p_!X)^n \to X_a^n$ induce a retract diagram $X_a \to p_!X  \to X_a$. Taking limits, this gives us a retract diagram $X_a^{-1} \to Z \to X_a^{-1}$ which extends the given retract $X_a \to p_!X \to X_a$ to \[X_a' \xrightarrow{\iota_a} p_!X' \xrightarrow{\rho_a} X_a'.\]
	
	The key of the argument now is that cosplit augmented objects are stable under retracts by \cite[Cor\ 4.7.2.13]{LurieHA}. Hence $X_a'$ is cosplit, since $p_!X'$ is. Taking the direct sum of these cosplittings gives us a cosplitting of the augmentation
	\[
		\bigoplus_{a \in \M} X_a^{-1} \to X,
	\]
	which shows 
	\[
		\lim(p_! X) = Z \simeq \bigoplus_{a \in \M} X_a^{-1} \simeq p_! (\lim X),
	\]
	since limits in $\CCC^\M$ are computed homogeneous-degree wise.
	
	Now the second claim follows from the fact that $A[\M] \simeq A \otimes \mathbbm{1}_\CCC[\M]$ for $A \in {\DAlg_{\CCC}}$.
\end{proof}

Observe that the equivalence between $\DAlg_{\CCC^\M}$ and the category of coalgebra objects over $\mathbbm{1}_\CCC[\M]$ from Proposition \ref{Prop:GradvsAct}  restricts to an equivalence between connective objects.

\begin{Rem}
	For $\CCC = \Mod$ and $\M = \Z$, Proposition \ref{Prop:GradvsAct} shows that the category of connective, $\Z$-graded algebras is anti-equivalent to the category of affine schemes with $\Spec (\Z[\Z])$-action. This recovers and generalizes \cite[Cor.\ 4.5.5]{HekkingGraded}. However, the argument in  \cite{HekkingGraded} was significantly more involved, essentially because $\DAlg^\Z_{\geq 0}$ was not realized as the connective algebras over the monad $\LSym_{\Mod_\Z^\Z}$.
\end{Rem}

\subsection*{The cotangent complex: local case}
\label{subsection:LocCotangent}
Let $\CCC$ be an algebraic context, and let $A \in {\DAlg_{\CCC}}$ be given. For $M \in \Mod_A$, we write $A \oplus M \in {\DAlg_{\CCC}}$ for the trivial square-zero extension as in \cite{RaksitHKR}. 

Write $\DAlg\Mod_A$ for the category of pairs $(C,M)$, where $C \in \DAlg_A$ and $M \in \Mod_C$. Then the functor 
\[
	\DAlg\Mod_A \to \DAlg_A \colon (C,M) \mapsto C \oplus M
\]
has as left adjoint  $B \mapsto (B,L_{B/A})$, where $L_{B/A}$ is by definition the  cotangent complex of $B$ over $A$.

\begin{Rem}
	Let $A \to B$ and $A\to C$ in ${\DAlg_{\CCC}}$ be given. For $M \in \Mod_C$, the universal property of the cotangent complex gives us an equivalence
	\[
		\DAlg\Mod_A((B,L_{B/A}),(C,M)  ) \simeq \DAlg_A(B,C \oplus M).
	\]
	The fiber of this equivalence over a given map $B \to C$ of $A$-algebras gives the familiar equivalence
	\[
		\Mod_C(L_{B/A}\otimes_B C, M) \simeq \DAlg_{A/C}(B,C \oplus M),
	\]
	where $\DAlg_{A/C}$ is again $(\DAlg_A)_{/C}$. Observe, however, that $A,B,C,M$ are now allowed to be nonconnective.
\end{Rem}

\begin{Prop}
	\label{Prop:CotangentofSur}
	Let $A \to B$ be a map in ${\DAlg_{\CCC_{\geq 0}}}$ for which the fiber is connective. Then $L_{B/A}$ is $1$-connective.
\end{Prop}

\begin{proof}
	By \cite[Rem.\ 4.4.12]{RaksitHKR}, we know that $L_{B/A}$ is connective. First observe that $\pi_0L_{B/A} \simeq L_{\pi_0B/\pi_0 A}$. Indeed, since $\pi_0$ and $L_{-/A}$ both preserve sifted colimits on connective objects, this can be checked on \[B = \LSym_A(A \otimes M)\] where $M \in \CCC^0$. Then the statement follows from \cite[Ex.\ 4.4.11]{RaksitHKR}.
	
	 It now suffices to show that $\pi_0 L_{\pi_0B/\pi_0A} \simeq 0$. Let $M$ be a discrete $\pi_0B$-module. Then
	\[
	\Mod_{\pi_0B} (L_{\pi_0B/\pi_0 A},M) \simeq \DAlg_{\pi_0A/\pi_0B}(\pi_0B,\pi_0B \oplus M) \simeq \{*\},
	\]
	since $\pi_0A \to \pi_0B$ is an epimorphism in the abelian category $\CCC^\heartsuit$. The statement follows.
\end{proof}

\section{Geometric contexts}
\label{Sec:Geometric_contexts}
Let still $\CCC$ be an algebraic context. 
\subsection*{$\CCC$-stacks and $\CCC_{\geq 0}$-stacks}
For any category $\EEE$, write $\PPP(\EEE) \coloneqq \Fun(\EEE^\op, \Space)$ for the category of prestacks on $\EEE$. 

\begin{Def}
	We define the category ${\Aff_{\CCC_{\geq 0}}}$ of \textit{affine $\CCC$-schemes} as 
	\[ {\Aff_{\CCC_{\geq 0}}} \coloneqq \DAlg_{\CCC_{\geq 0}}^\op \]
	and the category ${\Aff_{\CCC}}$ of \textit{nonconnective affine $\CCC$-schemes} as $\DAlg_{\CCC}^\op$.
	
	We write 
	\begin{align*}
		\Spec \colon \DAlg_{\CCC_{\geq 0}} &\to \PPP({\Aff_{\CCC_{\geq 0}}}) \\
		\Spec^\nc \colon \DAlg_\CCC &\to \PPP({\Aff_{\CCC}})
	\end{align*}
	for the contravariant functors induced by Yoneda.
\end{Def}

\begin{Rem}
	Formally, we take prestacks to be $\Space_1$-valued. Since we also assume $\CCC$ to be  $1$-small, in general $\PPP(\CCC)$ will be $2$-small.
\end{Rem}

\begin{Exm}
	\label{Ex:DAGstacks}
	Consider the context $\CCC = \Mod_\Z$. We recall the basic setup of derived algebraic geometry (see, e.g., \cite[\S 2]{GaitsgoryStudy}). In this example, we drop the $\CCC$ from our notation (as before). Put $\Aff^\heartsuit \coloneqq (\DAlg^\heartsuit)^\op$, and endow $\Aff, \Aff^\heartsuit$ with the \'{e}tale topology. Consider the adjunction
    \begin{center}
    \begin{tikzcd}
        \Aff^\heartsuit \arrow[r, shift left, "j"] & \Aff, \arrow[l, shift left, "(-)_\cl"]
    \end{tikzcd}
    \end{center}
	where $j$ is the inclusion and $(-)_\cl$ the truncation. Now composition with $j$ and left Kan extension gives us an adjunction
    \begin{center}
    \begin{tikzcd}
        \PPP(\Aff^\heartsuit) \arrow[r, shift left, "j_!"] & \PPP(\Aff), \arrow[l, shift left, "(-)_\cl"]
    \end{tikzcd}
    \end{center}
    where $j_!$ is fully faithful, and $(-)_\cl$ is composition with $j$. Recall that $(-)_\cl$ sends stacks in $\PPP(\Aff)$ to stacks in $\PPP(\Aff^\heartsuit)$, which gives us an adjunction
    \begin{center}
    \begin{tikzcd}
        \St(\Aff^\heartsuit) \arrow[r, shift left, "\iota"] & \St(\Aff), \arrow[l, shift left, "(-)_\cl"]
    \end{tikzcd}
    \end{center}
    where $\iota$ is given by the composition of the stackification functor $\PPP(\Aff) \to \St(\Aff)$ with the restriction of $j_!$ to $\St(\Aff^\heartsuit) \to \PPP(\Aff)$. The functor $\iota$ is still fully faithful, and $(-)_\cl$ remains unchanged.
	
	We also have an adjunction $(-)_\cl \dashv j_*$, given by right Kan extension. Since also $j_*$ preserves stacks, this adjunction restricts to stacks. Note, however, that $j_*$ is less commonly studied in derived algebraic geometry. 
\end{Exm}

Returning to the general algebraic context $\CCC$, we want to study stacks both on ${\Aff_{\CCC_{\geq 0}}}$ and on ${\Aff_{\CCC}}$. Moreover, we want these stacks to interplay nicely with one another. In particular, we wish for an adjunction
\begin{center}
    \begin{tikzcd}
        \St({\Aff_{\CCC}}) \arrow[r, shift left, "F"] & \St({\Aff_{\CCC_{\geq 0}}}) \arrow[l, shift left, "G"]
    \end{tikzcd}
\end{center}
between stacks on ${\Aff_{\CCC}}$ and stacks on ${\Aff_{\CCC_{\geq 0}}}$ (with respect to some topology). For applications in derived algebraic geometry, we are ultimately interested in $\St({\Aff_{\CCC_{\geq 0}}})$. We therefore would like $G$ to be fully faithful, to preserve limits and colimits, and to preserve affine objects. We will take our cue from Example \ref{Ex:DAGstacks}. 
\begin{Def}
	Upon applying $(-)^\op$ to the adjunction $\iota \dashv \tau_{\geq 0}$ from Proposition \ref{Prop:iotatau}, we obtain the adjunction
    \begin{center}
    \begin{tikzcd}
        \Aff_{\CCC} \arrow[r, shift left, "t"] & \Aff_{\CCC_{\geq 0}} \arrow[l, shift left, "i"]
    \end{tikzcd}
    \end{center}
where $i$ is the inclusion.
\end{Def} 
The adjunction $t \dashv i$ will be lifted to the level of stacks in Proposition \ref{Prop:StAdj}, and this will play a prominent role in the rest of the paper. For now, using Kan extensions, $t\dashv i$ gives us the following diagram
\begin{equation}
\label{Eq:PrestAdjs}
\begin{tikzcd}
\PPP({\Aff_{\CCC_{\geq 0}}}) \arrow[r, bend left, "i_!"] \arrow[r, bend right, swap, "i_*"] & \PPP({\Aff_{\CCC}})\arrow[l, swap, "i^*"] \arrow[r, bend left,  "t_!"] \arrow[r, bend right, swap, "t_*"]   & \PPP({\Aff_{\CCC_{\geq 0}}}) \arrow[l, swap, "t^*"]
\end{tikzcd}
\end{equation}
where $t^*$ resp.\ $i^*$ is given by composition with $t$ resp.\ $i$, and we have adjunctions 
\begin{align*}
	i_! \dashv i^* \dashv i_*,	 && 	t_! \dashv t^* \dashv t_*.
\end{align*}
\begin{Lem}
	\label{Lem:it}
	Consider the diagram (\ref{Eq:PrestAdjs}).
	\begin{enumerate}
		\item There are natural equivalences $i_! \simeq t^*$ and $i^* \simeq t_*$.
		\item The functor $i_!$ (hence $t^*$) is fully faithful and preserves affines.
	\end{enumerate}
\end{Lem}
\begin{proof}
	The equivalence $i_! \simeq t^*$ follows by the pointwise formula for left Kan extensions, together with the fact that $i$ is a right adjoint to $t$. Then $i^* \simeq t_*$ follows by uniqueness of adjoints.
	
	Now $i_!$ is fully faithful since $i$ is. Finally, we have
	\[
	(i_! \Spec R)(B) \simeq \Spec R (\tau_{\geq 0} B) \simeq \Spec^\nc(\iota R)(B)
	\]
	for $R \in {\DAlg_{\CCC_{\geq 0}}}$ and $B \in {\DAlg_{\CCC}}$, where \[\iota\colon {\DAlg_{\CCC_{\geq 0}}} \to {\DAlg_{\CCC}}\] is the inclusion.
\end{proof}

Let $\EEE$ be a complete category and $\TTT$ a site. We define an \textit{$\EEE$-valued stack} on $\TTT$ as a functor $\TTT^\op \to \EEE$ that satisfies descent for the topology $\tau$ in the obvious way.

\begin{Def}
	\label{Def:GeomCont}
	A \textit{geometric context} is an algebraic context $\CCC$ together with a subcanoncial topology $J$ on ${\Aff_{\CCC}}$ such that the following holds.
	\begin{enumerate}
		\item $\{ \Spec^\nc (\iota \tau_{\geq 0} A_\alpha) \to \Spec^\nc (\iota \tau_{\geq 0} A) \}_\alpha$ is a $J$-covering family, for any given $J$-covering family $\{ \Spec^\nc A_\alpha \to \Spec^\nc A \}_{\alpha}$.
		\item The functor $\Mod_{(-)} \colon \Aff_{\CCC}^\op \to \Cat$ is a $\Cat$-valued stack.
	\end{enumerate}	 
\end{Def}

	Let $(\CCC,J)$ be a geometric context. Let $J_{\geq 0}$ be the restriction of $J$ to ${\Aff_{\CCC_{\geq 0}}}$, which is a topology since ${\Aff_{\CCC_{\geq 0}}} \to {\Aff_{\CCC}}$ preserves fiber products. The category of \textit{$\CCC_{\geq 0}$-stacks}, written ${\St_{\CCC_{\geq 0}}}$, is then the full subcategory of $\PPP({\Aff_{\CCC_{\geq 0}}})$ spanned by stacks with respect to $J_{\geq 0}$. Likewise, the category of \textit{ $\CCC$-stacks}, written ${\St_{\CCC}}$, consists of stacks on ${\Aff_{\CCC}}$ with respect to $J$. Since we do not compare different topologies on ${\Aff_{\CCC}}$, suppressing $J$ from the notation should not lead to any confusion.

\begin{Rem}
	Since we assumed that $\CCC \in \Cat_1$, it follows that ${\St_{\CCC}}$ is a presentable category in $\Cat_2$.
\end{Rem}

The following example will be used in Section \ref{Sec:Applications_and_examples} to promote connective (derived, as always) algebraic/analytic geometry to nonconnective algebraic/analytic geometry.
\begin{Exm}
	\label{Ex:InducedContext}
	Let $\CCC$ be an algebraic context, and suppose $\Aff_{\CCC_{\geq 0}}$ is endowed with a subcanonical  topology $J_{\geq 0}$ such that
	\[ \Mod_{(-)} \colon \Aff_{\CCC_{\geq 0}}^\op \to \Cat \]
	is a $\Cat$-valued stack, and such that each covering family $\{T_\alpha' \to T'\}_{\alpha}$ in $J_{\geq 0}$ is finite. We claim that there is a smallest topology $J$ on $\Aff_{\CCC}$ which makes $(\CCC,J)$ into a geometric context, and such that the restriction of $J$ to $\Aff_{\CCC_{\geq 0}}$ is $J_{\geq 0}$.
	
	Let $T \coloneqq \Spec^\nc B$ in $\Aff_\CCC$ be given. For a $J_{\geq 0}$ covering family $\{T'_\alpha \to T'\}_{\alpha}$ of $T' \coloneqq \Spec \tau_{\geq 0}B$, define $\{T_\alpha \to T\}_\alpha$ by pulling back along the morphism $T \to T'$ induced by the counit $\iota \tau_{\geq 0} B \to B$. Then certainly $\{T_\alpha \to T\}_\alpha$ needs to be a $J$-covering family. 	
	
	In fact, taking all families of the form just described is already a topology, and clearly it is the smallest one that could make $(\CCC,J)$ into a geometric context. It thus suffices to show that $(\CCC,J)$ is indeed a geometric context.
	
	Let a $J$-covering family $\{T_\alpha \to T\}_\alpha$ be given, induced by a $J_{\geq 0}$-covering family $\{T'_\alpha \to T'\}_\alpha$ as above. Since $\bigsqcup_\alpha T_\alpha' \to T'$ is an effective epimorphism, so is $\bigsqcup_\alpha T_\alpha \to T$. Hence $J$ is subcanonical.
	
	Since the covering families in $J_{\geq 0}$ are assumed to be finite, to show descent for $\Mod_{(-)}$ we can assume that both of the given families $\{T_\alpha \to T\}_\alpha$ and $\{T_\alpha' \to T'\}_\alpha$ consist of a single morphism. By assumption, the diagram
	\[ \Mod_{T'} \to \Mod_{T'_\alpha} \rightrightarrows \Mod_{T'_\alpha \times_{T'} T'_\alpha} \mathrel{\substack{\textstyle\rightarrow\\[-0.6ex]
			\textstyle\rightarrow \\[-0.6ex]
			\textstyle\rightarrow}} \cdots \]
	is limiting, meaning that it exhibits $\Mod_{T'}$ as a limit of the simplicial diagram induced by the \v{C}ech nerve of $T'_\alpha \to T'$. By \cite[Prop.\ 7.3.2]{GaitsgoryStudy}, also
	\[ \Mod_{T} \to  \Mod_{T'_\alpha} \otimes_{\Mod_{T'}} \Mod_T \rightrightarrows \Mod_{T'_\alpha \times_{T'} T'_\alpha}\otimes_{\Mod_{T'}} \Mod_T \mathrel{\substack{\textstyle\rightarrow\\[-0.6ex]
			\textstyle\rightarrow \\[-0.6ex]
			\textstyle\rightarrow}} \cdots \]
	is limiting, where the tensor product is the Lurie tensor product. By \cite[Prop.\ 4.1]{BenIntegral}, this diagram is equivalent to 
	\[ \Mod_{T} \to \Mod_{T_\alpha} \rightrightarrows \Mod_{T_\alpha \times_{T} {T}_\alpha} \mathrel{\substack{\textstyle\rightarrow\\[-0.6ex]
			\textstyle\rightarrow \\[-0.6ex]
			\textstyle\rightarrow}} \cdots \]
	obtained by the \v{C}ech nerve of $T_\alpha \to T$. This concludes the argument.
	
	Note how  $\{T_\alpha \to T\}_\alpha$ being a pullback of $\{T'_\alpha \to T'\}_\alpha$ is reminiscent of the strongness condition in ordinary derived algebraic geometry from \cite[Def.\ 2.2.2.1]{ToenHAGII}.
\end{Exm}
	
	A morphism $f\colon X \to Y$ in $\PPP({\Aff_{\CCC}})$ is a \textit{$J$-equivalence} if for all $Z \in {\St_{\CCC}}$ it holds that 
	\[
		f^*\colon\PPP({\Aff_{\CCC}})(Y,Z) \to \PPP({\Aff_{\CCC}})(X,Z)
	\]
	is an equivalence. We similarly define \textit{$J_{\geq 0}$-equivalences}. Write the localizations as 
\begin{center}
    \begin{tikzcd}
        \PPP({\Aff_{\CCC_{\geq 0}}}) \arrow[r, shift left, "L_{\geq 0}"] & \St_{\CCC_{\geq 0}}, \arrow[l, shift left, "I_{\geq 0}"] && 
        \PPP({\Aff_{\CCC}}) \arrow[r, shift left, "L"] & \St_{\CCC}. \arrow[l, shift left, "I"]
    \end{tikzcd}
\end{center}   
\begin{Prop}
	\label{Prop:StAdj}
	Consider again the diagram (\ref{Eq:PrestAdjs}).
	\begin{enumerate}
		\item The functor $t^*$ (hence $i_!$) preserves stacks.
		\item The functor $t_!$ sends $J$-equivalences to $J_{\geq 0}$-equivalences.
		\item The functor $L_{\geq 0} t_! I$ is left adjoint to the restriction $t^* \colon {\St_{\CCC_{\geq 0}}} \to {\St_{\CCC}}$.
		\item The functor $i^*$ (hence $t_*$) preserves stacks.
		\item The functor $L_{\geq 0}t_!I$ restricts to the functor $t \colon \Aff_{\CCC} \to \Aff_{\CCC_{\geq 0}}$.
	\end{enumerate}
\end{Prop}	
\begin{proof}
	The first point follows from the fact that $t$ sends $J$-covering families to $J_{\geq 0}$-covering families, and then (2) follows from $t_! \dashv t^*$. Point (3) follows from the adjunction $L_{\geq 0} t_! \dashv t^* I_{\geq 0}$ together with (1). Now (4) holds, since $i$ sends $J_{\geq 0}$-covering families to $J$-covering families.
	
	For the last remaining point, put $\theta \coloneqq L_{\geq 0}t_!I$, and observe that the inclusions $h \colon \Aff_\CCC \to \St_\CCC$ and $h_{\geq 0} \colon \Aff_{\CCC_{\geq 0}}\to \St_{\CCC_{\geq 0}}$ have left adjoints $\Gamma \dashv h$ and $\Gamma_{\geq 0} \dashv h_{\geq 0}$, since $\CCC$ is presentable. We need to show that $\theta h = h_{\geq 0} t$.
	
	Since $h_{\geq 0}$ is fully faithful, it suffices to show that $\Gamma_{\geq 0} \theta h = t$. So let $S \in \Aff_\CCC$ and $T \in \Aff_{\CCC_{\geq 0}}$ be given. Then it holds
	\begin{align*}
		\Aff_{\CCC_{\geq 0}}(\Gamma_{\geq 0}\theta h S,T) & \simeq \St_{\CCC_{\geq 0}}(\theta h S, h_{\geq 0} T) \\
		& \simeq \St_\CCC(hS, t^*h_{\geq 0}T) \\
		& \simeq \St_\CCC(hS,hiT) \\
		& \simeq \Aff_\CCC(S,iT) \simeq \Aff_{\CCC_{\geq 0}}(tS,T)
	\end{align*}
	where we have used (in order) the adjunction $\Gamma_{\geq 0} \dashv h_{\geq 0}$, point (3), Lemma \ref{Lem:it}, fully faithfulness of $h$, and $t \dashv i$. 
\end{proof}
Define the functors

\begin{equation}
\label{Eq:functors}
\begin{tikzcd}
		t \coloneqq L_{\geq 0}t_!I,  && i \coloneqq t^* \simeq i_!,&& \rho \coloneqq t_* \simeq i^*.
\end{tikzcd}
\end{equation}

 By Lemma \ref{Lem:it} and Proposition \ref{Prop:StAdj}, the use of $i$ and of $t$ here is unproblematic. Notice that $\rho$ is given by restricting along $\Aff_{\CCC_{\geq 0}} \to \Aff_{\CCC}$. We can then summarize Proposition \ref{Prop:StAdj} in the diagram	
\begin{center}
	\begin{tikzcd}
		{\St_{\CCC}} \arrow[rr, bend right, swap, pos=0.55, "{\rho}"] \arrow[rr, bend left, "t"] && {\St_{\CCC_{\geq 0}}}, \arrow[ll, hook', swap, "i"] 
	\end{tikzcd}
\end{center}
where $t \dashv i \dashv \rho$. In particular, the adjunction $t \dashv i$ satisfies the requirements laid down in the beginning of this subsection. Since ${\St_{\CCC_{\geq 0}}}$ is cocomplete, we deduce the following immediate consequence.

\begin{Thm}
	\label{Thm:Stacks}
	The category ${\St_{\CCC_{\geq 0}}}$ is the smallest full subcategory of ${\St_{\CCC}}$ which contains ${\Aff_{\CCC_{\geq 0}}}$ and is closed under colimits.
\end{Thm}

For $X \in \St_{\CCC}$, we put $\St_X \coloneqq (\St_{\CCC})_{/X}$, and $\St_{X_{\geq 0}} \coloneqq (\St_{\CCC_{\geq 0}})_{/\rho X}$. This has the pleasant feature that if $X= iX'$, then $\St_{X_{\geq 0}} \simeq (\St_{\CCC_{\geq 0}})_{/X'}$.

\begin{Not}
	When convenient, we consider $\St_{\CCC_{\geq 0}}$ as full subcategory of $\St_\CCC$ via the functor $i$.
\end{Not}

\subsection*{Quasi-coherent modules and algebras}

\begin{Rem}
	Let $F\colon \Aff_{\CCC}^\op \to \Cat$ be given, and write $\widehat{F}$ for the right Kan extension of $F$ along $\Aff_{\CCC}^\op \to \St_{\CCC}^\op$. Then the right Kan extension, along $\Aff_{\CCC_{\geq 0}}^{\op} \to \St_{\CCC_{\geq 0}}^{\op}$, of the restriction of $F$ to $\Aff_{\CCC_{\geq 0}}^\op \to \Aff_{\CCC}^\op \to \Cat$ is equivalent to the composition of $\widehat{F}$ with $i \colon \St_{\CCC_{\geq 0}} \to \St_\CCC$. Indeed, for any $X \in {\St_{\CCC_{\geq 0}}}$, the functor
	\[
		i \colon (\Aff_{\CCC_{\geq 0}})_{/X} \to (\Aff_{\CCC})_{/iX}
	\]
	on the comma categories is cofinal, by \cite[Thm.\ 4.1.3.1]{LurieHTT} and using the adjunction $t \dashv i$. 
\end{Rem}

Define functors 
\begin{align*}
		\QCoh(-) : \St_{\CCC}^\op \to \Cat, && \QAlg(-) : \St_{\CCC}^\op \to \Cat, 
\end{align*}
by right Kan extension of the functors $\Mod_{(-)}$ and $\DAlg_{(-)}$ along the inclusion $\Aff_{\CCC}^\op \to \St_{\CCC}^\op$. Objects in $\QCoh(X)$ (resp.\ in $\QAlg(X)$) are called \textit{quasi-coherent modules} (resp.\ \textit{quasi-coherent algebras}) on $X$.

Note that the functor $\Mod_{(-)}$ lands in the category of stable, presentable symmetric monoidal categories, with left adjoint functors between them. This shows that $\QCoh(X)$ is stable and presentable, and that it carries a natural symmetric monoidal structure. Therefore, for $f\colon X \to Y$ in ${\St_{\CCC}}$, we have an adjunction
\begin{center}
    \begin{tikzcd}
        \QCoh(Y) \arrow[r, shift left, "f^*"] & \QCoh(X). \arrow[l, shift left, "f_*"]
    \end{tikzcd}
\end{center}
For $\MM \in \QCoh(X)$ and $f\colon\Spec^\nc A \to X$, we write the restriction of $\MM$ to $\Spec^\nc A$ as $\MM_A$ or $\MM_f$.
\begin{Lem}\label{Lem:DefOfTstr}
	For any $X$ in $\St_{\CCC_{\geq 0}}$, the category $\QCoh(X)$ inherits a t-structure from the categories $\Mod_A$, such that $\MM \in \QCoh(X)$ is connective if and only if $\MM_{A} \in \Mod_A$ is connective for all $\Spec A \to X$. 
\end{Lem}
\begin{proof}
	Write $\QCoh(X)_{\geq 0}$ for the full subcategory spanned by $\MM \in \QCoh(X)$ such that $\MM_A$ is connective, for all $\Spec A \to X$. Then $\QCoh(X)_{\geq 0}$ is closed under extensions, and $F\colon\QCoh(X)_{\geq 0} \to \QCoh(X)$ is a fully faithful left adjoint. The claim now follows by applying \cite[Prop.\ 1.2.1.16]{LurieHA} to the localization $L\colon\QCoh(X)^\op \to \QCoh(X)^\op$ induced by $F$.
\end{proof} 
From hereon endow $\QCoh(X)$ with the t-structure induced by the preceding lemma, for any $X \in \St_{\CCC_{\geq 0}}$.
\begin{Cor}
	\label{Cor:PullRightt}
	For any $f\colon X \to Y$ in $\St_{\CCC_{\geq 0}}$, the functor $f^* \colon \QCoh(Y) \to \QCoh(X)$ is right t-exact.
\end{Cor}
\begin{proof}
	For $\MM \in \QCoh(Y)_{\geq 0}$ and $g\colon\Spec  A \to X$, it holds $(f^*\MM)_{g} = \MM_{fg}$, which is connective.
\end{proof}

Likewise, the category $\QAlg(X)$ is presentable and has a symmetric monoidal structure, and for $f\colon X \to Y$ we have an adjunction
\begin{center}
    \begin{tikzcd}
        \QAlg(Y) \arrow[r, shift left, "f^*"] & \QAlg(X). \arrow[l, shift left, "f_*"]
    \end{tikzcd}
\end{center}

\begin{Rem}
	For $X$ of the form $\Spec^\nc(A)$, we have $\QCoh(X) \simeq \Mod_A$ and $\QAlg(X) \simeq \DAlg_A$. Moreover, if $A$ is connective, then the t-structures on $\QCoh(X)$ and on $\Mod_A$ coincide.
\end{Rem}

\begin{Rem}
	The adjunction $\LSym_A \dashv U_A \colon \Mod_A \rightleftarrows \DAlg_A$ globalizes to an adjunction 
    \begin{center}
    \begin{tikzcd}
        \QCoh(X) \arrow[r, shift left, "\LSym_X"] & \QAlg(X) \arrow[l, shift left, "U_X"]
    \end{tikzcd}
    \end{center}
	for any $X \in {\St_{\CCC}}$, where $U_X$ is the forgetful functor.
\end{Rem}

\subsection*{Descent}
Let $F\colon \Aff_{\CCC}^\op \to \Space$ be given. Write $L\colon\PPP({\Aff_{\CCC}}) \to {\St_{\CCC}}$ for the stackification functor, and $h\colon {\Aff_{\CCC}} \to \PPP({\Aff_{\CCC}})$ for the Yoneda embedding. Define the functors
\begin{align*}
		\widehat{F}_0\colon\PPP({\Aff_{\CCC}})^\op \to \Space, && \widehat{F}\colon\St_{\CCC}^\op \to \Space
\end{align*}
via Kan extension along $h$ and along $L h$, respectively. Since the topology on $\Aff_\CCC$ is subcanonical by assumption, $Lh$ is just the restriction of $h$ to $\Aff_\CCC \to \St_\CCC$.

\begin{Lem}
	\label{Lem:hatF0}
	The object $F \in \PPP({\Aff_{\CCC}})$ represents the functor $\widehat{F}_0$. Consequently, $\widehat{F}_0$ sends colimits in $\PPP({\Aff_{\CCC}})$ to limits in $\Space$. 
\end{Lem}	

\begin{proof}
	Let $X \in \PPP({\Aff_{\CCC}})$ be given. Observe that
	\[
		\PPP({\Aff_{\CCC}})(X,F) \simeq \lim_{\Spec^\nc A \to X} \PPP({\Aff_{\CCC}})(\Spec^\nc A, X) \simeq \widehat{F}_0(X)
	\]
	where the last equivalence follows from the formula for right Kan extensions.
\end{proof}

\begin{Prop}
	\label{Prop:KanExtDesc}
	The functor $F$ is a stack if and only if $\widehat{F}$ sends colimits in ${\St_{\CCC}}$ to limits in $\Space$, in which case $F$ represents $\widehat{F}$.
\end{Prop}

\begin{proof}
	If $F$ is a stack, then with the same argument as in Lemma \ref{Lem:hatF0}, it holds that $F$ represents $\widehat{F}$, since also every stack is a colimit of representables.
	
	Conversely, suppose that $\widehat{F}$ sends colimits to limits. By \cite[Prop.\ 5.5.2.2]{LurieHTT}, it follows that $\widehat{F}$ is representable, say by $X \in {\St_{\CCC}}$. We will show that $X \simeq F$ in $\PPP({\Aff_{\CCC}})$. To this end, observe that $\widehat{F} L \simeq \widehat{F}_0$, since $L$ preserves colimits, and colimit-preserving functors out of $\PPP({\Aff_{\CCC}})$ are determined by their restriction to ${\Aff_{\CCC}}$. Then for $Y \in \PPP({\Aff_{\CCC}})$, it holds
	\[
	\PPP({\Aff_{\CCC}}) (Y,X) \simeq  \widehat{F}(L Y) \simeq \widehat{F}_0(Y) \simeq \PPP({\Aff_{\CCC}}) (Y,F),
	\]
	where the first equivalence comes from the fact that $L$ is left adjoint to the inclusion $ {\St_{\CCC}} \subset \PPP({\Aff_{\CCC}})$, and the last equivalence follows from Lemma \ref{Lem:hatF0}. The claim follows.
\end{proof}

In the following two results fix a complete category $\EEE$.  We verify that \cite[Lem.\ I.1.3.2]{GaitsgoryStudy} also holds in our setting.

\begin{Lem}
	\label{Lem:EvaluedSt}
	A functor $F\colon\Aff_{\CCC}^\op \to \EEE$ is an $\EEE$-valued stack if and only if the composition
	\[
		\Aff_{\CCC}^\op \xrightarrow{F} \EEE \xrightarrow{h} \PPP(\EEE) \xrightarrow{\ev_x} \Space
	\]
	is an $\Space$-valued stack for each $x \in \EEE$. Here, $h$ denotes Yoneda, and $\ev_x$ denotes the evaluation-at-$x$ functor.
\end{Lem}
\begin{proof}
	This follows from the fact that the Yoneda embedding preserves limits, and that \[(\ev_x)_{x \in \EEE} \colon \PPP(\EEE) \to \prod_{x \in \EEE} \Space \] is conservative.
\end{proof}

\begin{Prop}
	\label{Prop:KanExtEDesc}
	Let $F \colon \Aff_{\CCC}^\op \to \EEE$ be given, with right Kan extension $\widehat{F}\colon\St_{\CCC}^\op \to \EEE$. Then $\widehat{F}$ sends colimits in ${\St_{\CCC}}$ to limits in $\EEE$ if and only if $F$ is an $\EEE$-valued stack.
\end{Prop}

\begin{proof}
	Combine Lemma \ref{Lem:EvaluedSt} with Proposition \ref{Prop:KanExtDesc}.
\end{proof}

\begin{Rem}
	Set-theoretically, the Kan extension $\widehat{F}$ from Proposition \ref{Prop:KanExtDesc} lands in $\Space_2$. Likewise, in Proposition \ref{Prop:KanExtEDesc}, the category $\EEE$ should be at least $2$-small, and $1$-complete. In practice, the functor $F$ will land in $\Cat_1$, while $\widehat{F}$ will land in $\Cat_2$. However, by Propostion \ref{Prop:UniversesCont} and with the same argument as the proof of Lemma \ref{Lem:EvaluedSt}, the inclusion $i\colon\Cat_1 \to \Cat_2$ preserves limits. Therefore, $F$ is a $\Cat_1$-valued stack if and only if $iF$ is a $\Cat_2$-valued stack.  
\end{Rem}

\begin{Cor}
	\label{Cor:QCohCont}
	The functors $\QCoh(-)$ and $\DAlg_{(-)}$ send colimits in ${\St_{\CCC}}$ to limits in $\Cat$.
\end{Cor}
\begin{proof}
	By definition of a geometric context, the functor
	\[
		\Mod_{(-)} \colon \Aff_{\CCC}^\op \to \Cat
	\]
	is a $\Cat$-valued stack. By Proposition \ref{Prop:KanExtEDesc}, it suffices to show that also
	\[
		\DAlg_{(-)} \colon \Aff_{\CCC}^\op \to \Cat
	\]
	is a stack. 
	
	Let $\CCC = \lim \CCC_\alpha$ be a limit in $\Cat$. Then for $x,y \in \CCC$, it holds that $\CCC(x,y) = \lim \CCC_\alpha(x_\alpha,y_\alpha)$, where $x_\alpha,y_\alpha$ denote the image of $x,y$ under the projection $\CCC \to \CCC_\alpha$. It follows that colimits in $\CCC$ are computed pointwise, i.e., the colimit of a diagram $i \mapsto x_i \in \CCC$ is the object $x \in \CCC$ such that $x_\alpha = \colim(x_i)_\alpha$ in $\CCC_\alpha$, for each $\alpha$. 
	
	Now the claim on $\DAlg_{(-)}$ follows by applying the Barr--Beck--Lurie Theorem to the adjunction 
    \begin{center}
    \begin{tikzcd}
        \Mod_A \arrow[r, shift left, "F"] & \lim_\alpha  \DAlg_{A_\alpha} \arrow[l, shift left, "G"]
    \end{tikzcd}
    \end{center}
    for any covering family $\{\Spec^\nc A_\alpha \to \Spec^\nc A\}$, where $F$ is the composition of $\LSym_A$ with $\DAlg_A \to \lim_\alpha \DAlg_{A_\alpha}$, and $G$ is induced by the forgetful functors $\DAlg_{A_\alpha} \to \Mod_{A_\alpha}$ and the equivalence $\Mod_A \simeq \lim_\alpha \Mod_{A_\alpha}$.
\end{proof}	

\subsection*{Affineness}
A quasi-coherent algebra on a $\CCC_{\geq 0}$-stack is \textit{connective} if the underlying quasi-coherent module is. 

Let $\BB \in \QAlg(X)$ be given, for $X \in \St_{\CCC}$. Then we define the \textit{nonconnective relative spectrum} of $\BB$ as the $\CCC$-stack $\Spec^\nc\BB$ that sends $\Spec^\nc A \in {\Aff_{\CCC}}$ to the space of pairs $(f,\varphi)$, where $f\colon \Spec^\nc A \to X$ is a map of $\CCC$-stacks and $\varphi\colon f^*\BB \to A$ a map of $\CCC$-algebras. Observe that $\Spec^\nc\BB$ naturally lives over $X$. We can write $\Spec^\nc_X(\BB)$ to emphasize the base $X$.

\begin{Lem}
	\label{Lem:SpecBC}
	The nonconnective relative spectrum commutes with base change.
\end{Lem}

\begin{proof}
	Immediate from the definition. 
\end{proof}
A (nonconnective) $\CCC$-stack is \textit{(nonconnectively) affine} if it is equivalent to a (nonconnective) affine $\CCC$-scheme. We relativize these notions as follows. 
\begin{Def}
	\label{Def:Affine}
	A morphism $f\colon X \to Y$ of $\CCC$-stacks is called \textit{nonconnectively affine} if it is of the form $\Spec^\nc(\BB) \to Y$ for some $\BB \in \QAlg(Y)$. If moreover $f$ lives in ${\St_{\CCC_{\geq 0}}}$ and $\BB$ is connective, then $f$ is called \textit{affine}.
\end{Def}
\begin{Lem}
	\label{Lem:AffineBC}
	Nonconnectively affine morphisms satisfy arbitrary base-change. That is, for any Cartesian diagram
	\begin{center}
		\begin{tikzcd}
			{X'} \arrow[d, "{f'}"] \arrow[r, "{g_2}"] & X \arrow[d, "f"] \\
			{Y'} \arrow[r, "g_1"] & Y
		\end{tikzcd}
	\end{center}
	where $f$ is nonconnectively affine, it holds that the natural transformation
	\[
		g_1^*f_* \to f'_*g_2^*
	\]
	of functors $\QCoh(X) \to \QCoh(Y')$ is an equivalence.
\end{Lem}
\begin{proof}
	The statement is clear when $Y$ is nonconnectively affine. Indeed, whether $g_1^*f_*(\MM) \to f_*'g_2^*(\MM)$ is an equivalence can be checked affine-locally on $Y'$. Since $f$ is nonconnectively affine, this then reduces to the case where $X,Y,X',Y'$ are all nonconnectively affine. Now the proof of \cite[Prop.\ 1.3.6]{DrinfeldFiniteness} goes through verbatim in the general case. 
\end{proof}

\begin{Prop}
	\label{Prop:AffLoc}
	Let $f \colon X \to Y$ in ${\St_{\CCC}}$ be given. Then $f$ is nonconnectively affine if and only if $X_A$ is nonconnectively affine for all $\Spec^\nc A \to Y$. Likewise, if $f$ lives in ${\St_{\CCC_{\geq 0}}}$, then $f$ is affine if and only if $X_A$ is affine for all $\Spec A \to Y$.
\end{Prop}
\begin{proof}
	The forward direction of the first point follows from stability of $\Spec^\nc(-)$ under base change, Lemma \ref{Lem:SpecBC}. Conversely, suppose that $X_A$ is nonconnectively affine, for all $\Spec^\nc A \to Y$. Consider the natural map $h\colon X \to \Spec^\nc f_* \OO_X$ over $Y$. We claim that $h$ is an equivalence. 
	
	The statement can be checked affine-locally on $Y$, so take $g\colon \Spec^\nc A \to Y$, and write $f'\colon X_A \to \Spec^\nc A$ for the projection. Then pulling back $h$ along $g$ gives
	\[
		h_A\colon X_A \to \Spec^\nc (g^*f_*\OO_X) \simeq \Spec^\nc (f'_* \OO_{X_A})
	\]
	where the second equivalence follows from the proof of Lemma \ref{Lem:AffineBC}, which only uses that $X_B$ is nonconnectively affine, for any $\Spec^\nc B \to Y$. Since $X_A$ is nonconnectively affine, $h_A$ is indeed an equivalence.
	
	The second point follows from the first, using that $f^*$ is right t-exact for the forward direction, and that any map $\Spec^\nc A \to iY$ uniquely factors through $\Spec^\nc(\iota \tau_{\geq 0} A)$ for the other direction.
\end{proof}

\begin{Cor}
	\label{Cor:RelSpec}
	The functor $\Spec^\nc(-) \colon  \QAlg(X)^\op \to \St_X$ is fully faithful, with essential image the category of $\CCC$-stacks over $X$ which are nonconnectively affine over $X$.
\end{Cor}

\begin{proof}
	By Proposition \ref{Prop:AffLoc}, this follows from descent, Corollary \ref{Cor:QCohCont}.
\end{proof}

For $X \in \St_{\CCC_{\geq 0}}$, we write $\QAlg(X)_{\geq 0}$ for the full subcategory of $\QAlg(X)$ spanned by connective objects. 

\begin{Rem}
	By descent, the adjunction from Proposition \ref{Prop:iotatau} globalizes to an adjunction
    \begin{center}
    \begin{tikzcd}
        \QAlg(X)_{\geq 0} \arrow[r, shift left, "\iota"] & \QAlg(X) \arrow[l, shift left, "\tau_{\geq 0}"]
    \end{tikzcd}
    \end{center}
for any $X \in \St_{\CCC_{\geq 0}}$. We then have the \emph{relative spectrum} functor
\[ \Spec(-) \colon  \QAlg(X)_{\geq 0}^\op \to \St_{X_{\geq 0}} \]
defined similarly as the nonconnective version. For $\BBB \in \QAlg(X)_{\geq 0}$ it holds $\Spec \BBB \simeq \Spec^\nc \iota \BBB$. A morphism $Y' \to Y$ in $\St_{\CCC_{\geq 0}}$ is affine if and only if it is of the form $\Spec \BBB \to Y$ for some $\BBB \in \DAlg_{Y_{\geq 0}}$, and  $\Spec(-)$ commutes with base-change along morphisms in $\St_{\CCC_{\geq 0}}$.
\end{Rem}

\subsection*{Closed immersions \& pseudocomplements}
A morphism $M \to N$ in $\CCC$ is called \textit{surjective} if the fiber is connective, and $A \to B$ in $\DAlg_{\CCC}$ is \emph{surjective} if it is so after forgetting to $\CCC$.

A morphism $X \to Y$ of $\CCC_{\geq 0}$-stacks is a \textit{closed immersion} if it is affine, and for all $\Spec A \to Y$ with $A$ connective, the map $A \to B $ is surjective, where $X_A \simeq \Spec B$.

\begin{Def}
	Let $X \to Y$ be a morphism ${\St_{\CCC}}$. Define the \textit{pseudocomplement} $Y \setminus X$ of $X$ in $Y$ as the prestack on ${\Aff_{\CCC}}$ which sends $T$ to the full subspace of $Y(T)$ spanned by those morphisms $T \to Y$ for which 
	\[
	T \times_Y X \simeq \emptyset.
	\]
\end{Def}

Let $\EEE$ be a presentable category, and $X \in \EEE$. Recall that the full subcategory $\mathrm{Sub}(X)$ of $\EEE_{/X}$ spanned by monomorphisms $X' \to X$ is equivalent to a 1-category (in fact a poset), by \cite[Prop.\ 6.2.1.4]{LurieHTT}. An object of $\mathrm{Sub}(X)$ for $X \in {\St_{\CCC}}$ is called a \textit{subobject} (in the category $ {\St_{\CCC}}$). 

\begin{Warn}
	At the moment, there is no good notion of the underlying topological space of a given $\CCC$-stack (nor $\CCC_{\geq 0}$-stack) in general, hence also no good notion of an open substack. Even so, a subobject $X' \to X$ should not be thought of as an open immersion. In the algebraic case, for example, it is not true that any monomorphism is locally of finite presentation. In extension, the pseudocomplement should not be thought of as the open complement in general. On the other hand, it seems reasonable to expect that a good definition of underlying topological spaces entails that the pseudocomplement of a closed immersion between connective stacks is an open immersion.
\end{Warn}	

\begin{Prop}
	\label{Prop:PseudoC}
	The prestack $Y \setminus X$ is a subobject of $Y$. In fact, it is the largest subobject of $Y$ such that $(Y \setminus X) \times_Y X = \emptyset$.
\end{Prop}

\begin{proof}
	Since the space $(Y \setminus X)(T)$ is a full subspace of $Y(T)$ by construction for any $T \in \Aff_{\CCC}$,  the morphism $Y \setminus X \to Y$ is a monomorphism, hence the first claim follows. 
	
	Let $\{U_\alpha \to U\}$ be a covering family in ${\Aff_{\CCC}}$. Then $U$ is the colimit of the \v{C}ech nerve of $\bigsqcup U_\alpha \to U$ by \cite[Lem.\ 6.2.3.19]{LurieHTT}. It follows that $U \times_Y X = \emptyset$ if and only if $U_\alpha \times_Y X = \emptyset$ for all $\alpha$. Thus $Y \setminus X$ is a stack, since $Y$ is.
	
	It is clear that $(Y \setminus X) \times_Y X = \emptyset$, and that any other subobject $W \subset X$ such that $W \times_YX = \emptyset$ factors through $Y \setminus X$.
\end{proof}

\begin{Rem}
	\label{Rem:highbrow}
	Here is a more highbrow argument for Proposition \ref{Prop:PseudoC}, which also provides another useful perspective.  Factorize $f \colon X \to Y$ through its image $\im (f) \to Y$, which is a monomorphism. Now observe that $\mathrm{Sub}(Y)$ is a locale, hence a Heyting algebra, hence is pseudocomplemented, meaning that for any object $U$ in $\mathrm{Sub}(Y)$ (considered as lattice) there exists a largest $U^* \in \mathrm{Sub}(Y)$ such that $U \wedge U^* = 0$ holds in $\mathrm{Sub}(Y)$ \cite[\S 7]{BlythLattice}. Since the meet $(-) \wedge (-)$ in $\mathrm{Sub}(Y)$ corresponds to $(-) \times_Y (-)$, it holds that $\im(f)^* \simeq Y \setminus X$.
\end{Rem}

\begin{Exm}
	Let $Z \to X$ be a closed immersion of schemes in the sense of ordinary derived algebraic geometry. Then $X \setminus Z$ is the open complement of $Z$ in $X$. If $U \to X$ is an open immersion of schemes, then $X \setminus (X \setminus U) \simeq U$. In general, for any quasi-compact morphism $f\colon X \to Y$ of schemes, with $Z$ the scheme-theoretic image of $f_\cl \colon  X_\cl \to Y_\cl$, it holds that $Y \setminus X$ is the open subscheme of $Y$ supported on $Y_\cl \setminus Z_\cl$, by Remark \ref{Rem:highbrow} and \cite[\href{https://stacks.math.columbia.edu/tag/01R8}{Tag 01R8}]{stacks-project}.
\end{Exm}

\subsection*{Vanishing loci and affine subobjects}
Let $Y \in \St_{\CCC}$ and $\MM \in \QCoh(Y)$ be given.

\begin{Def}
	The \textit{vanishing locus} of $\MM$ is the subobject $V(\MM) \to Y$ such that $V(\MM)(T)$ is the space of morphism $g \colon  T \to Y$ for which $g^*\MM \simeq 0$.
\end{Def}

\begin{Exm}
	Write $\Spec^\nc(\LSym_Y \MM) \setminus \{0\}$ for the pseudocomplement of the zero section $Y \to \Spec^\nc(\LSym_Y \MM)$. Then the vanishing locus of $\MM$ is equivalent to the pseudocomplement of $\Spec^\nc(\LSym_Y \MM) \setminus \{0\} \to Y$.
\end{Exm}

Let $A \in \DAlg$. Since the tensor product in $\Mod_A$ commutes with colimits in each variable separately, the symmetric monoidal category $\Mod_A$ is closed. We write $\Map_A(-,-)$ for the internal mapping object. Recall that $M \in \Mod_A$ is \textit{perfect} if the natural map
\[ M \otimes_A \Map_A(M,A) \to \Map_A(M,M) \]
is an equivalence \cite[Def.~1.2.3.6]{ToenHAGII}.

\begin{Def}
	A quasi-coherent module $\MM$ on $Y \in \St_{\CCC}$ is \textit{perfect} if $\MM_C \in \Mod_C$ is perfect for all $\Spec^\nc C \to Y$.
\end{Def}

\begin{Prop}
	\label{Prop:VanLocAff}
	If $\MM$ is perfect, then the vanishing locus $V(\MM) \to Y$ is nonconnectively affine.
\end{Prop}

\begin{proof}
	We may assume $Y = \Spec^\nc A$ by Proposition \ref{Prop:AffLoc}. The forgetful functor $\Theta \colon  \DAlg_A \to \CAlg_{\Theta A}$ has a left adjoint, written $\Theta_L$, by \cite[Prop.~4.2.27]{RaksitHKR}. By \cite[Prop.~1.2.10.1]{ToenHAGII}, there is an $A_\MM \in \CAlg_{\Theta A}$ such that, for all $C \in \CAlg_{\Theta A}$, it holds 
	\[ \CAlg_{\Theta A}(A_\MM,C) = \begin{cases}
		* & \text{if } \MM_C = 0, \\
		\emptyset & \text{otherwise}.
	\end{cases} \]
	Now since $\DAlg_A(\Theta_LA_\MM,R) \simeq \CAlg_{\Theta A}(A_\MM,\Theta R)$ for all $R \in \DAlg_A$, it follows that $\Spec^\nc (\Theta_LA_\MM)$ is the vanishing locus of $\MM$.
\end{proof}

\begin{Cor}
	Let $f\colon X \to Y$ be a nonconnectively affine morphism of $\CCC$-stacks such that $f_*\OO_X$ is perfect. Then the pseudocomplement $Y \setminus X$ is nonconnectively affine over $Y$.
\end{Cor}

\begin{proof}
	By Lemma \ref{Lem:AffineBC}, the pseudocomplement of $f$ is the vanishing locus of $f_* \OO_X$. Hence, the statement follows from Proposition \ref{Prop:VanLocAff}.
\end{proof}

\subsection*{Equivariant $\CCC$-geometry}
\label{Par:Equivariant_C_geometry}
Let $\M$ be a commutative monoid, and $A \in {\DAlg_{\CCC}}$. Consider $A$ as object in $\DAlg_\CCC^\M \coloneqq \DAlg_{\CCC^\M}$ by endowing it with the trivial grading, and put $\DAlg^\M_A \coloneqq (\DAlg_{\CCC^\M})_{A/}$. Let then 
\[
	\QAlg^\M(-) \colon  \St_{\CCC}^\op \to \Cat
\]
be the right Kan extension of the functor $\Aff_{\CCC}^\op \to \Cat$ that sends $\Spec^\nc A$ to $\DAlg^\M_A$. 

Observe that the unique morphism of algebraic contexts $\Mod_\Z \to \CCC$ sends $\Z[t,t^{-1}]$ to $\mathbbm{1}_\CCC[\Z]$. Put $\G_{m,\CCC} \coloneqq \Spec (\mathbbm{1}_\CCC[\Z])$, and observe that the group structure on $\G_m$ induces a group structure on $\G_{m,\CCC}$. 

For $X \in {\St_{\CCC}}$, write $\Aff^\nc(X)^{\G_{m,\CCC}}$ for the category of nonconnectively affine morphisms $Y \to X$, together with a $\G_{m,\CCC}$-action on $Y$, such that $Y \to X$ is $\G_{m,\CCC}$-equivariant when $X$ is endowed with the trivial $\G_{m,\CCC}$-action.

\begin{Prop}
	\label{Prop:ZgradeGm}
	We have an equivalence $\QAlg^\Z(X)^\op \simeq \Aff^\nc(X)^{\G_{m,\CCC}}$. 
\end{Prop}

\begin{proof}
	This follows from Proposition \ref{Prop:GradvsAct}.
\end{proof}

\begin{Rem}
	The preceding discussion can also be carried out, \textit{mutatis mutandis}, in the connective case.
\end{Rem}

Let $G$ be a group object in ${\St_{\CCC}}$. We use the theory of group objects in topoi as expounded in \cite{NikolausPrincipal}. Let $ {\St_{\CCC}^G}$ be the category of $\CCC$-stacks with a $G$-action, with $G$-equivariant morphisms between them.

We write $BG$ for the colimit of the simplicial diagram in ${\St_{\CCC}}$ that encodes the group structure of $G$.  Then for $X \in  {\St_{\CCC}^G}$, we write $[X/G]$ for the colimit of the simplicial diagram which encodes the $G$-action on $X$.

Observe, for $f\colon  X \to Y$ in $ {\St_{\CCC}^G}$, it holds that $f$ is the pullback of $\bar{f} \colon  [X/G] \to [Y/G]$ along the projection map $Y \to [Y/G]$. Moreover, any map $T \to [X/G]$ is of the form $[P/G] \to [X/G]$ for some $G$-equivariant $P \to X$.

Note that, if $X,G \in {\St_{\CCC_{\geq 0}}}$, then so is $[X/G]$, since $i$ commutes with colimits.

\begin{Not}
	For $M \in \CCC^\Z$ and $d \in \Z$, we define the \textit{twist} of $M$ by $d$ as the $\Z$-graded $\CCC$-module $M(d)$ such that $M(d)_n = M_{n+d}$ for all $n \in \Z$. 
\end{Not}

\subsection*{Projective spectra}
Consider the map $j\colon 0 \to \N$ of commutative monoids. Recall from Proposition \ref{Prop:EquivAdj} that we have an adjunction
\begin{center}
    \begin{tikzcd}
        \CCC  \arrow[r, shift left, "j_!"] &  \CCC^\N. \arrow[l, shift left, "j^!"]
    \end{tikzcd}
\end{center}
Observe that $j_!$ sends $M \in \CCC$ to the graded object $M(0)$ that is $M$ given degree $0$, and $j^!$ sends $K \in \CCC^\N$ to $K_0$. Using this description, we see that $j^!$ is also left adjoint to $j_!$. 
\begin{Lem}
	The functor $j^!\colon  \CCC^\N \to \CCC$ is a morphism of algebraic contexts.
\end{Lem}
\begin{proof}
	It suffices to show that $j^!$ is symmetric monoidal. This follows from the fact that
	\[
	j^!(K \otimes K') = (K \otimes K')_0 = \bigoplus_{b+c=0}(K_b \otimes K'_c) = K_0 \otimes K'_0
	\]
	by the Day convolution formula.
\end{proof}
We thus get an adjunction
\begin{center}
    \begin{tikzcd}
        \DAlg_{\CCC}^{\N} \arrow[r, shift left, "j^!"] & \DAlg_{\CCC}. \arrow[l, shift left, "j_!"]
    \end{tikzcd}
\end{center}
Observe, for $K \in \CCC^\N$, the counit $\epsilon$ of $j_! \dashv j^!$ is the inclusion $K_0 \to K$, while to unit $\eta$ of $j^! \dashv j_!$ is the projection $K \to K_0$. It follows that $\eta \circ \epsilon \simeq \id$, and hence also the composition
\[
B_0 = j_!j^! B \to B \to j_!j^!B = B_0
\]
is invertible, for any $B \in \DAlg_{\CCC}^{\N}$. From hereon, fix an $\N$-graded $\CCC$-algebra $B$. We will define the projective spectrum of $B$.

\begin{Def}
	Write $B_+$ for the graded object such that $(B_+)_0 = 0$ and $(B_+)_n = B_n$ for $n >0$. 
\end{Def}	

\begin{Lem}
	The sequence $B_+ \to B \to B_0$ is a fiber sequence.
\end{Lem}

\begin{proof}
	Since limits in $\CCC^\N$ are computed degree-wise, it suffices to show that $(B_+)_n \to B_n \to (B_0)_n$ is a fiber sequence, for all $n \in \N$. The latter is clear.
\end{proof}

Endow $\Spec^\nc B$ with the $\G_{m,\CCC}$-action induced from Proposition \ref{Prop:ZgradeGm}, where we consider $B$ as $\Z$-graded by putting zero in negative degrees. Formally, one does this via the fully faithful map $k_!\colon \DAlg_{\CCC}^{\N} \to \DAlg_{\CCC}^{\Z}$ induced by $k\colon \N \to \Z$.

\begin{Def}
	Write $\Spec^\nc B_0 \to \Spec^\nc B$ as $V(B_+) \to \Spec^\nc B$. 
\end{Def}

\begin{Prop}
	The $\G_{m,\CCC}$-action on $\Spec^\nc B$ restricts to an action on $(\Spec^\nc B) \setminus V(B_+)$.
\end{Prop}

\begin{proof}
	Since $f\colon (\Spec^\nc B) \setminus V(B_+) \to \Spec^\nc B$ is a monomorphism, it suffices to show that the composition
	\[
		\G_{m,\CCC} \times ((\Spec^\nc B) \setminus V(B_+)) \xrightarrow{\id \times f} \G_{m,\CCC} \times \Spec^\nc B \xrightarrow{\sigma} \Spec^\nc B
	\]
	factors through $f$, where $\sigma$ is the action map. By definition of pseudocomplements, it suffices to show that the outer square in the diagram
	\begin{center}
\begin{tikzcd}
			\emptyset \arrow[r] \arrow[d] & \G_{m,\CCC} \times V(B_+) \arrow[r] \arrow[d, "h"] & V(B_+) \arrow[d] \\
		\G_{m,\CCC} \times ((\Spec^\nc B) \setminus V(B_+)) \arrow[r] & \G_{m,\CCC} \times \Spec^\nc B \arrow[r] & \Spec^\nc B
\end{tikzcd}
	\end{center}
	is Cartesian, where we take the square on the right to be Cartesian by definition. Observe, since $B \to B_0$ is homogeneous, the map $g\colon V(B_+) \to \Spec^\nc B$ is equivariant. It follows that $h$ is the map $\id \times g$, and hence the square on the left is indeed Cartesian.
\end{proof}

\begin{Def}
	We define the \textit{projective spectrum} of $B$ as the stack quotient
	\[
		\Proj B \coloneqq [((\Spec^\nc B) \setminus V(B_+)) / \G_{m,\CCC}]
	\]
	in ${\St_{\CCC}}$.
\end{Def}

\begin{Rem}
	The above discussion globalizes to a given $\BB \in \QAlg^\N(X)$, where $X \in {\St_{\CCC}}$, which produces the relative projective spectrum 
	\[
		\Proj \BB \coloneqq [((\Spec^\nc \BB) \setminus V(\BB_+)) / \G_{m,\CCC}]
	\]
	over $X$.
\end{Rem}

\begin{Lem}\mbox{}
	\label{Lem:Proj-BC-Con}
	\begin{enumerate}
		\item $\Proj(-)$ commutes with base-change.
		\item If $\BB \in \QAlg^\N(X)$ is connective, then $\Proj \BB \in {\St_{\CCC_{\geq 0}}}$.
	\end{enumerate}
\end{Lem}

\begin{proof}
	Since colimits are universal in any topos, taking the quotient stack commutes with base-change. This shows the first claim. 
	
	For the second claim, we reduce to the affine case by the previous point. Then $\Proj B$ is a colimit of $\CCC_{\geq 0}$-stacks, and the claim follows since ${\St_{\CCC_{\geq 0}}}$ is closed under colimits.
\end{proof}

\subsection*{The cotangent complex: global case}
For $X \in {\St_{\CCC}}$, let $\StMod_X$ be the category of pairs $(T,\MM)$, where $T \in \St_X$ and $\MM \in \QCoh(T)$. A morphism $(T',\MM') \to (T,\MM)$ is a map $f \colon  T' \to T$ over $X$ together with a homomorphism $\varphi\colon  f^*\MM \to \MM'$ of quasi-coherent $\OO_{T'}$-modules.

For $(T,\MM) \in \St\Mod_X$, write $\OO_T \oplus \MM$ for the quasi-coherent $\OO_T$-algebra such that $(\OO_T \oplus \MM)(A) \simeq A \oplus f^*\MM$, for any $f\colon  \Spec^\nc A \to T$. Here, $A \oplus f^*\MM$ is the square-zero extension of $A$ by $f^*\MM$. Put \[T[\MM] \coloneqq \Spec^\nc_T (\OO_T \oplus \MM).\] Note this construction is functorial in $(T,\MM)$.

\begin{Def}
	\label{Def:GlobCotangent}
	Let $f\colon  X \to Y$ in ${\St_{\CCC}}$ be given. Write $\Der_{X/Y}$ for the functor
	\[
		\StMod_Y^\op \to \Space \colon  (T,\MM) \mapsto \St_Y(T[\MM],X).
	\]
	If $\Der_{X/Y}$ is representable by an object of the form $(X,\LL_{X/Y})$ (for which $X \to Y$ is of course $f$), then $f$ \textit{admits a cotangent complex}, in which case $\LL_{X/Y} \in \QCoh(X)$ is called the \textit{cotangent complex of $f$}.
\end{Def}

\begin{Rem}
	Observe that $\StMod_X$ is a 2-small, presentable category. Hence, the functor $\Der_{X/Y}$ takes values in $\Space_2$. 
\end{Rem}

\begin{Prop}
	For $f\colon  X \to Y$ in ${\St_{\CCC}}$, the cotangent complex exists if and only if $\Der_{X/Y}$ sends pushouts in $\StMod_Y$ to pullbacks in $\Space$.
\end{Prop}

\begin{proof}
	By \cite[Prop.\ 5.5.2.2]{LurieHTT}, the functor $\Der_{X/Y}$ is representable if and only if it sends colimits in $\StMod_Y$ to limits in $\Space$. By \cite[Prop.\ 4.4.2.7]{LurieHTT}, the latter is the case if and only if $\Der_{X/Y}$ sends coproducts to products and pushouts to pullbacks. Hence, it suffices to show that $\Der_{X/Y}$ always sends coproducts to products.
	
	Let $\{(T_\alpha,\MM_\alpha)\}_\alpha$ in $\StMod_Y$ be given. Put $T \coloneqq \bigsqcup_\alpha T_\alpha$. Using Corollary \ref{Cor:QCohCont}, we identify $\QCoh(T)$ with $\prod_\alpha \QCoh(T_\alpha)$. Let then $\MM \in \QCoh(T)$ be the sequence $(\MM_\alpha)_\alpha$. We claim that the canonical map
	\[
		\bigsqcup T_\alpha[\MM_\alpha] \to T[\MM]
	\]
	is an equivalence.
	
	Using the universality of colimits and disjointness of coproducts from the Giraud axioms, it follows that each $T_\alpha \to T$ is a monomorphism. Now let $U \to T$ be given. Write $U_\alpha$ for the pullback of $U$ along $T_\alpha \to T$. Then the canonical map $\bigsqcup U_\alpha \to U$ is an equivalence. Again using that coproducts are disjoint, this gives us
	\begin{align*}
		\St_T\left(U, \bigsqcup_{\alpha'} T_{\alpha'}[\MM_{\alpha'}]\right) &\simeq \prod_{\alpha} \St_{T}\left(U_{\alpha},\bigsqcup_{\alpha'} T_{\alpha'}[\MM_{\alpha'}]\right) \\ &\simeq \prod_{\alpha} \St_{T_\alpha}(U_\alpha, T_\alpha[\MM_\alpha]) \\ &\simeq \St_T(U,T[\MM]),
	\end{align*}
	where the last equivalence follows from the universal property of $\Spec^\nc(-)$, and again the equivalence $\QCoh(T) \simeq \prod_\alpha \QCoh(T_\alpha)$.
\end{proof}
\begin{Cor}
	\label{Cor:CotExistAff}
	For $f\colon X \to Y$ in ${\St_{\CCC}}$ the following are equivalent:
	\begin{enumerate}
		\item The cotangent complex of $f$ exists.
		\item The functor $\Der_{X/Y}$ sends pushouts in $\St\Mod_Y$ of the form $(T,\MM_{12}) = (T,\MM_1) \sqcup_{(T,\MM)} (T,\MM_2)$ with identities on $T$, to pullbacks in $\Space$.
		\item The same as in (2), but for $T$ of the form $\Spec^\nc B$.
	\end{enumerate}
\end{Cor}
\begin{proof}
	Write a morphism $(S',\NN') \to (S,\NN)$ as $(h,\alpha)$, where $h\colon S' \to S$ is a morphism in $\St_X$, and $\alpha\colon  h^* \NN \to \NN'$ is a map in $\QCoh(S')$. Consider a pushout
	\begin{center}
		\begin{tikzcd}
			(T,\MM) \arrow[rr, "{(f_1,\varphi_1)}"] \arrow[dd, "{(f_2,\varphi_2)}"] && (T_1,\MM_1) \arrow[dd, "{(g_1,\psi_1)}"] \\ \\
			(T_2, \MM_2) \arrow[rr, "{(g_2,\psi_2)}"] && (T_{12}, \MM_{12})
		\end{tikzcd}
	\end{center}
	in $\StMod_Y$. Put $gf \coloneqq g_1 f_1$, which is equivalent to $g_2f_2$. Then it holds
	\[
		(T_{12},\MM_{12}) \simeq (T_{12},(g_1)_* \MM_1) \sqcup_{(T_{12},(gf)_* \MM)} (T_{12},(g_2)_* \MM_2)
	\]
	which is straightforward to check from the description of morphisms in $\St\Mod_Y$. This shows the equivalence of (1) and (2).
	
	Now let $(T,\MM) \in \StMod_Y$ be given. Observe
	\[
		T[\MM] \simeq \colim_{\Spec^\nc B \to T} \Spec^\nc(B \oplus \MM_B).
	\]
	Hence, for a pushout of the form $(T,\MM_{12}) = (T,\MM_1) \sqcup_{(T,\MM)} (T,\MM_2)$, with identities on $T$, we have
	\begin{align*}
		\Der_{X/Y}(T,\MM_{12}) \simeq \lim_{U \in (\Aff_{\CCC}^\op)_{/T}} \Der_{X/Y}((U,(\MM_1)_U) \sqcup_{(U, \MM_{U}) } (U,(\MM_2)_U))
	\end{align*}
	which shows that (2) and (3) are equivalent.
\end{proof}
We verify familiar behavior of the cotangent complex. To this end, let a commutative diagram
\begin{equation}
	\label{Eq:CotSquare}
	\begin{tikzcd}
		X' \arrow[r, "f"] \arrow[d] & X \arrow[d] \\
		Y' \arrow[r] & Y
	\end{tikzcd}
\end{equation}
in ${\St_{\CCC}}$ be given, such that $X' \to Y'$ and $X \to Y$ admit cotangent complexes. Write $\varphi \colon  \StMod_{Y'}^\op \to \StMod_Y^\op$ for the functor induced by composing with $Y' \to Y$. Then we have a natural transformation
\[ \eta_f \colon  \Der_{X'/Y'} \to \Der_{X/Y} \circ\varphi \]
of functors $\StMod^\op_{Y'} \to \Space$.

\begin{Prop}
	Given a square as in (\ref{Eq:CotSquare}), then $\eta_f$ induces a morpshim
	\[
		\alpha \colon  f^*\LL_{X/Y} \to \LL_{X'/Y'}
	\]
	in $\QCoh(X)$, which is invertible if the given square is Cartesian.
\end{Prop}
\begin{proof}
The map $\alpha$ is induced by applying $\eta_f$ to the identity on $(X',\LL_{X'/Y'})$. If the square is Cartesian, then $\eta_f$ is an equivalence, hence so is $\alpha$.
\end{proof}	

\begin{Prop}
	Let a sequence $X \xrightarrow{f} Y \xrightarrow{g} Z$ in ${\St_{\CCC}}$ be given, such that $g$ admits a cotangent complex. Then $f$ admits a cotangent complex if and only if $g f$ does. In this case, we have a natural exact sequence
		\[
			g^*\LL_{Y/Z} \to \LL_{X/Z} \to \LL_{X/Y}
		\]
	in $\QCoh(X)$.
\end{Prop}
\begin{proof}
	The idea is to consider the fibers of the obvious natural transformation $\Der_{X/Z} \to \Der_{Y/Z}$. We omit the details, since the argument is similar as the previous proof.
\end{proof}

\begin{Rem}  
	If $f\colon X \to Y$ is of the form $\Spec^\nc B \to \Spec^\nc A$, then $\LL_{X/Y}$ corresponds to $L_{B/A}$ under the equivalence $\QCoh(X) \simeq \Mod_B$, since $\LL_{X/Y}$ satisfies the same universal property as $L_{B/A}$, which determines the latter uniquely.
\end{Rem}

\begin{Prop}
	\label{Prop:CotExistCover}
	Let $X \to Y$ in ${\St_{\CCC}}$ be given. Suppose there is a diagram $\alpha \mapsto Y_\alpha$ in $\St_{Y}$ such that each $X \times_Y Y_\alpha \to Y_\alpha$ admits a cotangent complex, and such that the natural map $\colim Y_\alpha \to Y$ is invertible. Then $X \to Y$ admits a cotangent complex as well.
\end{Prop}
\begin{proof}
	Let $(T,\MM) \in \StMod_Y$ be given. Put $X_\alpha \coloneqq X \times_Y Y_\alpha$, and $T_\alpha \coloneqq T \times_Y Y_\alpha$. By Lemma \ref{Lem:SpecBC}, it holds
	\[
	T_\alpha[\MM_\alpha] \simeq T[\MM] \times_T T_\alpha
	\]
	where $\MM_\alpha$ is the pullback of $\MM$ along $T_\alpha \to T$. Moreover, since colimits are universal in ${\St_{\CCC}}$, it holds
	\begin{align*}
	X \simeq \colim X_\alpha, && T[\MM] \simeq \colim T_\alpha[\MM_\alpha], && T \simeq \colim T_\alpha.
	\end{align*}
	It follows that
	\[
	\St_{Y} (T[\MM],X) \simeq \lim_\alpha \St_{Y_\alpha}(T_\alpha[\MM_\alpha],X_\alpha).
	\]
	
	Now suppose $(T,\MM)$ is given as a pushout
	\[
	(T,\MM) \simeq (T,\MM_1) \sqcup_{(T,\MM_0)} (T,\MM_2)
	\]
	in $\StMod_Y$, where the maps on $T$ defining the pushout are the identity. Then it holds
	\[
	(T_\alpha,\MM_\alpha) \simeq (T_\alpha,(\MM_1)_\alpha) \sqcup_{(T_\alpha,(\MM_0)_\alpha)} (T_\alpha,(\MM_2)_\alpha).
	\]
	The claim now follows from Corollary \ref{Cor:CotExistAff}.
\end{proof}

\begin{Cor}
	\label{Cor:CotAffExist}
	Suppose that $X \to Y$ is nonconnectively affine. Then the cotangent complex $\LL_{X/Y}$ exists.
\end{Cor}

\begin{proof}
	Apply Proposition \ref{Prop:CotExistCover} to $\{\Spec^{\nc}B \to Y\}_{B \in \DAlg_\CCC}$.
\end{proof}

\section{Deformation spaces, normal bundles and blow-ups}
\label{Sec:Deformation_spaces_normal_bundles_and_blowups}
Let still $(\CCC,J)$ be a geometric context. Throughout, fix a morphism \[f\colon  X \to Y\] of $\CCC$-stacks.

\subsection*{Weil restrictions}
 Define the functors
\begin{align*}
		f^*\colon  \St_{Y} \to \St_{X}, && f^*_{\geq 0} \colon  \St_{Y_{\geq 0}} \to \St_{X_{\geq 0}} 
\end{align*}
by pulling back. Observe, if $f$ is of the form $if' \colon  iX'  \to iY' $ (see \ref{Eq:functors})
for 
\[f'\colon  X' \to Y'\] in ${\St_{\CCC_{\geq 0}}}$, then $f^*_{\geq 0}$ coincides with pulling back along $f'$. 

\begin{Prop}
	\label{Prop:Weil}
	If $f$ is nonconnectively affine, then $f^*$ has a right adjoint, written $\Res^\nc_f$. If moreover $f$ is in ${\St_{\CCC_{\geq 0}}}$ and is  affine, then $f^*_{\geq 0}$ has a right adjoint, written $\Res_f$, such that $\rho \circ \Res^\nc_f \circ i \simeq \Res_f$.
\end{Prop}

\begin{proof}
	For $Z \to X$ in ${\St_{\CCC}}$, let $\Res^\nc_fZ$ be the presheaf that sends $\Spec^\nc A$ to the space of pairs $(g,h)$, where $g\colon  \Spec^\nc A \to Y$ is a map of $\CCC$-stacks, and $h\colon X_A \to Z$ is a map in $\St_X$, where $X_A$ is the pullback of $\Spec^\nc A$ to $X$ along $f$. To see that $\Res^\nc_fZ$ is a $\CCC$-stack, one reduces to the case where $X,Y$ are both nonconnectively affine, which is straightforward. 
	
	For the second point, assume that $f\colon  X \to Y$ is in ${\St_{\CCC_{\geq 0}}}$. By the same reasoning as for $f^*$, we have a right adjoint $\Res_f$ to $f^*_{\geq 0}$. Now the equivalence 
	\[
		\rho \circ \Res^\nc_f \circ i \simeq \Res_f
	\]
	follows from the fact that $i$ commutes with pullbacks.
\end{proof}

\begin{Rem}
	Since colimits are universal in $\St_{\CCC}$, the Weil restriction along any morphism $f$ exists, but the result might not be valued in $\Space_1$ anymore. To avoid having to deal with this, we only consider the affine case, since it is the only case that we need.
\end{Rem}

\begin{Def}
	For $f \colon  X \to Y$ nonconnectively affine and $Z \in \St_X$, we call $\Res_f^\nc(Z) \in \St_Y$ the \emph{Weil restriction} of $Z$ along $f$.
\end{Def}

Let $Z \in \St_{X_{\geq 0}}$ be given. By Proposition \ref{Prop:Weil}, the stack $\Res_fZ$ can be computed as the restriction of $\Res_f^\nc(iZ ) \colon  {\Aff_{\CCC}} \to \Space$ to ${\Aff_{\CCC_{\geq 0}}}$. Our definition of Weil restrictions therefore recovers the preexisting one for the connective case---used in \cite{Weil}---in a straightforward way. In the proof of Theorem~\ref{Thm:Rees-affine} we will see another reason to take $\Res_f^\nc(-)$ as our definition of the Weil restriction, namely that Rees algebras cannot be expected to be connective in general, even for connective input. 

\begin{Rem}
	\label{Rem:GequivWeil}
	Let $G$ be an affine group object in ${\St_{\CCC_{\geq 0}}}$. Then Proposition \ref{Prop:Weil} also holds $G$-equivariantly, meaning that for $f \colon  X \to Y$ in $ {\St_{\CCC}^G}$ nonconnectively affine, the pullback functor $f^* \colon   \St^G_{Y} \to  \St^G_{X}$ has a right adjoint, written $f_*^G$.
	
	Write $\overline{f} \colon  [X/G] \to [Y/G]$. Then the $G$-equivariant Weil restriction along $f$ induces a Weil restriction along $\overline{f}$. That is, let $[Z/G] \to [X/G]$ be given. Then the diagram
	\begin{center}
		\begin{tikzcd}
			{f_*^GZ} \arrow[r] \arrow[d] & {[f_*^GZ/G]} \arrow[d] \\
			Y \arrow[r] & {[Y/G]}
		\end{tikzcd}
	\end{center}
	is Cartesian. It follows from the universal property of the Weil restrictions that $[f_*^GZ/G]$ is the Weil restriction of $[Z/G]$ along $\overline{f}$, which we again write as $\Res^{\nc}_{\overline{f}}[Z/G]$.
\end{Rem}

\subsection*{Deformation spaces}
Let $F \colon  \DAlg \to {\DAlg_{\CCC}}$ be the functor induced by the initiality of $\Mod_\Z$ among algebraic contexts. Then this induces a functor $F^\Z \colon  \DAlg^\Z \to \DAlg_\CCC^\Z$. Recall that $\LSym_{\Mod_\Z}(\Z(1)) $ is the graded ring $\Z[t^{-1}]$ where $t^{-1}$ has degree $-1$. Then it holds
\[
F^\Z(\Z[t^{-1}]) \simeq \LSym_\CCC(\mathbbm{1}_\CCC(1))
\]
which we simply write as $\Z[t^{-1}]$ if it is clear in which context we consider this object. Put then \[\A^1_\CCC \coloneqq \Spec(\LSym(\mathbbm{1}_\CCC(1))),\] and write $\A^1_Y \coloneqq Y \times \A^1_\CCC$ and $\G_{m,Y} \coloneqq Y \times \G_{m,\CCC}$ when convenient. 

Note that $\A^1_\CCC$ has the $\G_{m,\CCC}$-action of degree $-1$, i.e., the one corresponding to the $\Z$-grading on $\LSym_\CCC(\mathbbm{1}_\CCC(1))$. This is the same action as the one induced by the $\G_m$-action on $\A^1$ with degree $-1$ via the functor $F$.

The zero section $\{0\} \to \A^1_\CCC$ induces a map
\[
	\zeta\colon B\G_{m,\CCC} \to [\A^1_{\CCC}/\G_{m,\CCC}].
\]
Put $\zeta_Y \simeq \id_Y \times \zeta$, which is equivalent to the nonconnectively affine map
\[
	Y \times B\G_{m,\CCC} = [Y \times \{0\} / \G_{m,\CCC}] \stackrel{\zeta_Y}\longrightarrow [Y \times \A^1_\CCC/\G_{m,\CCC}]
\]
where we endow $Y$ with trivial $\G_{m,\CCC}$-action. By Remark \ref{Rem:GequivWeil}, we can thus Weil-restrict along $\zeta_Y$.

\begin{Def}
	\label{Def:DefSpace}
	The \textit{deformation space} of $X \to Y$ in ${\St_{\CCC}}$ is the Weil restriction of $X \times B\G_{m,\CCC} \to Y \times B\G_{m,\CCC}$ along $\zeta_Y$. It is written
	\[
		\cD_{X/Y} \to Y \times [\A^1_{\CCC}/\G_{m,\CCC}].
	\]
	We write $\sD_{X/Y}$ for the pullback of $\cD_{X/Y}$ along \[Y \times \A^1_\CCC \to Y \times [\A^1_\CCC/\G_{m,\CCC}]\] and endow $\sD_{X/Y}$ with the $\G_{m,\CCC}$-action such that $[\sD_{X/Y}/\G_{m,\CCC}] \simeq \cD_{X/Y}$.
\end{Def}

Observe that $\sD_{X/Y}$ is also the Weil restriction of $X \times\{0\} \to Y \times \{0\}$ along $Y \times \{0\} \to Y \times \A^1_\CCC$. If $X \to Y$ is of the form $\Spec^{\nc} B \to \Spec^{\nc} A$, then put $\sD_{B/A} \coloneqq\sD_{X/Y}$.

\subsection*{Rees algebras}\label{Par:Rees_algebras} The goal of this subsection is to show the following.
\begin{Thm}
	\label{Thm:Rees-affine}
	If $X \to Y$ lives in ${\St_{\CCC_{\geq 0}}}$ and is affine, then $\sD_{X/Y} \to Y \times \A^1_{\CCC}$ is nonconnectively affine. Consequently, the same is true for $\cD_{X/Y} \to Y \times [\A^1_{\CCC}/\G_{m,\CCC}]$.
\end{Thm}

As before, we write $\Z[x] \coloneqq \LSym_\CCC(\mathbbm{1}_\CCC)$. Let $R \in {\DAlg_{\CCC}}$. Then we call a morphism 
\[
	f\colon  \Z[x] \to R
\]
an \textit{element} of $R$, written $f \in R$. By tensoring the morphism $\times x \colon  \Z[x] \to \Z[x]$ with $R$ along $f$, we obtain the \textit{multiplication map} \[\times f \colon  R \to R.\] We write $R/(f)$ for the pushout of $f$ along $\Z[x] \to \Z \colon  x \mapsto 0$, hence 
\[
	R/(f) \coloneqq R \otimes_{\Z[x]} \Z.
\]

\begin{Lem}
	Let $B \in \DAlg_A$ and $f \in B$. Then there is a fiber sequence
	\[
		B \xrightarrow{\times f} B \to B/(f)
	\]
	in $\Mod_B$. 
\end{Lem}

\begin{proof}
	Consider the exact sequence \[ \Z[x] \xrightarrow{\times x} \Z[x] \to \Z \] in $\CCC$. Consider $B$ as $\Z[x]$-module via the structure map $\Z[x] \to B$ determined by $f$. Then upon tensoring $\Z[x] \xrightarrow{\times x} \Z[x] \to \Z$ with $B$ over $\Z[x]$, we obtain the required fiber sequence.
\end{proof}

\begin{Exm}
\label{Ex:ReesofSym}
	Consider a map $A \to B = \LSym_A(M)$ in ${\DAlg_{\CCC_{\geq 0}}}$ for some connective $A$-module $M$. Then by rotating the fiber sequence \[M[t^{-1}] \xrightarrow{\times t^{-1}} M[t^{-1}] \to M\] we obtain a pushout square
	\begin{center}
	\begin{tikzcd}
		{M[t^{-1}] [-1]} \arrow[r] \arrow[d, swap, "\times {t^{-1}[-1]}"] & 0 \arrow[d] \\
		{M[t^{-1}][-1]} \arrow[r] & M[-1].
	\end{tikzcd}
	\end{center}
	We claim that $\LSym_{A[t^{-1}]}(M[-1])$ corepresents $\sD_{B/A}$. Indeed, for any $Q \in \DAlg_{A[t^{-1}]}$, we have a fiber sequence
	\[
		Q \xrightarrow{t^{-1}} Q \to Q/(t^{-1})
	\]
	in $\Mod_{A}$, and thus $Q/(t^{-1})$ is the fiber of $\times t^{-1}[1] \colon Q[1] \to Q[1]$. It follows that
	\begin{align*}
		\DAlg_A(B,Q/(t^{-1})) &\simeq \Mod_A(M,Q/(t^{-1})) \\
		& \simeq \Mod_A(M[-1], Q) \times_{\Mod_A(M[-1],Q)} \Mod_A(M[-1],0) \\
		& \simeq \Mod_{A[t^{-1}]} (M[t^{-1}][-1],Q) \times_{\Mod_{A[t^{-1}]} (M[t^{-1}][-1],Q)} 
		\{0\} \\
		&\simeq \Mod_{A[t^{-1}]}(M[-1],Q) \\
		& \simeq \DAlg_{A[t^{-1}]}(\LSym_{A[t^{-1}]}(M[-1]),Q)
	\end{align*}
	where the left factor in the fiber product on the second line is induced by multiplication with $t^{-1}$ on $Q$, and likewise on the third line by multiplication with $t^{-1}$ on $A[t^{-1}][-1]$. The claim follows by the universal property of $D_{B/A}$ as Weil restriction.	
\end{Exm}

\begin{proof}[Proof of Theorem \ref{Thm:Rees-affine}]
	By Proposition \ref{Prop:AffLoc}, we may assume that $X,Y$ are affine. Say $X =\Spec B$ and $Y= \Spec A$. We will define a colimit-preserving functor
	\[
		R^\ext_{(-)/A} \colon  \DAlg_{A_{\geq 0}} \to \DAlg_{A[t^{-1}]}
	\]
	and show that $\Spec R^\ext_{B/A}$ has the desired universal property, for any $B \in \DAlg_{A_{\geq 0}}$.
	
	Let $B = \LSym_A(A \otimes M)$ with $M \in \CCC^0$ be given. Consider the functor
	\begin{align*}
		H_B \colon  \DAlg_{A[t^{-1}]} &\to \Space \\
		Q & \mapsto \DAlg_A(B,Q/(t^{-1})).
	\end{align*}
	We claim that $H_B$ is corepresentable. By \cite[Prop.\ 5.5.2.7]{LurieHTT}, it suffices to show that $H_B$ preserves limits and $\kappa$-filtered colimits for some regular cardinal $\kappa$. Observe, since $B \in \DAlg_{A^0}$ as defined in Lemma \ref{Lem:ZeroDef}, it is compact in $\DAlg_{A_{\geq 0}}$. Since $A,B$ are connective, and since $\tau_{\geq 0}$ is right-adjoint to the inclusion ${\DAlg_{\CCC_{\geq 0}}} \to {\DAlg_{\CCC}}$ and preserves filtered colimits, $B$ is also compact in $\DAlg_{A}$. It follows that $H_B$ preserves filtered colimits, since $Q \mapsto Q/(t^{-1})$ preserves filtered colimits. Therefore, it suffices to show that $H_B$ preserves limits. 
	
	Observe that $Q/(t^{-1})$ is the fiber of $\times t^{-1}[1] \colon Q[1] \to Q[1]$ in $\Mod_A$. It follows that
	\[
		H_B(Q) \simeq \Mod_A(M,Q/(t^{-1})) \simeq \Mod_A(M,Q[1]) \times_{\Mod_A(M,Q[1])} \{0\}.
	\]
	Now $\Mod_A(M,(-)[1]) \colon  \DAlg_{A[t^{-1}]} \to \Space$ is the composition of the forgetful functor $\DAlg_{A[t^{-1}]} \to \Mod_{A[t^{-1}]}$, followed by the shift functor $(-)[1]$ and the functor $\Mod_A(M,-)$. Since these are all limit preserving, it follows that $H_B$ is a limit of limit-preserving functors, hence itself limit-preserving.
	
	Write $R_{B/A}^\ext \in \DAlg_{A[t^{-1}]}$ for a corepresenting object of $H_B$, for any $B \in \DAlg_{A^0}$. Then left-Kan extend  \[R^\ext_{(-)/A}\colon \DAlg_{A^0} \to \DAlg_{A[t^{-1}]}\] to a colimit-preserving functor
	\[
		R^\ext_{(-)/A}\colon  \DAlg_{A_{\geq 0}} \to \DAlg_{A[t^{-1}]}
	\]
	using Proposition \ref{Prop:DACsum}.
	
	Let now $B \in \DAlg_{A_{\geq 0}}$ be given such that $X \to Y$ is $\Spec B \to \Spec A$. Write $B = \colim_{\alpha} B_\alpha$ for $B_\alpha \in \DAlg_{A^0}$. Then for $Q \in \DAlg_{A[t^{-1}]}$, it holds
	\begin{align*}
		\DAlg_{A[t^{-1}]}(R^\ext_{B/A},Q) &\simeq \lim_{\alpha} \DAlg_{A[t^{-1}]}(R^\ext_{B_\alpha/A}, Q) \\
		&\simeq \lim_{\alpha} \DAlg_A(B_\alpha, Q/(t^{-1})) \\
		& \simeq \DAlg_A(B,Q/(t^{-1})).
	\end{align*}
	Geometrically, this means that
	\[
		\St_{\A^1_\CCC \times Y}(T,\Spec R_{B/A}^\ext) \simeq \St_Y(T \times_{\A^1_\C} \{0\}, X)
	\]
	for any $T \in \St_{\A^1_\CCC \times Y}$, since stacks are colimits of representables. Therefore $\Spec R_{B/A}^\ext$ satisfies the universal property of $D_{X/Y}$ as defined via Weil restriction, which was to be shown.
\end{proof}

\begin{Def}
	\label{Def:Rees}Suppose that $X \to Y$ lives in ${\St_{\CCC_{\geq 0}}}$ and is affine. The \textit{extended Rees algebra} of $X \to Y$ is the nonconnective, quasi-coherent algebra over $Y \times\A^1_\CCC $ which corepresents $\sD_{X/Y}$, written $\RR^\ext_{X/Y}$. We endow $\RR^\ext_{X/Y}$ with the $\Z$-grading which corresponds to the $\G_{m,\CCC}$-action on $\sD_{X/Y}$. We let $\RR_{X/Y}$ be the part of $\RR^{\ext}_{X/Y}$ with non-negative homogeneous degrees, and call it the \emph{Rees algebra}. If $X \to Y$ is of the form $\Spec B \to \Spec A$, we put $R_{B/A}^\ext \coloneqq \RR_{X/Y}^\ext$ and $R_{B/A} \coloneqq \RR_{X/Y}$.
\end{Def}

\begin{Rem}
	Let $A \to B$ in ${\DAlg_{\CCC_{\geq 0}}}$ be given. Taking the $\Z$-grading into account, the Rees algebra satisfies the following universal property
	\[
		\DAlg^\Z_{A[t^{-1}]}(R_{B/A}^\ext, Q) \simeq \DAlg_A(B,(Q/(t^{-1}))_0)
	\]
	for any $Q \in \DAlg^\Z_{A[t^{-1}]}$, where $(-)_0$ refers to taking the homogeneous-degree zero part.
\end{Rem}

\begin{Prop}
	\label{Prop:ReesBC}
	The extended Rees algebra is stable under base-change. That is, for $X \to Y$ affine in ${\St_{\CCC_{\geq 0}}}$  and $f\colon Y' \to Y$ in ${\St_{\CCC_{\geq 0}}}$  it holds that \[(f,\id)^*\RR_{X/Y}^\ext \simeq \RR_{X'/Y'}^\ext\] where $X' = X \times_Y Y'$ and $(f,\id)$ is the obvious map $Y'\times \A^1_\CCC   \to Y \times \A^1_\CCC  $.  Likewise for the Rees algebra: \[(f,\id)^*\RR_{X/Y} \simeq \RR_{X'/Y'}.\]
\end{Prop}
\begin{proof}
	This follows from the fact that Weil restrictions commute with base-change.
\end{proof}

\begin{Exm}
	Consider $\CCC = \Mod_\Z$, i.e., the algebraic context of derived algebraic geometry. Let $A \in \Alg_{\geq 0}$ be given. We use multi-index notation, writing $\underline{x}$ for $x_1,\dots,x_n$ etc. By Example \ref{Ex:ReesofSym}, we know that
	\[
	R^\ext_{A[\underline{x}]/A} = \LSym_{A[t^{-1}]}(A[-1]^{\oplus n}).
	\]
	
	Let $\underline{f} = f_1,\dots,f_n \in A$ be given, and put $B = A/(\underline{f})$. Let us recover the known formula for $R_{B/A}^\ext$ from \cite{HekkingGraded}. Since $B $  is the pushout $A \otimes_{A[\underline{x}]} A$ of $\underline{f} \colon A[\underline{x}] \to A$ along the map $x_i \mapsto 0$, we have a pushout diagram
	\begin{center}
		\begin{tikzcd}
			\LSym_{A[t^{-1}]}(A[-1]^{\oplus n}) \arrow[r, "\varphi"] \arrow[d,"\psi"] \arrow[r] & A[t^{-1}] \arrow[d] \\
			A[t^{-1}]\arrow[r] & R_{B/A}^\ext
		\end{tikzcd}
	\end{center}
	where $\psi$ is induced by $0\colon A[-1]^{\oplus n} \to A[t^{-1}]$. For the definition of $\varphi$, consider the diagram
	\begin{center}
		\begin{tikzcd}
			A[t^{-1}][-1]^{\oplus n} \arrow[r] \arrow[d] & 0 \arrow[d] \\
			A[t^{-1}][-1]^{\oplus n} \arrow[r] \arrow[d] & A[-1]^{\oplus n} \arrow[r] \arrow[d] & 0 \arrow[d] \\
			0 \arrow[r] & A[t^{-1}]^{\oplus n} \arrow[r] \arrow[dr, swap, "\underline{f}"] & A[t^{-1}]^{\oplus n} \\
			&& A[t^{-1}]
		\end{tikzcd}
	\end{center}
	where all the square are pushout squares in $\Mod_{A[t^{-1}]}$. The map $\varphi$ is then induced by the map $A[-1]^{\oplus n} \to A[t^{-1}]$ given by composition in this diagram. Using the bottom right square in the above diagram, we can thus rewrite the pushout diagram defining $R_{B/A}^\ext$ as
	\begin{center}
		\begin{tikzcd}
			\LSym_{A[t^{-1}]}(A[-1]^{\oplus n}) \arrow[r] \arrow[d] & \LSym_{A[t^{-1}]}(A[t^{-1}]^{\oplus n}) \arrow[r, "{\underline{v} \mapsto \underline{f}}"] \arrow[d, "{\underline{v} \mapsto t^{-1}\underline{v}}"] & A[t^{-1}] \arrow[d] \\
			A[t^{-1}] \arrow[r] & \LSym_{A[t^{-1}]}(A[t^{-1}]^{\oplus n}) \arrow[r] & R^\ext_{B/A}
		\end{tikzcd}
	\end{center}
	where $\underline{v}$ are the coordinates of \[\LSym_{A[t^{-1}]}(A[t^{-1}]^{\oplus n}) = A[t^{-1},\underline{v}].\] We thus recover
	\[
	R^\ext_{B/A} \simeq \frac{A[t^{-1},\underline{v}]}{(t^{-1}\underline{v} - \underline{f})}
	\]
	as expected.
\end{Exm}

Recall that $M \to N$ in $\CCC$ is defined to be surjective if the fiber is connective. For $k \in {\DAlg_{\CCC}}$, we call $R \in \DAlg^\Z_{k}$ \textit{generated by degree 1 over $k$} if $\LSym_k(R_1) \to R$ is surjective.

\begin{Prop}
	\label{Prop:Rgendeg1}
	Let $A \to B$ be in ${\DAlg_{\CCC_{\geq 0}}}$. Then $R_{B/A}^\ext \in \DAlg^\Z_{A[t^{-1}]}$ is generated by degree $1$ over $A[t^{-1}]$.
\end{Prop}

\begin{proof}
	By the proof of Theorem \ref{Thm:Rees-affine}, it suffices to show that this condition is stable under sifted colimits, and that $R_{B/A}^\ext$ satisfies this condition for any $B \in \DAlg_{A^0}$. 
	
	To show that the condition is stable under sifted colimits, observe that the forgetful functor $\DAlg^\Z_{A[t^{-1}]} \to \Mod^\Z_{A[t^{-1}]}$ preserves sifted colimits, and likewise $\LSym_\CCC(-)$. We thus reduce to showing that, for a given map of sifted diagrams
	\[
		f_\alpha \colon  M_\alpha \to N_\alpha
	\]
	in $\Mod^\Z_{A[t^{-1}]}$, with colimit $f \colon  M \to N$ and such that each $f_\alpha$ is surjective, the map $f$ is surjective as well. Let $F$ be the fiber of $f$, and $F_\alpha$ the fiber of $f_\alpha$. Then $F[1]$ is the cofiber of $f$, and $F[1] \simeq \colim F_\alpha[1]$, since colimits commute with cofibers. Since each $F_\alpha[1]$ is 1-connective, so is $F[1]$.
	
	From Example \ref{Ex:ReesofSym}, it follows that $R_{B/A}^\ext \in \DAlg^\Z_{A[t^{-1}]}$ is generated by degree $1$ over $A[t^{-1}]$, for any $B \in \DAlg_{A^0}$.
\end{proof}

\subsection*{Deformation to the normal bundle}
 Let still $f\colon X \to Y$ in ${\St_{\CCC}}$, and suppose $f$ admits a cotangent complex. Let $\NN_{X/Y}$ be the shifted cotangent complex $\LL_{X/Y}[-1]$, and put 
\[
	\sN_{X/Y} \coloneqq \Spec^\nc (\LSym_X (\NN_{X/Y}(-1))).
\]
\begin{Prop}
	\label{Prop:TNXY}
	Endow $X,Y$ with trivial $\G_{m,\CCC}$-action. Then for any $\G_{m,\CCC}$-equivariant $h \colon T \to Y$, we have a natural equivalence
	\[
		\St^{\G_{m,\CCC}}_Y(T,\sN_{X/Y}) \simeq \St^{\G_{m,\CCC}}_Y(T \times_{\A^1_\CCC} \{0\},X)
	\]
	where $T$ lives over $\A^1_\CCC$ as $T \to Y \to \{0\} \to \A^1_\CCC$, which is $\G_{m,\CCC}$-equivariant.
\end{Prop}
\begin{proof}
	It holds
	\begin{align*}
		\St^{\G_{m,\CCC}}_Y(T,\sN_{X/Y}) &\simeq \QCoh^\Z(T)(h^*\LL_{X/Y},\OO_T[1](1)) \\ &\simeq \St^{\G_{m,\CCC}}_Y(T[\OO_T[1](1)],X).
	\end{align*}
	Note $\OO_T[1](1) \simeq (\OO_T/(0))(1)$, and so $T[\OO_T[1](1)]$ is the $\G_{m,\CCC}$-equivariant fiber product $T \times_{\A^1_\CCC} \{0\}$.
\end{proof}

\begin{Thm}
	\label{Thm:Deformation}
	We have a $\G_{m,\CCC}$-equivariant Cartesian diagram
	\begin{center}
		\begin{tikzcd}
			X \times \{0\}\ar[r]\ar[d] & X\times \A^1_\CCC\ar[d] & X\times \G_{m,\CCC}\ar[l]\ar[d]\\
			\sN_{X/Y}\ar[r]\ar[d] &\sD_{X/Y}\ar[d] & Y\times \G_{m,\CCC}\ar[l]\ar[d]\\
			Y \times \{0\}\ar[r] & Y\times \A^1_\CCC & Y\times \G_{m,\CCC}.\ar[l]		
		\end{tikzcd}	
	\end{center}
\end{Thm}

\begin{proof}
	By Proposition \ref{Prop:TNXY}, it remains to show that the bottom right square is Cartesian. Let $T \to Y \times \G_{m,\CCC}$ be given. Then
\begin{align*}		
		\St_{\G_{m,Y}}(T,\sD_{X/Y} \times_{\A^1_Y} \G_{m,Y}) &\simeq \St_{\A^1_Y}(T,\sD_{X/Y}) \simeq \St_Y(T \times_{\A^1_\CCC} \{0\},X)	
\end{align*}
	by definition of $\sD_{X/Y}$. Now observe that
	\[
		\G_{m,\CCC} \times_{\A^1_\CCC} \{0\} \simeq \Spec (F(0)) \simeq \Spec(0)
	\]
	for $F\colon \Mod_\Z \to \CCC$ the unique morphism of derived algebraic contexts. It follows that \[\St_Y(T \times_{\A^1_\CCC} \{0\},X) \simeq \{*\}\] which shows our claim.
\end{proof}

\begin{Cor}
	\label{Cor:G_A(R)=B}
	For any morphism $A \to B$ of $\CCC_{\geq 0}$-algebras, the map \[(R_{B/A}^\ext)_0 \to (R_{B/A}^\ext/(t^{-1}))_0\] is naturally equivalent to $A \to B$.
\end{Cor}

\begin{proof}
	Consider the map of $A$-algebras $h\colon B \to (R_{B/A}^\ext/(t^{-1}))_0$ corresponding to $\id\colon  R_{B/A}^\ext \to R_{B/A}^\ext$ under the universal property of the extended Rees algebra. We claim that $h$ is an equivalence. From the deformation to the normal bundle, it follows that $h$ is equivalent to the composition
	\[
		B \to \LSym_B(\NN_{B/A}) \to (\LSym_B(\NN_{B/A}))_0
	\]
	which is equivalent to the identity.
\end{proof}

\subsection*{Virtual Cartier divisors}
Let still $X \to Y$ be in $\St_{\CCC}$.
\begin{Def}
	A \textit{virtual Cartier divisor} on $T \in \Aff_{\CCC}$ is a morphism $D \to T$ such that there is a Cartesian diagram
	\begin{center}
		\begin{tikzcd}
			D \arrow[r] \arrow[d] & T \arrow[d] \\
			B\G_{m,\CCC} \arrow[r] & {[\A^1_\CCC /\G_{m,\CCC}]}.
		\end{tikzcd}
	\end{center}
\end{Def}
Let $D \to T$ be a virtual Cartier divisor. Pulling back the corresponding classifying morphism $T \to [\A^1_\CCC/\G_{m,\CCC}]$ along $\A^1_\CCC \to [\A^1_\CCC/\G_{m,\CCC}]$ and using the description of effective epimorphisms as local surjections, we see that $D \to T$ is locally on $T$ of the form $\Spec^\nc A / (f) \to \Spec^\nc A $.
\begin{Def}
	A \textit{virtual Cartier divisor over $X \to Y$} is a commutative diagram of the form
	\begin{center}
		\begin{tikzcd}
			D \arrow[r] \arrow[d] & T \arrow[d] \\
			X \arrow[r] & Y
		\end{tikzcd}
	\end{center}
	such that $D \to T$ is a virtual Cartier divisor.
\end{Def}
If $X \to Y$ admits a cotangent complex, then put $\cN_{X/Y} \coloneqq [\sN_{X/Y}/\G_{m,\CCC}] $.
\begin{Prop}
	\label{Prop:VCD}
	Suppose that $X \to Y$ admits a cotangent complex. Then $\cN_{X/Y} \to \cD_{X/Y}$ is the universal virtual Cartier divisor over $X \to Y$, in the sense that $\cD_{X/Y}(T)$ is the space of virtual Cartier divisors $D \to T$ over $X \to Y$, and any such $D \to T$ is obtained as a pullback, along $\cN_{X/Y} \to \cD_{X/Y}$, of some map $T \to \cD_{X/Y}$.
\end{Prop}
\begin{proof}
	By the deformation to the normal bundle, this follows straight from the definitions. 
\end{proof}
\begin{Def}
	A virtual Cartier divisor $D \to T$ over $X \to Y$ is called \textit{strict} if for any commutative diagram
	\begin{equation}
		\label{Eq:svCd}
		\begin{tikzcd}
			D_S \arrow[d] \arrow[r] & S \arrow[d] \arrow[ddl] \\
			D \arrow[d] \arrow[r] & T \arrow[d] \\
			X \arrow[r] & Y
		\end{tikzcd}
	\end{equation}
	where the top square is Cartesian, it holds that $S = \emptyset$.
\end{Def}
One can show that $\cD_{X/Y}$ is not only functorial in $X$, but also in $Y$. We then get the following.
\begin{Prop}
	\label{Prop:strictVCD}
	The pseudocomplement of $\cD_{X/X} \to \cD_{X/Y}$ classifies strict virtual Cartier divisors over $X \to Y$.
\end{Prop}
\begin{proof}
	A commutative diagram 
	\begin{center}
		\begin{tikzcd}
			S \arrow[d] \arrow[r] & T \arrow[d] \\
			\cD_{X/X} \arrow[r] & \cD_{X/Y}
		\end{tikzcd}
	\end{center}
	corresponds to a diagram (\ref{Eq:svCd}) in the obvious way, from which the claim follows.
\end{proof}

\subsection*{Rees algebras of closed immersions}
For $R \in \DAlg_A$ and $f \in R$, the \textit{localization} of $R$ at $f$ is defined as the colimit
\[ R_f \coloneqq \colim \left( R \xrightarrow{ \times f} R \xrightarrow{\times f} R \xrightarrow{\times f} \cdots \right).  \]
For $M \in \Mod_R$ we put $M_f \coloneqq M \otimes_R R_f$.
\begin{Thm}
	\label{Thm:ReesClosed}
	Let $X \to Y$ be a closed immersion of $\CCC_{\geq 0}$-stacks. Then the extended Rees algebra $R_{X/Y}^\ext$ is connective.
\end{Thm}

\begin{proof}
	Since the extended Rees algebra is stable under base-change, and by definition of the t-structure on $\QCoh(Y)$ from Lemma \ref{Lem:DefOfTstr}, we may assume without loss of generality that $X \to Y$ corresponds to $A \to B$, with $A$ and $B$ connective. 
	
	Let $R$ be the extended Rees algebra $R_{B/A}^\ext$. Since $A \to B$ has connective fiber, the cotangent complex $L_{B/A}$ is 1-connective by Proposition \ref{Prop:CotangentofSur}. By the deformation to the normal bundle, we have
	\[
		R/(t^{-1}) \simeq \LSym_B (L_{B/A}[-1])
	\]
	and thus $R/(t^{-1})$ is connective. From the long exact sequence associated to the fiber sequence
	\[
		R \xrightarrow{\times t^{-1}} R \to R/(t^{-1})
	\]
	 it now follows that the maps $\times t^{-1} \colon  \pi_nR \to \pi_nR$ are isomorphisms, for all $n<-1$. Also from the deformation to the normal bundle, it follows that $R_{t^{-1}} \simeq A[t,t^{-1}]$, which is connective. Hence
	\[                                                                                                                                                  		\pi_n(R_{t^{-1}}) \simeq \colim \left(\pi_nR \xrightarrow{\times t^{-1}} \pi_n R \xrightarrow{\times t^{-1}} \pi_n R  \xrightarrow{\times t^{-1}} \cdots \right) \simeq \pi_nR
	\]
	vanishes for $n<-1$.
	
	For $n=-1$, recall that the connecting map in the long exact sequence associated to the above fiber sequence comes about from the map $R/(t^{-1}) \to R[1]$, which is homogeneous of degree $-1$. We thus have an exact sequence
	\[
		(\pi_0R)_1 \xrightarrow{\times t^{-1}} (\pi_0R)_0 \to (\pi_0(R/(t^{-1})))_0 \to (\pi_{-1}R)_1 \xrightarrow{\times t^{-1}} (\pi_{-1} R)_0 \to 0
	\]
	on the homogeneous pieces.	By Corollary \ref{Cor:G_A(R)=B}, it follows that the second map is equivalent to $\pi_0A \to \pi_0B$, which is surjective. Therefore
	\[
		(\pi_{-1}R)_1 \cong (\pi_{-1} R)_0 \cong \pi_{-1} A \cong 0
	\]
	and thus $R_1$ is connective. Now since $R$ is generated by degree $1$ by Proposition \ref{Prop:Rgendeg1}, it follows that $R$ is connective as well.
\end{proof}

\subsection*{Blow-ups}
Let $W \to X$ be an affine morphism of $\CCC_{\geq 0}$-stacks, with Rees algebra $\RR_{W/X}$.
\begin{Def}
	\label{Def:Blup}
	The \textit{blow-up} of $X$ in $W$ is the relative projective spectrum
	\[
		\Bl_WX \coloneqq \Proj \RR_{W/X}
	\]
	over $X$.
\end{Def}
\begin{Prop}
	We have the following fundamental properties.
	\begin{enumerate}
		\item The blow-up is stable under base-change.
		\item If $W \to X$ is a closed immersion, then $\Bl_WX$ lives in ${\St_{\CCC_{\geq 0}}}$.
	\end{enumerate}
\end{Prop}
\begin{proof}
	By Lemma \ref{Lem:Proj-BC-Con}, the first claim follows from Proposition \ref{Prop:ReesBC}, and the second  from Theorem \ref{Thm:ReesClosed}.
\end{proof}

For the remainder of this section, we will describe the functor of points of the blow-up of a closed immersion. To this end, we need an additional assumption, which in practice will always be met.
\begin{Ass}
	\label{Ass:Bcover}
	For any $\N$-graded, connective $\CCC_{\geq 0}$-algebra $R$ which is generated by degree $1$ over $R_0$, the map
	\[ \bigsqcup_{f \in R_1} \Spec R_f \to \Spec R \setminus \Spec R_0 \]
	is an effective epimorphism.
\end{Ass}
Our final main result is the following.
\begin{Thm}
	\label{Thm:BlFOPaff}
	Let $Z \to X$ be a closed immersion of $\CCC_{\geq 0}$-stacks. Assume that Assumption \ref{Ass:Bcover} holds.
	\begin{enumerate}
		\item\label{PI:1} There is a natural map $\Bl_ZX \to \cD_{Z/X}$ which exhibits $\Bl_ZX$ as the pseudocomplement of $\cD_{Z/Z} \to \cD_{Z/X}$.
		\item\label{PI:2} $\Bl_ZX$ classifies strict virtual Cartier divisors over $Z \to X$.
		\item\label{PI:3} Suppose that $\pi_*\OO_{\cD_{Z/Z}}$ is perfect, for $\pi\colon  \cD_{Z/Z} \to \cD_{Z/X}$ the natural map. Then $\Bl_ZX \to X \times [\A^1_{\CCC}/\G_{m,\CCC}]$ is nonconnectively affine.
	\end{enumerate}
\end{Thm}                       
We need some preliminary results.
\begin{Lem}
	\label{Lem:Loc}
	Let $R \in \CAlg_{\CCC^\heartsuit}^\N$ be given, and let $\varphi \colon  M \to N$ be a morphism of discrete, $\Z$-graded $R$-modules such that $M_{\geq 0} \to N_{\geq 0}$ is an equivalence. Then for any homogeneous element $f \in R$ of degree $1$, it holds that $\varphi_f \colon M _f \to N_f$ is an equivalence.
\end{Lem}
\begin{proof}
	Write $\varphi_f$ as the colimit of
	\begin{center}
		\begin{tikzcd}
			M \arrow[r, "\times f"] \arrow[d] & M(1) \arrow[r, "\times f"] \arrow[d] & M(2) \arrow[r, "\times f"] \arrow[d] & \cdots \arrow[r] & M_f \arrow[d, "\varphi_f"] \\
			 N \arrow[r, "\times f"]  & N(1) \arrow[r, "\times f"]  & N(2) \arrow[r, "\times f"] & \cdots \arrow[r] & N_f 
		\end{tikzcd}
	\end{center}
	where $P(k)$ is the twisted module with $P(k)_d = P_{d + k}$, for any $P$. By assumption, for any $k' \geq k$ it holds that $M(k')_{\geq -k} \to N(k')_{\geq -k}$ is an equivalence. Moreover, the natural map
	\[ \colim \left( M(k)_{\geq -k} \xrightarrow{\times f} M(k+1)_{\geq -k} \xrightarrow{\times f} \cdots \right) \to  (M_f)_{\geq -k}\]
	is an equivalence. Hence $(M_f)_{\geq -k} \to (N_f)_{\geq -k}$ is an equivalence for all $k$, from which the claim follows.
\end{proof}
Since $(\RR_{Z/X})_0 = \OO_X$, we have a natural map $X \to \Spec R_{Z/X}$. Applying Assumption \ref{Ass:Bcover} to the Rees algebra---allowed by Theorem \ref{Thm:ReesClosed}---it is straightforward to show that \[\sD_{Z/Z} \times_{\Spec \RR_{Z/X}} (\Spec \RR_{Z/X} \setminus X) = \emptyset\] We thus obtain a $\G_{m,\CCC}$-equivariant diagram
\begin{equation}
	\label{Eq:complement}
	\begin{tikzcd}
		\sD_{Z/X} \setminus\sD_{Z/Z} \arrow[r] \arrow[d] & \Spec \RR_{Z/X} \setminus X \arrow[d] \\
		\sD_{Z/X} \arrow[r] & \Spec \RR_{Z/X}.
	\end{tikzcd}
\end{equation}
\begin{Lem}
	\label{Lem:ComplCart}
	The square (\ref{Eq:complement}) is Cartesian.
\end{Lem}
\begin{proof}
	Since the question is local, we assume that $X$ is affine. We will show the stronger statement that $R_{W/X}^\ext \otimes_{R_{W/X}} \OO_X \simeq \OO_W[t^{-1}]$ for any affine scheme $W$ over $X$. By Lemma \ref{Lem:ZeroDef}, it suffices to do the case $\OO_W = \LSym_{\OO_X}(M)$ for some connective $\OO_X$-module $M$.
	
	Let $i \colon  \N \to \Z$ be the natural map, and recall that $R_{W/X} = i^! R^{\ext}_{W/X}$, where $i^!$ is the right adjoint to the morphism of algebraic contexts $i_! \colon  \CCC^{\N} \to \CCC^{\Z}$ (Proposition \ref{Prop:EquivAdj}). By Example \ref{Ex:ReesofSym}, it follows that $R_{W/X} \simeq \LSym_{\OO_X}(M[-1])$. Now rotating the exact sequence
	\[ M[t^{-1}] \xrightarrow{\times t^{-1}} M[t^{-1}] \to M \]
	gives us the exact sequence $M[-1] \to M[-1] \to M[t^{-1}]$, which shows that the square
	\begin{center}
		\begin{tikzcd}
			\LSym_{\OO_X}(M[-1]) \arrow[r] \arrow[d] & \LSym_{\OO_X[t^{-1}]}(M[-1]) \arrow[d] \\
			\LSym_{\OO_X}(0) \arrow[r] & \LSym_{\OO_X[t^{-1}]}(M[t^{-1}])
		\end{tikzcd}
	\end{center}
	is coCartesian. 
\end{proof}

\begin{proof}[Proof of Theorem \ref{Thm:BlFOPaff}]
	Lemma \ref{Lem:ComplCart} gives us a natural map
	\[ \Bl_ZX = [(\Spec \RR_{Z/X} \setminus X)/\G_{m,\CCC}] \to [ (\sD_{Z/X} \setminus\sD_{Z/Z} )/\G_{m,\CCC} ] \to \cD_{Z/X}.  \]
	For \itemref{PI:1}, it thus suffices to show that $\sD_{Z/X} \setminus\sD_{Z/Z} \to \Spec \RR_{Z/X} \setminus X$ is an equivalence, which indeed it is by Lemma \ref{Lem:Loc} and Assumption \ref{Ass:Bcover}. Then \itemref{PI:2} follows from \itemref{PI:1} by Proposition \ref{Prop:strictVCD}.
	
	By \itemref{PI:1} and Proposition \ref{Prop:VanLocAff}, it holds that $\Bl_ZX$ is nonconnectively affine over $\cD_{Z/X}$. Hence \itemref{PI:3} follows, since $\cD_{Z/X}$ is nonconnectively affine over $X \times [\A^1_\CCC/\G_{m,\CCC}]$ by Theorem \ref{Thm:Rees-affine}.
\end{proof}
\begin{Rem}
	The only point in the proof of Theorem \ref{Thm:BlFOPaff} where we use that $Z \to X$ is a closed immersion of $\CCC_{\geq 0}$-stacks is in the appeal to Assumption \ref{Ass:Bcover}. If the context $\CCC$ satisfies the stronger assumption that 
	\[ \bigsqcup_{f \in R_1} \Spec R_f \to \Spec R \setminus \Spec R_0 \]
	is an effective epimorphism for any (possibly nonconnective) $\N$-graded $R$, then Theorem \ref{Thm:BlFOPaff} holds for any affine morphism of $\CCC_{\geq 0}$-stacks. 
\end{Rem}

\section{Examples}
\label{Sec:Applications_and_examples}
We give our two main examples of a geometric context. 

\subsection*{Derived algebraic geometry}
The primary example of a geometric context is given by derived algebraic geometry. We spell out some well-known details as an illustration.

Let $\Mod^{\fgf}$ be the 1-category of finitely generated, free $\Z$-modules, put $\Mod_{\geq 0} \coloneqq \PPP_{\Sigma}(\Mod^{\fgf})$, and let $\Mod$ be the stabilization of $\Mod_{\geq 0}$. The inclusion $\Mod_{\geq 0} \to \Mod$ has a right adjoint $\tau_{\geq 0}$, and this determines a t-structure by \cite[Prop.\ 1.2.1.16]{LurieHA}, which is both left and right complete by \cite[Prop.\ 1.2.1.19]{LurieHA}.

The tensor product on $\Mod^{\fgf}$ induces a symmetric monoidal structure on $\Mod$ through Day-convolution and stabilization. The natural functor from $\Mod^{\fgf}$ to the $\infty$-category $\Sp$ of spectra extends to $\Mod_{\geq 0} \to \Sp$ by taking the left derived functor, and to $\Mod \to \Sp$ by completeness. Using \cite[\S 25]{LurieSpectral}, one shows that $\Mod$ (with the t-structure and symmetric monoidal structure just mentioned) is an algebraic context. 

We have that $\DAlg_{\geq 0}$ is the familiar $\infty$-category of simplicial commutative rings, and that $\Mod$ is equivalent to the $\infty$-category of $\Z^\circ$-modules in $\Sp$, where $(-)^\circ \colon  \DAlg_{\geq 0} \to \CAlg_{\geq 0}$ it the unique colimit-preserving functor which sends $\Z[x]$ to $\Z[x] \in \CAlg_{\geq 0}$. 

We use Example \ref{Ex:InducedContext} to enhance the \'{e}tale topology on $\Aff_{\Mod_{\geq0}}$ to a topology on $\Aff_{\Mod}$, also called the \emph{\'{e}tale topology}. Clearly, this gives a geometric context for which the $\infty$-category of $\Mod_{\geq 0}$-stacks recovers ordinary derived algebraic geometry. 

Proposition \ref{Prop:Weil} shows that, in the connective case, our definition of deformation spaces, hence of derived blow-ups, coincides with the one from \cite{Weil}. 

\subsection*{Derived analytic geometry}
Since the following example involves passing between model 1-categories and $\infty$-categories, we drop our convention that everything is implicitly $\infty$-categorical. We will use the script $\ttC$ for 1-categories, and $\Hom(-,-)$ for the mapping sets in 1-categories.

Let $R$ be a Banach ring. Recall that this means that $R$ is a ring endowed with a complete norm $\lvert - \rvert$ such that the underlying group is a normed group (with respect to $\lvert - \rvert$), and such that there is a constant $C$ for which 
\[ \lvert xy \rvert \leq C \lvert x \rvert \cdot \lvert y \rvert \]
holds for all $x,y \in R$. We call $R$ \emph{non-Archimedean} if 
\[ \lvert x + y \rvert \leq \max \{ \lvert x \rvert, \lvert y \rvert \} \]
holds for all $x,y\in R$. Else, $R$ is called \emph{Archimedean}, in which case $\lvert x+y \rvert \leq \lvert x \rvert + \lvert y \rvert$ holds by definition.

A \emph{Banach $R$-module} is a complete normed abelian group $V$ endowed with an $R$-module structure, such that there is a constant $C'$ for which 
\[ \lVert \lambda v \rVert \leq C' \lvert \lambda \rvert \cdot \lVert v \rVert  \]
holds for all $\lambda \in R, v \in V$. This gives us the $1$-category of Banach $R$-modules $\ttBan_R$, with morphisms bounded $R$-linear maps.
\begin{Def}\label{def:projective}
A Banach $R$-module $P$ is called \textit{projective} if for any strict epimorphism $f \colon V \to W$ of Banach $R$-modules the induced map $\Hom(P,V)\to \Hom(P,W)$ of sets is surjective. In this setting, we can define $f$ to be a \emph{strict epimorphism} if $V / \ker(f) \to W$ is an isomorphism \cite{KashiwaraCats}.
\end{Def}
 
 The \emph{separated completion} of a semi-normed space $N$ is the set of Cauchy-sequences in $N$ modulo the sequences which converge to $0$.
 
 The category of Banach modules over $R$ carries a symmetric monoidal structure with respect to the \textit{projective tensor product}---written $(-)\wotimes_{R}(-)$, and defined as the separated completion of the algebraic tensor product $V \otimes_R W$ with respect to the \emph{projective semi-norm} 
\[\lVert x \rVert \coloneqq \inf \left\{ \sum_{i=1}^{
n} \lVert v_i\rVert \lVert w_i\rVert  \Bigm\vert x= \sum_{i=1}^{
n} v_i  \otimes_{R} w_i \right\}
\]
for all $x \in V \otimes_R W$.

For $R, V, W$ non-Archimedean, we can use either the above, or the \textit{non-Archimedean projective tensor product} $(-)\wotimes^{\na}_{R}(-)$, which is the separated completion of $V \otimes_R W$ with respect to the \emph{non-Archimedean projective semi-norm}
\[\lVert x \rVert \coloneqq \inf \left\{ \max_{i=1, \dots,  n} \{ \lVert v_i\rVert \lVert w_i\rVert \} \Bigm\vert x= \sum_{i=1}^{
	n} v_i  \otimes_{R} w_i \right\}
\]
for all $x \in V \otimes_R W$.  

\begin{Rem}
	Any projective object of $\ttBan_R$ is flat with respect to $\wotimes_{R}$. The category of projective Banach $R$-modules is a symmetric monoidal subcategory of the category of Banach $R$-modules. 
\end{Rem} 

The category $\ttBan_R$  does not have countable products or coproducts. If we consider the category $\ttBanleq_R$ with the same objects as $\ttBan_R$, but with morphisms those maps in $\ttBan_R$ which are of norm less than or equal to $1$, then arbitrary products and coproducts do exist. This can be used to show that $\ttBan_R$ has enough projectives, meaning that for any $V \in \ttBan_R$ there is a projective $P$ and a strict epimorphism $P \to V$. To see this, denote the coproduct in $\ttBanleq_R$ by $\coprod^{\leq 1}$. Then 
\[\sideset{}{^{\leq 1}}\coprod_{ v \in V \setminus \{0\}} R_{\lVert v \rVert} \longrightarrow V
\]
is a strict epimorphism, and $\coprod_{v \in V-\{0\}}^{\leq 1} R_{\lVert v \rVert }$ is projective. We are using the notation $R_s$ for the module $R$ equipped with the norm $|r|_s = s|r|$.

The 1-categorical ind-completion $\Ind(\ttBan_{R})$ is a good setting for merging functional analysis and homological algebra, since it is cocomplete, quasi-abelian, and closed symmetric monoidal, where the symmetric monoidal structure is induced by the projective tensor product. The theory of derived categories of quasi-abelian categories from \cite{QACS} is used to produce an $\infty$-categorical version of $\Ind(\ttBan_R)$, which will be the algebraic context of derived analytic geometry over $R$. We review the main steps.

One can use formal colimits of projectives in $\ttBan_R$ as a strictly generating set of projective objects in $\Ind(\ttBan_{R})$ as follows. The functor $\ttBan_R \to \Ind(\ttBan_R)$ preserves projectives and has as essential image the full subcategory of compact objects of $\Ind(\ttBan_R)$. Let then $\ttP$ be the essential image of the functor $\ttBan_R \to \Ind(\ttBan_R)$ applied to the compact projectives of $\ttBan_R$. Then for any monomorphism $S \to E$ in $\Ind(\ttBan_R)$ which is not an isomorphism, there is some $P \to E$ with $P \in \ttP$ which does not factor through $E$.

The category of simplicial objects in $\Ind(\ttBan_{R})$ has a nice (projective) monoidal model category structure, first discussed in \cite{BassatNonArchimedean}. In particular, the weak equivalences are morphisms $X_{\bullet}\to Y_{\bullet}$ such that for all $P$ projective in $\ttBan_R$, $\Hom(P, X_{\bullet} )\to \Hom(P, Y_{\bullet})$ is a weak  equivalence of simplicial sets. Now $\BBBan_{R_{\geq 0}}$ is defined as the $\infty$-category we get from the localization at these morphisms.

Notice that \cite[Prop.\ 2.1.14]{QACS} tells us that the left heart of $\Ind(\ttBan_{R})$ (an abelian category) is equivalent to the category of additive functors from the opposite category of $\ttP$ to the category of abelian groups. Here, the \emph{left heart} is the category of monomorphisms $V \to W$ in $\Ind(\ttBan_R)$, localized at the biCartesian squares
\begin{center}
	\begin{tikzcd}
		V \arrow[r, hook] \arrow[d] & W \arrow[d] \\
		V' \arrow[r, hook] & W'.
	\end{tikzcd}
\end{center}
This category is in turn equivalent to the category of finite product preserving functors from the opposite category of $\ttP$ to the category of sets. It can be shown that $\PPP_\Sigma(\ttP)$ is then equivalent to $\BBBan_{R_{\geq 0}}$.

\begin{Def}
 The category $\BBBan_R$ is defined as the stabilization of $\BBBan_{R_{\geq 0}}$. There is a t-structure on $\BBBan_R$ for which $\BBBan_{R_{\geq 0}}$ are the connective objects, and the symmetric monoidal structure $\wotimes_{R}$ extends to $\BBBan_R$. Then $\BBBan_R$ is an algebraic context, called the context of \emph{derived analytic geometry over $R$}.
 
 If $R$ is non-Archimedean, then the same construction but with $\wotimes_{R}^{\na}$ instead of $\wotimes_{R}$ gives us the algebraic context $\BBBan_R^{\na}$ of \emph{derived non-Archimedean analytic geometry over $R$.}	
\end{Def}

 Observe that $\BBBan_R^\heartsuit$ is the left heart of $\Ind(\ttBan_R)$ by Proposition \ref{Prop:DACsum}. Now $\ttBan_R$ is not abelian in general, essentially because the image of $V \to W$ is the closure of the set-theoretic image. One can thus not hope for an algebraic context which has $\ttBan_R$ at its heart. Since the left heart of $\Ind(\ttBan_R)$ is the abelian envelope of $\ttBan_R$, this construction is in some sense the most economical solution to this problem \cite[\S 1.2.4]{QACS}. 

\begin{Exm}
	The ring of integers endowed with trivial norm is a non-Archimedean Banach ring, written $\Z_{\triv}$, and every abelian group is a Banach $\Z_{\triv}$-module (endowed with trivial norm). For two Banach $\Z_{\triv}$-modules $V,W$ with trivial norm, the non-Archimedean projective semi-norm is again trivial, hence the completed tensor product $V \wotimes^{\na}_{\Z_{\triv}} W$ is just the algebraic tensor product. 
	
	One might be tempted to think that the canonical morphism $\Mod_{\Z} \to \BBBan^{\na}_{\Z_{\triv}}$ of algebraic contexts is an equivalence. This is however not the case. Indeed, for any prime $p$, the $p$-adic numbers $\Q_p$ (with $p$-adic norm, written $\lVert \cdot \rVert_p$ for consistency) is a Banach $\Z_{\triv}$-module, since for any nonzero $x \in \Q_p$ and $n \in \Z_{\triv}$ it holds
	\[ \frac{\lVert n x \rVert_p}{\lVert x \rVert_p} = \lVert n \rVert_p = p^{-v_p(n)} \leq 1 \] 
	so we can take $C=1$ in the definition of a Banach ${\Z_{\triv}}$-module. Now a quasi-abelian category is always a full subcategory of its left heart \cite[Prop.\ 1.2.27]{QACS}, hence $\Q_p \in \BBBan_{\Z_{\triv}}^{\na,\heartsuit}$, meanwhile any object in $\BBBan_{\Z_{\triv}}^{\na,\heartsuit}$ which is in the image of $\Mod_{\Z} \to \BBBan^{\na}_{\Z_{\triv}}$ has trivial norm.
\end{Exm}

\begin{Rem}
	By Example \ref{Ex:InducedContext}, any subcanonical topology $J_{\geq 0}$ on $\BBBan_{R_{\geq 0}}$ for which $\Mod_{(-)} \colon \Aff^{\op}_{\BBBan_{R_{\geq 0}}} \to \Cat$ satisfies descent, induces a geometric context. This part of the theory of analytic geometry is currently still under development, and will appear in \cite{FDA}.
	
	The primary example of such a topology---which works for any Banach ring $R$, and which has been investigated in \cite{BassatAnalytification}---will be the \emph{homotopy monomorphism topology}. Roughly speaking, a covering family for a connective affine $T \in  \DAlg_{{\BBBan_R}}^\op$ is here a finite family of monomorphisms (in the $\infty$-categorical sense) $\{T_\alpha \to T\}_\alpha$, where each $T_\alpha$ is affine and connective, such that $\bigsqcup_\alpha T_\alpha \to T$ is an effective epimorphism. Other topologies discussed in \cite{FDA} include the flat topology, the \'{e}tale topology, and the Zariski topology. \end{Rem}

\appendix
\renewcommand{\thesection}{\hspace{-1.1mm}}
 \section{Set-theoretic background}
We will briefly indicate how we deal with size-issues. Fix Grothendieck universes $\mathfrak{U}_0 \subset \mathfrak{U}_1 \subset \mathfrak{U}_2 \subset \cdots$ such that $\mathfrak{U}_i \in \mathfrak{U}_{i+1}$ for all $i$. For each $i$, let $\Space_i$ be the category of $\mathfrak{U}_i$-small spaces, and $\Cat_i$ the category of $\mathfrak{U}_i$-small categories. This gives us inclusions $\Cat_i \subset \Cat_{i+1}$ and  $\Space_i \subset \Space_{i+1}$, and moreover $\Space_i,\Cat_i \in \Cat_{i+1}$.

An \textit{$i$-small (co)limit} is a (co)limit indexed by some $K \in \Cat_i$.

\begin{Prop}
	\label{Prop:UniversesCont}
	For $i \in \N$, the inclusion $j\colon  \Space_i \to \Space_{i+1}$ preserves $i$-small limits and colimits.
\end{Prop}

\begin{proof}
	Write $\Set_i$ for the category of $i$-small sets. Then $\Set_i \subset \Set_{i+1}$, and this inclusion preserves $i$-small limits and colimits in the 1-categorical sense.
	
	Write now $\sSet_i$ for the 1-category $\Set_i^{\bDelta^\op}$, and let $F\colon \sSet_i \to \sSet_{i+1}$ be the inclusion induced by $\Set_i \subset \Set_{i+1}$. We will show that $F$ preserves and reflects the model structure of $\sSet_i$, in the sense that $g\colon K \to L$ in $\sSet_i$ is a (trivial) fibration, resp.\ a (trivial) cofibration, resp.\ a weak equivalence in $\sSet_i$ if and only if $F(g)$ is. 
	
	Recall that Grothendieck universes are transitive by definition, which means that for $X \in \mathfrak{U}_i$ and $Y \in X$, it holds $Y \in \mathfrak{U}_i$. By the characterization of Kan fibrations as those maps that have the right lifting property with respect to all horn inclusions $\Lambda^k[n] \to \Delta[n]$, it then follows that $g$ is fibration if and only if $F(g)$ is. A similar argument shows the claim on trivial fibrations. 
	
	Since cofibrations are those maps with the left lifting property against all trivial fibrations, it follows that $F(g)$ being a cofibration implies that $g$ is a cofibration as well. Conversely, if $g$ is a cofibration, then it can be obtained as a retract of a relative $I$-cell complex, where $I$ is the set of horn inclusions. Since both $I$-cell complexes and retracts are preserved by $F$, also $F(g)$ is a cofibration. Likewise for the claim on trivial cofibrations.
	
	Now we show that $F$ also preserves weak equivalences. Let $g\colon K \to L$ be given. Since the standard bifibrant replacements of $K,L$ are preserved by $F$ by what we just saw, we may assume without loss of generality that $K,L$ are both Kan complexes. Hence, the  homotopy groups of $K,L$ can be computed purely combinatorially, i.e., $\pi_n(K,x)$ is the set of equivalences classes of pointed maps $(S^n,*) \to (K,x)$ modulo simplicial homotopy (relative to the base point), where $S^n$ is the simplicial $n$-sphere. It follows that $F$ preserves the homotopy groups, and hence $g$ is a weak equivalence if and only if $F(g)$ is.
	
	Using that $F$ preserves the model structure, together with the formula for homotopy limits and homotopy colimits in simplicial model categories as given in \cite[\S 18.1]{HirschhornModel}, one shows that $F$ preserves homotopy limits and homotopy colimits. Moreover, since $F$ is homotopical, it is both its own left and right derived functor. In other words, localizing $F$ at the weak equivalences produces the functor $j$. Since $F$ preserves homotopy limits and homotopy colimits, $j$ thus preserves limits and colimits in the $\infty$-categorical sense.
\end{proof}

	\bibliographystyle{dary}
	\bibliography{refs}

\providecommand{\MR}{\relax\ifhmode\unskip\space\fi MR }
\providecommand{\MRhref}[2]{%
  \href{http://www.ams.org/mathscinet-getitem?mr=#1}{#2}
}
\providecommand{\href}[2]{#2}
\begin{thebibliography}{{Rak}20}

\bibitem[AHPS23]{GMS}
Eric Ahlqvist, Jeroen Hekking, Michele Pernice, and Michail Savvas, \emph{Good moduli spaces in derived algebraic geometry}, 2023, \href{http://arXiv.org/abs/2309.16574}{\mbox{arXiv:2309.16574}}.

\bibitem[Ann21a]{AnnalaBivariant}
Toni Annala, \emph{Bivariant derived algebraic cobordism}, J. Algebraic Geom. \textbf{30} (2021), no.~2, 205--252.

\bibitem[Ann21b]{AnnalaPrecobordism}
Toni Annala, \emph{Precobordism and cobordism}, Algebra Number Theory \textbf{15} (2021), no.~10, 2571--2646.

\bibitem[Ant20]{JorgeSpreading}
Jorge António, \emph{{Spreading out the Hodge filtration in non-archimedean geometry}}, 2020, \href{http://arXiv.org/abs/2005.00774}{\mbox{arXiv:2005.00774}}.

\bibitem[BB16]{BB}
Federico Bambozzi and Oren {Ben-Bassat}, \emph{Dagger geometry as {B}anach algebraic geometry}, J. Number Theory \textbf{162} (2016), 391--462.

\bibitem[BBK18]{BBK}
Federico Bambozzi, Oren {Ben-Bassat}, and Kobi Kremnizer, \emph{Stein domains in {B}anach algebraic geometry}, J. Funct. Anal. \textbf{274} (2018), no.~7, 1865--1927.

\bibitem[BBK19]{BBK2}
Federico Bambozzi, Oren {Ben-Bassat}, and Kobi Kremnizer, \emph{Analytic geometry over {$\Bbb{F}_1$} and the {F}argues-{F}ontaine curve}, Adv. Math. \textbf{356} (2019), 106815, 73.

\bibitem[BHK]{LocCont}
Oren {Ben-Bassat}, Jeroen Hekking, and Jack Kelly, \emph{Derived analytic geometries via localized contexts}, Forthcoming.

\bibitem[BK17]{BassatNonArchimedean}
Oren {Ben-Bassat} and Kobi Kremnizer, \emph{{Non-Archimedean analytic geometry as relative algebraic geometry}}, Annales de la Facult{\'{e}} des sciences de Toulouse : Math{\'{e}}matiques \textbf{26} (2017), no.~1, 49--126.

\bibitem[BK23]{BeKr}
Oren {Ben-Bassat} and Kobi Kremnizer, \emph{Fréchet modules and descent}, Theory and Applications of Categories \textbf{39} (2023), no.~9, 207--266.

\bibitem[BKK24]{FDA}
Oren {Ben-Bassat}, Jack Kelly, and Kobi Kremnizer, \emph{A perspective on the foundations of derived analytic geometry}, 2024, \href{http://arXiv.org/abs/2405.07936}{\mbox{arXiv:2405.07936}}.

\bibitem[Bly05]{BlythLattice}
T.S. Blyth, \emph{Lattices and ordered algebraic structures}, Universitext, Springer-Verlag London, Ltd., London, 2005.

\bibitem[BM22]{BassatAnalytification}
Oren {Ben-Bassat} and Devarshi Mukherjee, \emph{Analytification, localization and homotopy epimorphisms}, Bulletin des Sciences Math{\'{e}}matiques \textbf{176} (2022), 103129.

\bibitem[BZFN10]{BenIntegral}
David Ben-Zvi, John Francis, and David Nadler, \emph{Integral transforms and {D}rinfeld centers in derived algebraic geometry}, J. Amer. Math. Soc. \textbf{23} (2010), no.~4, 909--966.

\bibitem[DG13]{DrinfeldFiniteness}
Vladimir Drinfeld and Dennis Gaitsgory, \emph{On some finiteness questions for algebraic stacks}, Geom. Funct. Anal. \textbf{23} (2013), no.~1, 149--294.

\bibitem[ER21]{EdidinCanon}
Dan Edidin and David Rydh, \emph{Canonical reduction of stabilizers for {A}rtin stacks with good moduli spaces}, Duke Math. J. \textbf{170} (2021), no.~5, 827--880.

\bibitem[Ful98]{FultonIntersection}
William Fulton, \emph{Intersection theory}, Springer, 1998.

\bibitem[GH]{Ideals}
Zachary Gardner and Jeroen Hekking, \emph{A note on ideals in derived geometry}, Forthcoming.

\bibitem[GR17a]{GaitsgoryStudy}
Dennis Gaitsgory and Nick Rozenblyum, \emph{A study in derived algebraic geometry {I}: {C}orrespondences and duality}, vol.~1, American Mathematical Soc., 2017.

\bibitem[GR17b]{GaitsgoryStudyII}
Dennis Gaitsgory and Nick Rozenblyum, \emph{A study in derived algebraic geometry {II}: {D}eformations, {L}ie theory and formal geometry}, Mathematical Surveys and Monographs, vol. 221, American Mathematical Society, Providence, RI, 2017.

\bibitem[HA]{LurieHA}
Jacob Lurie, \emph{{Higher Algebra}}, Preprint, available at math.ias.edu/\~{}lurie (2016).

\bibitem[HAGI]{ToenHAGI}
Bertrand To{\"e}n and Gabriele Vezzosi, \emph{{Homotopical algebraic geometry I: Topos theory}}, Advances in mathematics \textbf{193} (2005), no.~2.

\bibitem[HAGII]{ToenHAGII}
Bertrand To{\"e}n and Gabriele Vezzosi, \emph{Homotopical algebraic geometry {II}: {G}eometric stacks and applications}, vol.~2, American Mathematical Soc., 2008.

\bibitem[Hek]{HekHKR}
Jeroen Hekking, \emph{A global {HKR} theorem in geometric contexts}, Forthcoming.

\bibitem[Hek21]{HekkingGraded}
Jeroen Hekking, \emph{Graded algebras, projective spectra and blow-ups in derived algebraic geometry}, 2021, \href{http://arXiv.org/abs/2106.01270}{\mbox{arXiv:2106.01270}}.

\bibitem[Hir03]{HirschhornModel}
Philip~S. Hirschhorn, \emph{Model categories and their localizations}, Mathematical Surveys and Monographs, vol.~99, American Mathematical Society, Providence, RI, 2003.

\bibitem[HKR]{Weil}
Jeroen Hekking, Adeel Khan, and David Rydh, \emph{Deformation to the normal cone and blow-ups via derived {W}eil restrictions}, Forthcoming.

\bibitem[HRS22]{DRS}
Jeroen Hekking, David Rydh, and Michail Savvas, \emph{Stabilizer reduction for derived stacks and applications to sheaf-theoretic invariants}, 2022, \href{http://arXiv.org/abs/2209.15039}{\mbox{arXiv:2209.15039}}.

\bibitem[HTT]{LurieHTT}
Jacob {Lurie}, \emph{{Higher Topos Theory}}, Princeton University Press, 2006.

\bibitem[Kha19]{KhanVirtualofStacks}
Adeel~A. Khan, \emph{{Virtual fundamental classes of derived stacks I}}, 2019, \href{http://arXiv.org/abs/1909.01332}{\mbox{arXiv:1909.01332}}.

\bibitem[Kha20]{KhanBlowCDH}
Adeel~A. Khan, \emph{Algebraic {K}-theory of quasi-smooth blow-ups and {CDH} descent}, Ann. H. Lebesgue \textbf{3} (2020), 1091--1116.

\bibitem[Kha22]{KhanKG}
Adeel~A. Khan, \emph{K-theory and {G}-theory of derived algebraic stacks}, Jpn. J. Math. \textbf{17} (2022), no.~1, 1--61.

\bibitem[KLS17]{KiemGeneralized}
Young-Hoon Kiem, Jun Li, and Michail Savvas, \emph{Generalized {D}onaldson-{T}homas invariants via {K}irwan blowups}, 2017, \href{http://arXiv.org/abs/1712.02544}{\mbox{arXiv:1712.02544}}.

\bibitem[KR18]{KhanVirtual}
Adeel~A. {Khan} and David {Rydh}, \emph{{Virtual Cartier divisors and blow-ups}}, 2018, \href{http://arXiv.org/abs/1802.05702}{\mbox{arXiv:1802.05702}}.

\bibitem[KS06]{KashiwaraCats}
Masaki Kashiwara and Pierre Schapira, \emph{Categories and sheaves}, Grundlehren der mathematischen Wissenschaften [Fundamental Principles of Mathematical Sciences], vol. 332, Springer-Verlag, Berlin, 2006.

\bibitem[KST18]{Kerz}
Moritz Kerz, Florian Strunk, and Georg Tamme, \emph{Algebraic {$K$}-theory and descent for blow-ups}, Invent. Math. \textbf{211} (2018), no.~2, 523--577.

\bibitem[Lur11]{LurieDAGIX}
Jacob Lurie, \emph{Derived algebraic geometry {IX}: {C}losed immersions}, 2011.

\bibitem[Mao21]{MaoRevisiting}
Zhouhang Mao, \emph{Revisiting derived crystalline cohomology}, 2021, \href{http://arXiv.org/abs/2107.02921}{\mbox{arXiv:2107.02921}}.

\bibitem[Ni22a]{NiHilbertSamuel}
Dorian Ni, \emph{On the arithmetic {H}ilbert-{S}amuel theorem: a proof by deformation}, 2022, \href{http://arXiv.org/abs/2207.05165}{\mbox{arXiv:2207.05165}}.

\bibitem[Ni22b]{NiDeformation}
Dorian Ni, \emph{{On the deformation to the normal cone in {A}rakelov geometry}}, 2022, \href{http://arXiv.org/abs/2206.07954}{\mbox{arXiv:2206.07954}}.

\bibitem[NSS15]{NikolausPrincipal}
Thomas Nikolaus, Urs Schreiber, and Danny Stevenson, \emph{{Principal {$\infty$}-bundles: General theory}}, Journal of Homotopy and Related Structures \textbf{10} (2015), no.~4.

\bibitem[Por15a]{PortaDerivedI}
Mauro Porta, \emph{Derived complex analytic geometry {I}: {GAGA} theorems}, 2015, \href{http://arXiv.org/abs/1506.09042}{\mbox{arXiv:1506.09042}}.

\bibitem[Por15b]{PortaDerivedII}
Mauro Porta, \emph{Derived complex analytic geometry {II}: square-zero extensions}, 2015, \href{http://arXiv.org/abs/1507.06602}{\mbox{arXiv:1507.06602}}.

\bibitem[Por17]{PortaRiemannHilbert}
Mauro Porta, \emph{The derived {R}iemann-{H}ilbert correspondence}, 2017, \href{http://arXiv.org/abs/1703.03907}{\mbox{arXiv:1703.03907}}.

\bibitem[PY18]{PortaNonArchimedean}
Mauro Porta and Tony~Yue Yu, \emph{Derived non-archimedean analytic spaces}, Selecta Math. (N.S.) \textbf{24} (2018), no.~2, 609--665.

\bibitem[PY20]{PortaRepresentability}
Mauro Porta and Tony~Yue Yu, \emph{Representability theorem in derived analytic geometry}, J. Eur. Math. Soc. (JEMS) \textbf{22} (2020), no.~12, 3867--3951.

\bibitem[{Rak}20]{RaksitHKR}
Arpon {Raksit}, \emph{Hochschild homology and the derived de {R}ham complex revisited}, 2020, \href{http://arXiv.org/abs/2007.02576}{\mbox{arXiv:2007.02576}}.

\bibitem[SAG]{LurieSpectral}
Jacob Lurie, \emph{{Spectral Algebraic Geometry}}, {Preprint, available at math.ias.edu/\~{}lurie} (2016).

\bibitem[Sav20]{SavvasGeneralizedDT}
Michail Savvas, \emph{Generalized {D}onaldson-{T}homas invariants of derived objects via {K}irwan blowups}, 2020, \href{http://arXiv.org/abs/2005.13768}{\mbox{arXiv:2005.13768}}.

\bibitem[Sch99]{QACS}
Jean-Pierre Schneiders, \emph{Quasi-abelian categories and sheaves}, M\'{e}m. Soc. Math. Fr. (N.S.) (1999), no.~76, vi+134.

\bibitem[Sch22]{SchefersMicrolocal}
Kendric Schefers, \emph{An equivalence between vanishing cycles and microlocalization}, 2022, \href{http://arXiv.org/abs/2205.12436}{\mbox{arXiv:2205.12436}}.

\bibitem[{Sta}18]{stacks-project}
The {Stacks Project Authors}, \emph{\textit{Stacks Project}}, \url{https://stacks.math.columbia.edu}, 2018.

\end{thebibliography}
\end{document}